\documentclass[11pt,reqno,oneside]{amsart}

\usepackage{amsmath,amsthm,amssymb}
\usepackage{graphicx,xcolor}

\usepackage{etoolbox}
\numberwithin{figure}{section}
\numberwithin{equation}{section}
 
\DeclareFontFamily{U}{mathb}{\hyphenchar\font45}
\DeclareFontShape{U}{mathb}{m}{n}{
	<-6> mathb5 <6-7> mathb6 <7-8> mathb7
	<8-9> mathb8 <9-10> mathb9
	<10-12> mathb10 <12-> mathb12
}{}
\DeclareSymbolFont{mathb}{U}{mathb}{m}{n}
\DeclareMathSymbol{\llcurly}{\mathrel}{mathb}{"CE}
\DeclareMathSymbol{\ggcurly}{\mathrel}{mathb}{"CF}

\definecolor{Red}{cmyk}{0,1,1,0}

\definecolor{Blue}{cmyk}{1,1,0,0}

\theoremstyle{plain}
\newtheorem{maintheorem}{Theorem}

\newtheorem{maincorollary}{Corollary}
\newtheorem{theorem}{Theorem }[section]
\newtheorem{proposition}[theorem]{Proposition}
\newtheorem{lemma}[theorem]{Lemma}
\newtheorem{corollary}[theorem]{Corollary}

\theoremstyle{definition} \theoremstyle{remark}
\newtheorem{remark}[theorem]{Remark}
\newtheorem{example}[theorem]{Example}
\newtheorem{definition}[theorem]{Definition}

\newcommand{\ud}{\underline}

\newcommand{\al} {\alpha}

\newcommand{\de} {\delta}       

\newcommand{\vep}{\varepsilon}

      \newcommand{\La}{\Lambda}

\newcommand{\om} {{\underline{\omega}}}       

\newcommand{\Z}{\mathbb{Z}}
\newcommand{\N}{\mathbb{N}}
\newcommand{\R}{\mathbb{R}}
\newcommand{\supp}{\operatorname{supp}}

\newcommand{\topp}{\operatorname{top}}

\newcommand{\ov}{\overline}

\newcommand{\htop}{h_{\topp}}

\newcommand{\cI}{{\mathcal I}}

\newcommand{\cM}{\mathcal{M}}

\newcommand{\cG}{\mathcal{G}}

\usepackage{mathrsfs}
\usepackage[colorlinks=true,linkcolor=blue, urlcolor=red, citecolor=blue, backref]{hyperref}	

 \title[Lyapunov irregular points]{Lyapunov non-typical behavior for linear cocycles \\ through  the lens of semigroup actions}
 \date{June 2021}

\begin{document}

\author{Giovane Ferreira and Paulo Varandas}
\address{Giovane Ferreira, Departamento de Matem\'atica, Universidade Federal do Maranh\~ao\\
Av. dos Portugueses, 1966, Vila Bacanga, 65065-545 S\~ao Lu\'is, MA}
\email{giovane.ferreira@ufma.br}

\address{Paulo Varandas, Departamento de Matem\'atica, Universidade Federal da Bahia\\
Av. Ademar de Barros s/n, 40170-110 Salvador, Brazil \& CMUP - Universidade do Porto, Portugal}
\email{paulo.varandas@ufba.br}


\maketitle
 \begin{abstract} 
 The celebrated Oseledets theorem \cite{O}, building over seminal works of Furstenberg and Kesten on random products of matrices and random variables taking values on  non-compact semisimple Lie groups \cite{FK,Furstenberg}, ensures that the Lyapunov exponents of $\mathrm{SL}(d,\mathbb R)$-cocycles $(d\geqslant 2)$ over the shift 
are well defined for all points in 
 a total probability set, ie, a full measure subset for all invariant probabilities. Given a locally constant $\mathrm{SL}(d,\mathbb R)$-valued cocycle we are interested 
 both in the set of points on the shift space for which some Lyapunov exponent is not well defined, and in the set of directions on the projective space 
 $\mathbf P \mathbb R^d$ along which there exists no well defined exponential growth rate of vectors for a certain product of matrices.  
We prove that if the semigroup generated by finitely many matrices in $\mathrm{SL}(d,\mathbb R)$ is not compact and is strongly projectively accessible then 
there exists a dense set of directions in $\mathbf P \mathbb R^d$ along which the Lyapunov exponent of a typical product of such matrices is not well defined. 
As a consequence, we deduce that the presence of Lyapunov non-typical behavior is prevalent among $SL(3,\mathbb R)$-valued hyperbolic cocycles. 
These results arise as a consequence of the more general description of the set of non-typical points for Ces\`aro  averages of continuous observables 
and infinite paths in finitely generated semigroup actions by homeomorphisms on a compact metric space $X$ under a mild hitting times condition.

\end{abstract}

\tableofcontents

\section{Introduction} 

\subsection{Linear cocycles} 

Important results of Furstenberg, Kesten, Oseledets, Kingman, Ledrappier, Guivarc'h, 
Raugi, Margulis and other mathematicians, initiated in the sixties built 
the study of Lyapunov characteristic exponents into a very active research field in its own right,
with a vast array of interactions with other areas of Mathematics and Physics. 
We refer the reader to the monographs by Duarte, Klein and Viana \cite{DK, VianaB} and references therein
for a some historical remarks and description of the state of the art.

\smallskip
Lyapunov exponents have been studied first in the context of random products of matrices \cite{FK, Furstenberg}. Here 
consider linear cocycles where the iteration of the matrices is determined by typical orbits of an invariant probability measure for a deterministic map. 
Let us recall these notions.
Let $(M,\mu)$ be a probability space, $f \colon M \to M$ a measure-preserving automorphism and $A \colon M \to \mathrm{SL}(d,\mathbb{R})$ a measurable matrix-valued map.
The pair $(A,f)$ is called a \emph{linear cocycle} and, for each $n\geqslant 1$ set
\begin{equation}\label{e.product}
A^{(n)}(x) := A(f^{n-1}(x)) \cdots A(f(x)) A(x)
\quad\text{and}\quad
A^{(-n)}(x)=(A^{(n)}(f^{-n}(x)))^{-1}.
\end{equation}
The dynamics of the linear cocycle can be described by the skew-product
\begin{equation}\label{eq:skewF}
\begin{array}{cccc}
F_A \colon  & M\;\times\;\mathbb R^{d} & \longrightarrow  & M\;\times\;\mathbb R^{d} \\
& (x,v) & \mapsto  & (f(x), A(x)\cdot v),
\end{array}
\end{equation}
where the joint base and fiber dynamics is given by $F^n_A(x,v)=(f^n(x),A^{(n)}(x)v)$ for all $(x,v)\in M\times\mathbb R^{d}$
and every $n\geqslant 1$.
Under the integrability assumption $\log \|A^{\pm1}\| \in L^1(\mu)$, the Oseledets' theorem~\cite{O}  ensures that 
for $\mu$-a.e. $x\in M$ there exist $1\leqslant k(x) \leqslant d$, a splitting 
$
\{x\} \times \mathbb R^d = E^1_x \oplus E^2_x\oplus \dots \oplus E_{x}^{k(x)} \;
$ 
(called the \emph{Oseledets splitting}) varying measurably with $x$ and so that $A(x)E^i_x=E^i_{f(x)}$ for every $1\leqslant i \leqslant k(x)$,
and real numbers $\lambda_1(A,f,x) > \lambda_2(A,f,x)  > \dots > \lambda_{k(x)}(A,f,x)$ (called the \emph{Lyapunov exponents}
associated to $f$, $A$ and $x$) 
such that
\begin{equation}\label{e.top_LE_i}
\lambda_i(A,f,x)= \lim_{n\to\pm \infty} \frac1n \log \|A^{(n)}(x) v_i\|,
	\quad \forall v_i \in E^i_x\setminus \{\vec0\}.
\end{equation}
If, in addition, $\mu$ is ergodic then both the number and the value of 
the Lyapunov exponents are almost everywhere constant, and the Lyapunov exponents are denoted simply by  $\lambda_i(A,f,\mu)$.

Moreover, the {largest and smallest Lyapunov exponents} of the cocycle with respect to 
an ergodic probability $\mu$ can be expressed, using Kingman's subadditive ergodic theorem, 
by

$$
\lambda_+(A,\mu) := \lim_{n \to \infty} \frac{1}{n} \log \big\| A^{(n)}(x) \big\|
	\quad\text{and}\quad
	\lambda_-(A,\mu) := \lim_{n \to \infty} \frac{1}{n} \log \big\| A^{(n)}(x)^{-1} \big\|^{-1},
$$
respectively (see also the previous work by Furstenberg and Kesten \cite{FK} on the products of random matrices).
 
The projectivization $P_A$ of the skew-product $F_A$
is the map
\begin{equation}\label{eq:skewP}
\begin{array}{cccc}
P_A \colon  & M\;\times\; \mathbf P \mathbb R^d & \longrightarrow  & M\;\times\;  \mathbf P \mathbb R^d \\
& (x,[v])\;\; & \mapsto  & \Big(f(x), \big[\frac{A(x)\cdot v}{\|A(x)\cdot v\|}\big]\Big)
\end{array}
\end{equation}
acting on the compact metric space $M\;\times\; \mathbf P \mathbb R^d$. 
We denote by $\pi: M\times \mathbf P \mathbb R^d \to M$ the projection on the first coordinate. Clearly
the fibers $\pi^{-1}(\{x\})= \mathbf P \mathbb R^d$ are compact for every $x\in M$. 
The random products of projectivized matrices can be traced to the work of Furstenberg~\cite{Furstenberg} on random
products of matrices. In other words, Furstenberg considered the case when $f$ coincides with the full shift $\sigma: \Sigma_\kappa \to \Sigma_\kappa$
on $\kappa$ symbols ($\kappa\geqslant 2$) and $\mu=\nu^{\mathbb Z}$ is a Bernoulli probability measure on $\Sigma_\kappa$. 
  
  \medskip

In view of Oseledets theorem, the set of points for which the Lyapunov exponents are well defined, often referred as the set of \emph{Lyapunov typical points}, 
form a total probability set in the ambient space. Moreover, 
if $x\in M$ is a Lyapunov typical point for the linear cocycle $A$ it is clear from the invariance of the Oseledets splitting together with the existence of the limits
in ~\eqref{e.top_LE_i}, that for every $v \in \mathbb R^d\setminus\{\vec 0\}$ the limit
$$
\lim_{n\to+ \infty} \frac1n \log \|A^{(n)}(x) v\|
$$
is well defined:  it is equal to $\lambda_{k(x)}(A,f,x)$ whenever $v\in E^{k(x)}_x$ and, otherwise, it coincides with 
$\lambda_{i-1}(A,f,x)$ where $2\leqslant i \leqslant k(x)$ is the 
largest integer so that $v\in \mathbb R^d \setminus (E^{k(x)}_x\oplus E^{k(x)-1}_x\oplus \dots \oplus E^{i}_x)$.
In the late eighties Ma\~n\'e suggested that the set of points for which the Lyapunov exponents are not well defined 
should to be prevalent in the Baire topological sense.
This idea should be tested in the following two distinct contexts, eventually with completely different techniques:
\begin{itemize}
\item[$\circ$] $d=1$: the cocycle $A: M \to \mathbb R$ takes values on the abelian group $(\mathbb R,+)$;
\item[$\circ$] $d\geqslant 2$: the cocycle $A: M \to \mathrm{SL}(d,\mathbb{R})$ takes values on the non-abelian group $(\mathrm{SL}(d,\mathbb{R}),\circ)$.
\end{itemize}

\medskip
In the first context, of real valued observables,  the Lyapunov non-typical points (also known as irregular points) 
of a smooth map $f: \mathbf S^1 \to \mathbf S^1$ are exactly those points which are non-typical for the 
Birkhoff averages of $A(x)=\log |f'(x)|$ ($x\in M$) provided that $f$ is smooth.  
Ma\~n\'e's perception could be formalized in rather strong ways. First, in the context of $C^{1+\alpha}$-uniformly expanding circle maps,
Barreira and Schmelling~\cite{Barreira} proved that although irregular points have zero measure for every invariant measure, these may have full topological entropy and even full Hausdorff dimension.
Here, the key mechanism to construct points with non-typical behavior is the notion of specification introduced by Bowen in ~\cite{Bo71}.
This turned out to be a useful tool to develop a multifractal analysis of Lyapunov exponents and to prove that irregular sets are either empty
or  Baire generic, have full Hausdorff dimension, topological entropy or mean dimension (cf. \cite{AP19,Barreira1,BV14,LV,ZC13} and references therein).     
Actually, while part of these results can be extended for topological dynamical
systems having the specification property shadowing or the gluing orbit properties (see e.g. \cite{DOT,LV} and references therein),
these properties fail for large classes of partially hyperbolic diffeomorphisms 
\cite{BTV,SVY}. 
Nevertheless, the set points with non-convergent Birkhoff averages form a Baire residual subset of the ambient space
under the mild assumption that there are dense subsets of points with distinct largest possible Birkhoff averages (see \cite{CV19} for the precise statement).
This is a very mild assumption which, in particular, seems not enough to determine whether the set of non-typical points carry large entropy. 

\medskip
In the second context, of linear cocycles taking values on non-abelian groups, the situation is subtler and fewe results are
known both on the level sets of Lyapunov exponents and the set of Lyapunov non-typical points. 
Inspired by Herman~\cite{Herman}, Furman~\cite{Fu} proved
that, even though irrational rotations are minimal and uniquely ergodic, some cocycles over irrational rotations have a 
residual subset of Lyapunov irregular points. More recently, Tian~\cite{Tian, Tian2} proved that H\"older cocycles over homeomorphisms with an exponential specification property either have the same top Lyapunov exponent or the set of Lyapunov irregular points is Baire generic and carries full topological entropy. In the aforementioned results, the use of  
the exponential specification property - which holds for uniformly expanding dynamics - is used to obtain bounded distortion properties as in the work of Kalinin ~\cite{Kalinin}, and the author asks whether the specification 
property could be enough to prove the theorem. In the complementary direction, 
the level sets of Lyapunov exponents of certain $\mathrm{SL}(2,\mathbb R)$-valued cocycles have been described in terms of 
the Legendre-Fenchel transform of a certain pressure function (cf ~\cite{DGR}).

\medskip
Here we consider cocycles satisfying a strong projective accessibility condition (see Definition~\ref{def:access}),
a condition which can be expressed more generally in terms of the semigroup generated by the projective linear maps
(cf. Definition~\ref{def:transition} and Remark~\ref{rmk:def.hitting}). 
We will not require any fiber bunching condition on the cocycle  (see e.g. \cite{Va20} for the definition), 
hence the projectivized cocycle $P_A$ need not be partially hyperbolic. Moreover, even if this was the case, 
a typical fiber-bunched cocycle $A$ generates a partially hyperbolic skew-product
$P_A$ which fails to satisfy the specification, the gluing orbit property and their nonuniform versions  (see \cite{BTV2,SVY} and Example~\ref{MS}). 
Inspired by \cite{Furstenberg}, in Theorem~\ref{thm:abstract} we prove that if the group generated by a finite number of invertible matrices 
is non-compact and is strongly projectively accessible then there exist a Baire generic set of points in the shift space for which 
Lyapunov exponents are not well defined on a dense set of directions in $\mathbf P\mathbb R^d$. 
We use this fact to conclude that, in the special context of projective actions for $SL(3,\mathbb R)$-valued hyperbolic cocycles, the presence of non-typical behavior is prevalent (see Theorem~\ref{thm:rigidity} for the precise statement).

\subsection{Semigroup actions} 

Group, semigroup and pseudo-group actions have been of increasing interest in dynamical systems. Apart from its genuine interest, 
in which some aspects of the group action are determined by the algebraic structure of the underlying group (e.g. existence of invariant measures),
such objects  appear as well 
in the description of classical dynamical systems. For instance, these are used to provide a notion of entropy for invariant foliations and
to model the central dynamics for certain partially hyperbolic diffeomorphisms  
(see e.g. \cite{DER,GLW, HN} and references therein).

\medskip
In the last three decades, there has been an intense study about pointwise ergodic theorems for actions of amenable, Lie or free groups, and we refer the reader to 
 \cite{BN,Buf,Grig,Lindenstrauss,Nevo,Templeman}, just to mention a few contributions. Again, the situation is much subtler than the case of 
 $\mathbb Z$-actions. In order to illustrate that, we note that, in the case of free groups, the spherical averages of $L^p$-observables $(p>1)$ converge almost everywhere and in $L^1$ ~\cite{Buf,NevoStein} but this is no longer true for general $L^1$-observable \cite{Tao}. In the case of measure preserving 
 Markov semigroups an extra Ces\`aro averaging is enough to guarantee the convergence of spherical averages (see \cite{BKK}).
It seem to exist almost no results concerning the set of points presenting irregular behavior for such pointwise ergodic theorems, an exception made to the case of 
amenable group actions satisfying a suitable specification property (cf. \cite{RTZ}). 
However, as the specification property turns is rare even among $\mathbb Z$ or $\mathbb R$-actions (cf. \cite{SVY,SVY2}) it is natural to look for other 
mechanisms which can be used as a tool to construct points with non-typical behavior.

\medskip
Here we focus on the case of semigroup actions acting by bi-Lipschitz homeomorphisms on a compact metric space $X$. Let $G_1=\{id,f_1, f_2, \dots, f_\kappa\}$
denote the generators of the semigroup $G$.
 For that reason, 
inspired by the context of random products of matrices, for each continuous map $\psi: X\to \mathbb R$ we consider alternatively the problem of describing 
both the set of infinite paths on the semigroup $G$ (determined by an infinite word $\om\in\{1,2,\dots, \kappa\}^{\mathbb N}$ 
and the homeomorphisms $g_{\om_n} \circ \dots \circ g_{\om_2}\circ g_{\om_1}\in G$)
and points $x\in X$ for which the Ces\`aro averages
\begin{equation}\label{eq:groupaverage}
\frac1n \sum_{j=0}^{n-1} \psi(\,g_{\om_j} \circ \dots \circ g_{\om_2}\circ g_{\om_1}(x)\,)
\end{equation}
diverge as $n\to\infty$. 
Under a mild assumption on the semigroup action, concerning a quantitative estimate on the hitting times to balls  
of fixed size - which we refer to as frequent hitting times (cf. Definition~\ref{def:transition}) - 
we have the following goals: given a continuous observable $\psi: X\to \mathbb R$
\begin{itemize}
\item[$\circ$] describe the set of points $x\in X$ for which there exist $\om\in\{1,2,\dots, \kappa\}^{\mathbb N}$ which determines a concatenation 
	of elements in the semigroup along which \eqref{eq:groupaverage} diverges as $n\to\infty$;
\item[$\circ$] construct a 'relevant' set of infinite paths on the semigroup $G$ (parameterized by infinite words $\om\in\{1,2,\dots, \kappa\}^{\mathbb N}$) 
	so that the set $I_\om(\psi)\subset X$ of points   
	so that the Ces\`aro averages \eqref{eq:groupaverage} diverge is a Baire generic subset of $X$ and carries large entropy.
\end{itemize}
We refer the reader to Theorems~\ref{thm:IFS} and \ref{thm:Bb} 
for positive results on this direction,
using the notions of entropy for the semigroup introduced by Ghys, Langevin and Walczak~\cite{GLW} and by Bufetov \cite{Bufetov}, and the notion
of entropy for non-autonomous dynamical systems introduced by Kolyada and Snoha \cite{KS}.
As the notion of entropy of a free semigroup is related to the the notion of entropy for a certain skew-product \cite{Bufetov}, part of the
argument in Theorem~\ref{thm:Bb} is a consequence of a result for non-typical points on skew-products (Theorem~\ref{thm:B}). 
It is worth mentioning that, in opposition to the action by projective linear maps on projective spaces, the generators in $G_1$
may have positive topological entropy, hence one can specifically address the set of paths which are not only non-typical but which
have large entropy (see Theorem~\ref{thm:Bb-2}).

\section{Statement of the main results}\label{sec:statements}

This section is devoted to the statement of the main results. 
The main results concerning 
matrix valued cocycles are stated in Subsection~\ref{sec:Lyap-irregular}, while the ones concerning real valued cocycles associated to semigroup actions are stated in Subsection~\ref{sec:IFS}.

In Subsection~\ref{sec:overview} we discuss the methods in the proofs and previous works. 
In Subsection~\ref{se:organization} we describe the organization of
the arguments.

\subsection{Lyapunov non-typical behavior}\label{sec:Lyap-irregular}

Let $\sigma~:~\Sigma_\kappa \to\Sigma_\kappa$ be a the full shift on $\kappa\geqslant 1$ symbols, where 
$\Sigma_\kappa\subset \{1,2,\dots, \kappa\}^{\mathbb Z}$ and consider the distance
\begin{equation*}\label{def:metric-sigma}
d_{\Sigma_\kappa}(\om, \om')=e^{-N(\om, \om')},
	\quad\text{where} \; N(\om, \om')=\inf\{k\geqslant 1\colon \omega_k\neq \omega_k'\}. 
\end{equation*}

We denote the 1-cylinder sets by $[j]:=\{(\omega_i)_{i\in\mathbb Z}\in \Sigma_\kappa \colon \omega_0=j\}$, 
for $1\leqslant j \leqslant \kappa$ and, for any integers $m,n\in \mathbb Z$ with $m\leqslant n$ we will use the following notation 
$$
\om_{[m,n]}:=
\big\{
\tilde \om \in\Sigma_\kappa\colon \omega_j=\tilde \omega_j, \; \forall m\leqslant j \leqslant n
\big\}
$$
for the cylinder set in the shift determined by the entries in the interval $[m,n]$.

\smallskip 

Denote by $C^0(\Sigma_\kappa,\mathscr G)$ the Banach space of continuous linear cocycles 
$A: \Sigma_\kappa\to \mathscr G$ taking values on a closed subgroup $\mathscr G\subset \mathrm{SL}(d,\mathbb R)$, 
endowed with the $\|A\|_\infty$ norm. 
Let $C_{\text{loc}}^0(\Sigma_\kappa,\mathscr G)\subset C^0(\Sigma_\kappa,\mathscr G)$ denote the closed set of locally constant cocycles
on cylinders of size 1, which is isomorphic to ${\mathscr G}^\kappa$.
As mentioned before, the joint base and fiber dynamics is given by $F^n_A(x,v)=(\sigma^n(x),A^{(n)}(x)v)$, where the matrix $A^{(n)}(x)\in \mathscr G$ is defined by ~\eqref{e.product}.
We say that a linear cocycle $A$ is locally constant if is is constant on cylinder sets of $\Sigma_\kappa$. Up to a slight change of
the alphabet these can be assumed to be step skew-products, meaning that 
the cocycle $A$ is constant on the

We start by stating the following consequence of the pioneering work of Furstenberg on the products of random variables 
taking values on semisimple Lie groups, which serves as a motivation for Theorem~\ref{thm:abstract} below.

\begin{theorem}(\cite[Theorem~8.6]{Furstenberg}) \label{thm:Furstenberg} 
Let $\mathscr G\subset \mathrm{SL}(d,\mathbb R)$ be a closed and non-compact subgroup, 
$A\in C_{\text{loc}}^0(\Sigma_\kappa,\mathscr G)$ and $\nu=\nu_0^{\mathbb Z}$ be a Bernoulli probability measure on $\Sigma_\kappa$. 
Assume that:
\begin{enumerate}
\item the semigroup generated by the matrices in the support of $\nu_0$ is not contained in a compact subgroup of $\mathscr G$,
\item the cocycle is strongly irreducible on $\supp\nu$ (ie, there is no finite union of proper subspaces of $\mathbb R^d$ invariant by $A(\om)$
	for every $\om\in\supp\nu$).
\end{enumerate}
Then $\lambda_+(A,\nu)>0$.
\end{theorem}

The previous theorem provides geometric conditions under which the top Lyapunov exponent, almost everywhere defined, is positive.
 Other refinements, still in the i.i.d.\ setting, have been obtained by Guivarc'h and Raugi \cite{GR}, Gol'dsheid and Margulis \cite{GM}, among others.
\smallskip

Our first goal here is to understand the complementary problem of determining which directions, in the projective space $\mathbf P \mathbb R^d$,
one can observe Lyapunov non-typical behavior (meaning that there exists some point in the shift such that the Lyapunov exponent is not well defined along that direction, under forward iteration). While the set of Lyapunov non-typical directions may be meager in the projective space (see e.g. Example~\ref{ex:meager}), 
we will introduce a geometric property on a finite family of matrices under which the generated cocycle has plenty Lyapunov non-typical directions.
\begin{definition}\label{def:access}
\emph{
Let $\mathscr G\subset \mathrm{SL}(d,\mathbb R)$ be a closed and non-compact subgroup. A cocycle 
$A\in C_{\text{loc}}^0(\Sigma_\kappa,\mathscr G)$ is \emph{strongly projectively accessible} if for any $\varepsilon>0$ there exists an integer $K(\varepsilon)\geqslant 1$ so that for any ball $B_1 \subset \mathbf P \mathbb R^d$ of radius $\vep$ and a ball $B_2 \subset \mathbf P \mathbb R^d$ of radius 
$0<\delta\leqslant \vep$ there exists $0\leqslant p \leqslant K(\vep)$ and ${\underline{\omega}} \in \Sigma_\kappa$ so that 
the image of $B_1$ by the projectivization of $A^{(p)}(\om)$ contains a ball $B_2'\subset B_2$ of radius $r\geqslant \min\{\vep/4,\delta/2\}$.
}
\end{definition}

 The previous condition is natural among cocycles which have elements in the orthogonal Lie group $SO(d,\mathbb R)$ 
(also known as rotation group) but can also be used to depict the projective action inside two-dimensional hyperbolic subbundles  
(see Section~\ref{sec:linear-cocycles} for more details). 
 Our first main result says that, under the previous condition, the set of Lyapunov non-typical directions is Baire generic and 
 has large entropy.

\begin{maintheorem}\label{thm:abstract}
Assume that $\mathscr G\subset \mathrm{SL}(d,\mathbb R)$ is a closed subgroup and that
$A\in C^0_{loc}(\Sigma_\kappa, \mathscr G)$. is a locally constant cocycle, and write $A_i=A\mid_{[i]}$ for each $1\leqslant i\leqslant \kappa$. 
Assume that
\begin{enumerate}
\item the subgroup generated by the matrices $\{A_1,A_2, \dots, A_\kappa\}$ is non-compact, 
\item $A$ is strongly projectively accessible.
\end{enumerate}
Then for each $v\in \mathbf P\mathbb R^d$ there exists a Baire residual subset $\mathscr R_v\subset \Sigma_\kappa$ of entropy at least $h_*(\varphi_A)$ (defined by ~\eqref{eq:T})  so that, for every 
	$\om \in \mathscr R_v$, 
		\begin{equation}\label{eq:divergence-cocycle}\tag{$\star$}
		\liminf_{n\to\infty} \frac1n \log \|A^n(\om) v\| < \limsup_{n\to\infty} \frac1n \log \|A^n(\om) v\|
		\end{equation} 
In particular, there exists a Baire residual subset $\mathscr R\subset \Sigma_\kappa$  
and a dense subset $\mathscr D\subset \mathbf P\mathbb R^d$ so that  
$(\star)$
holds for every 
		$\om\in \mathscr R$ and every $v\in \mathscr D$.
\end{maintheorem}
Note that the assumptions of Theorem~\ref{thm:abstract} are related to the ones of Theorem~\ref{thm:Furstenberg}, 
the difference being that the strong irreducibility condition is replaced by the condition  
that the cocycle is strongly projectively accessible which, nonetheless, have the same flavour of indecomposability of the the action on the projective space
$\mathbf P \mathbb R^d$.  However, their conclusions are not, as there exist cocycles so that all invariant measures have a positive Lyapunov
exponent 
while the conclusion of Theorem~\ref{thm:abstract} fails (cf. Example~\ref{ex:different-types}).
The proof of Theorem~\ref{thm:abstract} is given in Section~\ref{sec:linear-cocycles}.

\begin{remark}
Projective linear maps on the projective space have a simple  Morse-Smale type of structure. In particular, these maps (and their non-autonomous 
compositions) have zero topological entropy (see Subsection~\ref{subsec:projectivelinear} for more details).
\end{remark}

\begin{remark}
First results on irregular points for the top Lyapunov exponent appeared first in
~\cite{Fu,Herman, Tian, Tian2}. 
The results in ~\cite{Tian,Tian2}
can be summarized as follows:
if $A: \Sigma_\kappa \to \mathscr G$ is a H\"older continuous cocycle then either $\lambda_+(A,\nu_1)=\lambda_+(A,\nu_2)$ for every $\sigma$-invariant probability measures $\nu_1,\nu_2$, or there exists a Baire residual and large topological entropy subset of the shift $\Sigma_\kappa$ formed by 
points for which the top Lyapunov exponent is not well defined. 
\end{remark}

The projective accessibility can be also useful to describe the behavior of a linear cocycle restricted to some invariant subbundle. We will be particularly 
interested in cocycles displaying dominated splittings, whose notion we now recall. Given a cocycle $A\in C^0(\Sigma_\kappa,\mathscr G)$, we say that
a splitting $\Sigma_\kappa\times \mathbb R^d=E^1\oplus E^2$ is $A$-invariant if $A(\om) E^i_\om=E^i_{\sigma(\om)}$ for every $\om\in \Sigma_\kappa$.
An $A$-invariant splitting $\Sigma_\kappa\times \mathbb R^d=E^1\oplus E^2$ is \emph{dominated} if
there exists $C>0$ and $\lambda \in (0,1)$ so that 
$$
\| A^{(n)}(\om) \mid_{E_\om^2}\| . \| (A^{(n)}(\om)\mid_{E^1_\om})^{-1}\| \leqslant C \lambda^n
\text{ for every $n\geqslant 1$ and $\om\in \Sigma_\kappa$.}
$$
It is well known  
that the existence of a dominated splitting  is $C^0$-open on the space of continuous linear cocycles over a compact space and
that the subbundles vary continuously with the point $\om$. In particular, as the shift $\sigma$ is transitive, the dimensions of the subbundles is constant. 
A simple example of a dominated splitting occurs whenever
$
\| A^{(n)}(\om)\mid_{E^2_\om} \| \leqslant C \lambda^n 
$ and 
$
	\| (A^{(n)}(\om)\mid_{E^1_\om})^{-1} \| \leqslant C \lambda^n 
$
for every $\om\in \Sigma_\kappa$ and $n\geqslant 1$, in which case we say that the cocycle is \emph{uniformly hyperbolic}, and we refer to $\Sigma_\kappa\times \mathbb R^d=E^1\oplus E^2$ as its hyperbolic splitting.
We refer the reader to \cite{Yoc} for other characterizations of uniform hyperbolicity in the context of $SL(2,\mathbb R)$-cocycles.
We say that $A$  has a \emph{one-dimensional finest dominated splitting} if it admits a continuous $A$-invariant splitting 
$\Sigma_\kappa\times \mathbb R^d = E^1_\om \oplus E^2_\om\oplus \dots \oplus E_{\om}^{d}$  in one-dimensional subspaces at every $\om\in \Sigma_\kappa$ such that the splitting $E^i\oplus E^{i+1}$ is dominated for each $1\leqslant i \leqslant d -1$.
The next result, and the discussion following it, ensures that the irregular behavior prevails.

\begin{maintheorem}\label{thm:rigidity}
Assume that $A\in C^0_{loc} (\Sigma_\kappa, SL(3,\mathbb R))$ is uniformly hyperbolic and its hyperbolic spliting 
$\Sigma_\kappa\times \mathbb R^3=E^s\oplus E^u$ 
satisfies $\dim E^s=2$. There exists a $C^0$-open neighborhood 
$\mathscr U \subset C^0_{\text{loc}}(\Sigma_\kappa,SL(3,\mathbb R))$ and a $C^0$-Baire residual and full Haar measure subset $\mathscr R\subset \mathscr U$ so that for every $B\in \mathscr R$ either:
\begin{enumerate}
\item $B$ admits a finest one-dimensional dominated splitting $\mathbb R^3=E_{B,\om}^u\oplus E_{B,\om}^{s,1} \oplus 	
	E_{B,\om}^{s,2}$, or 
\item there exists a Baire residual subset $\mathscr R\subset \Sigma_\kappa$ and for each $\om\in \mathscr R$ there exists a
	dense subset $\mathscr D_\om\subset \mathbf P E^s_\om$
	so that 
	$$
		\liminf_{n\to\infty} \frac1n \log \|B^n(\om) v\| < \limsup_{n\to\infty} \frac1n \log \|B^n(\om) v\|
			\quad\text{for every $v\in \mathscr D_\om$.}
	$$
\end{enumerate} 
\end{maintheorem}

\begin{remark}
The property of having a one-dimensional finest dominated splitting is a $C^0$-open condition in $C^0_{\text{loc}}(\Sigma_\kappa,SL(3,\mathbb R))$. 
Hence, in the context of Theorem~\ref{thm:rigidity}, if the $C^0$-open set $\mathscr U_1\subset \mathscr U$ of cocycles having a one-dimensional finest 
dominated splitting is non-empty and $B\in \mathscr U_1$ then it follows from 
 the description of irregular sets for Birkhoff averages for the shift and the real valued observables $\varphi_i$ 
 that either these functions are cohomologous to a constant or the space of points in $\Sigma_\kappa$ displaying irregular behavior is Baire generic and carries full topological entropy \cite{Barreira1,CKS}. 
Actually we prove that, in the context of Theorem~\ref{thm:rigidity}, a typical cocycle whose Lyapunov
irregular set is empty is cohomologous to a constant matrix cocycle (see Proposition~\ref{cor:rigidity} for the precise statement).
As the class of cocycles in $\mathscr U_1$ that are cohomologous to a constant cocycle form a closed set with empty interior,
 one may say that the irregular behavior prevails. 
\end{remark}

\begin{remark}
The dichotomy in Theorem~\ref{thm:rigidity} is related to the proof of \cite[Theorem~2.5 and Corollary~2.6]{BRV} which,
in this context of locally constant linear cocycles acting on $\mathbb R^3$, 
ensures that if the Lyapunov exponents at periodic points are all distinct
by small $C^0$-perturbations of the cocycle then there exists a finest one-dimensional dominated splitting, 
which coincides with the Oseledets splitting in a total probability set. 
\end{remark}

\subsection{Non-typical points for semigroup actions}\label{sec:IFS}

Given a semigroup action $\mathbb S$ generated by a collection $G_1=\{id,f_1, f_2, \dots, f_\kappa\}$  of continuous maps
on a compact metric space $X$ and a continuous observable $\psi: X\to\mathbb R$, we denote by
\begin{equation}\label{eq:def-I-IFS}
I_{\mathbb S}(\psi)
	=\Big\{ x\in X \colon 
	\exists\, {\underline{\omega}}\in \Sigma_\kappa   \;\text{s.t.}\;  \lim_{n\to\infty} \frac1n \sum_{j=0}^{n-1} \psi(f_{\underline{\omega}}^j(x)) \; \text{does not exist}
	\Big\}
\end{equation}
the set on \emph{Birkhoff non-typical points} for the semigroup action $\mathbb S$
(with respect to $\psi$)
where, for each ${\underline{\omega}} \in \Sigma_\kappa:=\{1,2,\dots, \kappa\}^\mathbb N$ and $n\geqslant 1$, we write 
$
f^n_{\underline{\omega}}:= f_{{\underline{\omega}}_{n-1}}\circ \dots \circ f_{{\underline{\omega}}_1}\circ f_{{\underline{\omega}}_0}
$
for the composition of maps in $G_1$.
It will be a standing assumption that the generators semigroup action are bi-Lipschitz homeomorphisms on $X$ and we set
\begin{equation*}\label{standing-assumption-IFS}
L:=\max\big\{1\,, \;\max_{1\leqslant i \leqslant \kappa} \text{Lip}(f_i^{\pm 1})\big\}.
\end{equation*}
The dynamics of the semigroup action generated by $G_1$, endowed with a random walk, is intimately 
related to the dynamics of the skew-product
\begin{equation}\label{def:skew-eq}
\begin{array}{cccc}
F \colon & \{1,2,  \dots, \kappa\}^{\mathbb Z} \times X & \longrightarrow  & \{1,2,  \dots, \kappa\}^{\mathbb Z} \times X \\
& (\underline{\omega},x) & \mapsto  & (\sigma(\underline{\omega}), f_{\omega_0}(x)),
\end{array}
\end{equation}
through the observation that $F^n(\om,x)=(\sigma^n(\om),g_\om^n(x))$ for every $(\underline{\omega},x)\in 
 \{1,2,  \dots, \kappa\}^{\mathbb Z} \times X$ and every $n\geqslant 1$.
Moreover, if $\varphi= \{1 , 2, \dots, \kappa\}^{\mathbb N} \times X \to \mathbb R$ is continuous
and one defines the set of Birkhoff irregular points for the skew-product $F$ as
\begin{equation}\label{eq:def-I-skew}
I_F(\varphi):=
\Big\{(\om,x)\in \{1,2,\dots,\kappa\}^{\Z}\times X: \lim_{n\to\infty} \frac1n \sum_{j=0}^{n-1} \varphi(F^j(\om,x)) \; \text{does not exist}\Big\},
\end{equation}
then one has that
\begin{equation}\label{eq:inclus-irr}
I_F(\varphi_\psi) 
	\subset
	\{1 , 2, \dots, \kappa\}^{\mathbb N} \times I_{\mathbb S}(\psi) 
	\subset \{1 , 2, \dots, \kappa\}^{\mathbb N} \times X,
\end{equation}
where $\varphi_\psi(\om,x):=\psi(x)$ for every $(\om,x) \in \{1 , 2, \dots, \kappa\}^{\mathbb N} \times X$. 

We aim to describe the topological complexity (Baire genericity and topological entropy) 
of these sets of Birkhoff irregular points observing them as follows:
\begin{itemize}
\item[$\circ$] $I_{\mathbb S}(\psi)\subset X$ measured through the geometric entropy introduced in Ghys, Langevin and {Walczak} for a finitely generated 
	semigroup action;
\item[$\circ$] $I_{\mathbb S}(\psi)\subset X$ measured through the entropy introduced by Bufetov for a free semigroup action; \smallskip
\item[$\circ$] $I_{\mathcal F_\om}(\psi)\subset X$ as a set of points which are Birkhoff irregular for the non-autonomous dynamical system 
	$\mathcal F_\om=(f_{\omega_{i_j}})_{j\geqslant 1}$ determined by a realization $\om\in \Sigma_\kappa$; \smallskip
\item[$\circ$] $I_F(\varphi_\psi)$ as an $F$-invariant subset of $\{1 , 2, \dots, \kappa\}^{\mathbb N} \times X$.\smallskip
\end{itemize}

The entropy in the second item above corresponds to the annealed topological entropy of the skew-product $F$ and, in some special cases, 
the third one corresponds to the notion of Pinsker relative entropy (cf. Subsection~\ref{fibentropies}). 
The Lipschitz assumption on the generators implies that both notions of entropy on the semigroup $\mathbb S$ and the non-autonomous dynamical
system are finite.
These sets have been successively described for both single dynamics and a semigroup actions satisfying a specification 
property (see e.g. \cite{Barreira,Barreira1,CKS,ZhuMa2021} and references therein).
However, most smooth dynamical systems beyond hyperbolicity fail to satisfy the specification property and this is also the case
for the skew-products which arise in the study of linear cocycles \cite{BTV2}. 
For that reason we were led us to define the following mild topological concept which measures recurrence between balls 
of controlled radius. We denote by $|J|$ the radius of a ball $J\subset X$.

\begin{definition}\label{def:transition}
\emph{
We say that the semigroup action $\mathbb S$ generated by the collection $G_1=\{id, f_1, f_2, \dots, f_\kappa\}$ of continuous maps
has \emph{frequent hitting times} if for any $\varepsilon>0$ there exists $K(\varepsilon) \geqslant 1$ 
so that for any ball $B_1$ of radius $|B_1|=\vep$ and any ball $B_2$ of radius $0<|B_2| \leqslant \frac\vep2$ there exists $0\leqslant p \leqslant K(\vep)$
and an element ${\underline{\omega}} \in \Sigma_\kappa:=\{1,2,\dots, \kappa\}^\mathbb N$ 
so that $f^p_{\underline{\omega}}(B_1)$ contains a ball $B_2' \subset B_2$ of radius $|B_2|/2$.
}
\end{definition}

As defined above, the frequent hitting times property for a single map $f$ differs substantially from the specification or gluing orbit property.  It combines 
not only the control of the hitting time between balls of comparable radius and the size of their intersection. In particular, the latter implies that the set of hitting times $N(A,B)=\{n\in \mathbb N \colon  f^n(A)\cap B\neq \emptyset\}$ between two balls of the same radius is syndetic (cf. \cite{AAN} for the definition) and that the frequency of visits to balls of a definite radius is bounded away from zero by some constant which depends only on the radius of the balls
(cf. \cite{BTV}). 
We refer the reader to Section~\ref{sec:prelim} for an ample discussion on this property and examples.

\begin{maintheorem}\label{thm:IFS}
Let $X$ be a compact metric space and consider the semigroup action $\mathbb S$ generated by the collection of bi-Lipschitz homeomorphisms 
$G_1=\{id,f_1, f_2, \dots, f_\kappa\}$.
Assume that $\psi: X\to \mathbb R$ 
is a continuous observable and that: 
\begin{itemize}
\item[(a)] the semigroup action $\mathbb S$ has frequent hitting times,  
\item[(b)] there exists $1\leqslant i \leqslant \kappa$ so that $\psi$ is not a coboundary with respect to $f_i$. 
\end{itemize}
The set
$$
\mathscr R_\Sigma=\Big\{ \om \in \Sigma_\kappa\colon I_{\mathcal F_\om}(\psi) \; \text{is a Baire residual subset of}\, X \Big\}
$$ 
is a Baire residual subset of $\Sigma_\kappa$. 
In particular, $I_{\mathbb S}(\psi)$ is a Baire residual subset of $X$.
\end{maintheorem}

\color{black}

The previous result can be compared with the historic behavior obtained by Nakano~\cite{Nakano} in the context of random compositions of expanding maps in the circle. Indeed, exploring the uniformly expanding behavior of the maps, it is proven that for almost all $\omega\in \Sigma_\kappa$ (with respect to some reference probability $\mathbb P$) 
there exists a residual subset $\mathcal R_\omega\subset \mathbf S^1$ such that for each $x\in \mathcal R_\omega$
the random orbit $(f_\omega^n(x))_{n\geqslant 1}$ has historic behavior. Here we notice that no hyperbolic behavior is assumed on the fibers. 
Theorem ~\ref{thm:IFS} will be proven in Section~\ref{teoC}.

\medskip
The remaining results concern the entropy of the set of irregular points using two important concepts of entropy in the context of semigroup actions,
namely the entropy of free semigroup actions endowed with a random walk introduced by Bufetov~\cite{Bufetov} (see ~\cite{CRV,CRV2, JMW,ZhuMa2021} and references therein for further studies) and the geometric entropy introduced by Ghys, Langevin and Walczak \cite{GLW} (see also \cite{Bis2013} and references therein), which we recall in Subsection~\ref{sec:entropy-semigroup-annealed}.

\smallskip
The first notion is related, but does not coincide, with the entropy of the skew-product $F$ defined in ~\eqref{def:skew-eq}. Hence, given a continuous 
observable  $\psi: X\to \mathbb R$ it is natural to consider the quantity
\begin{align}
h_*(\varphi_\psi):=\sup\Big\{ 
c\geqslant 0 \colon 
& \text{there exist} \, \mu_1,\mu_2 \in \mathcal M_{erg}(F) \nonumber \\ 
& \,\text{so that}\, h_{\mu_i}(F)\geqslant c  \;\text{and}\, \int\varphi_\psi\, d\mu_1< \int\varphi_\psi\, d\mu_2
\Big\} \label{hestrela}
\end{align}
(where $\varphi_\psi$ is defined as before),  
which is the largest value of entropy almost attainable by ergodic probabilities which evaluate the observable 
$\varphi_\psi$ differently.   
The following result estimates the topological entropy of the set of Birkhoff irregular points for the skew-product $F$ with respect to the 
potential $\varphi_\psi$. 

\begin{maintheorem}\label{thm:B}
Let $X$ be a compact metric space and consider the semigroup action generated by the collection 
$G_1=\{id,f_1, f_2, \dots, f_\kappa\}$ of bi-Lipschitz homeomorphisms, and denote by $F: \{1,2,\dots,\kappa\}^{\Z}\times X \to \{1,2,\dots,\kappa\}^{\Z}\times X$ 
the skew-product $F(\om,x)=(\sigma(\om),f_{\omega_0}(x))$ where 
$\sigma$ is the shift.
Assume that $\varphi:\{1,2,\dots,\kappa\}^{\Z}\times X \to \R $ is a continuous observable. 
If
\begin{itemize}
	\item[(a)] the semigroup $\mathbb S$ generated by $G_1$ has frequent hitting times, and  
	\item[(b)] $\varphi$ is not a coboundary with respect to $F$
\end{itemize}
then $\htop(F, I_F(\varphi))\geqslant h_*(\varphi)$.
\end{maintheorem}

Some comments are in order. First, we observe that if the skew-product $F$ satisfies the specification property and the entropy map $\mu\mapsto h_\mu(F)$  
is upper-semicontinuous then $h_*(\varphi)=\htop(F)$ for every continuous map $\varphi$ which is not cohomologous to a constant (see \cite[Theorem~B]{EKW}).
However, the skew-product in ~\eqref{def:skew-eq} may even fail to satisfy a non-uniform specification property (cf.~Example~\ref{MS}) making it unclear if ergodic probabilities are dense or entropy dense in the simplex of $F$-invariant probabilities. Second,
the entropy $\htop(F, I_F(\varphi))$ appearing in the theorem is the entropy of the $F$-invariant subset $I_F(\varphi)$ of $\{1,2,\dots,\kappa\}^{\Z}\times X$,
as defined in Subsection~\ref{sec:single+sequential}.
The proof of the previous theorem will occupy Section~\ref{top-skew}.

\medskip
The notion of geometric entropy $h^{GLW}(\mathbb S)$ for the semigroup $\mathbb S$ depends on the generating set $G_1$ but does not depend on any intrinsic structure  
of the semigroup. Indeed, its definition encloses the exponential behavior of all possible concatenations of maps with the same number of 
generators (cf. Subsection~\ref{sec:entropy-semigroup-annealed} for the definition) and, consequently, it is larger or equal 
than the entropy $\htop(\mathcal F_\om)$ 
of any sequential (or non-autonomous) dynamical systems $\mathcal F_\om$ obtained with generating set $G_1$ (cf. Subsection~\ref{sec:single+sequential}
and equation~\eqref{comp:GLW.KS}). For that reason it is natural to consider the following quantity:
\begin{align}
	H^{KS}(\psi):=\sup\big\{ \htop(\mathcal F_\om) \colon \om \in\Sigma_\kappa
	\big\}. \label{Hestrela0}
\end{align}
It could occur that the elements $\om\in\Sigma_\kappa$ leading to a large entropy $\htop(\mathcal F_\om)$ above do not belong to the ergodic basin of attraction of some $\sigma$-invariant probability, in which case no information is known about the time averages of $\psi$ along the non-autonomous dynamical system 
$\mathcal F_\om$.
For that reason we introduce the following alternative quantities:
\begin{align}
	H^{\text{Pinsker}}(\psi):=\sup\Big\{ 
	c\geqslant 0 \colon & \text{for every}\, \vep>0\, \text{there exist} \, \mu_1,\mu_2 \in \mathcal M_{erg}(F) 
	\,\text{so that} \nonumber \\ 
	& h_{\mu_i}(F\mid\sigma) > c-\vep  \;\text{and}\, \int\varphi_\psi\, d\mu_1< \int\varphi_\psi\, d\mu_2
	\Big\}. \label{Hestrela2}
\end{align}
and
\begin{align}
	H^{\sigma}(\psi):=\sup\Big\{ 
	c\geqslant 0 \colon & \text{for every}\, \vep>0\, \text{there are} \, \mu_1,\mu_2 \in \mathcal M_{erg}(F) 
	\,\text{so that}  
	\nonumber \\ 
	& h_{\pi_*\mu_i}(\sigma) > c-\vep  \;\text{and}\, \int\varphi_\psi\, d\mu_1< \int\varphi_\psi\, d\mu_2
	\Big\}, \label{Hestrela3}
\end{align}
where we keep denoting by $\pi: \{1 , 2, \dots, \kappa\}^{\mathbb Z} \times X \to \{1 , 2, \dots, \kappa\}^{\mathbb Z} $ the projection on the first coordinate.
The latter quantities ~\eqref{Hestrela2} and ~\eqref{Hestrela3} measure the largest Pinsker entropy and entropy of the marginal on the shift, respectively, 
among the space of probability measures invariant by the skew-product which 
distinguish the observable $\varphi_\psi$. Once more, it can occur that the measures with larger entropy on the shift can have small fibered entropy. 
For that reason we consider take $H^{\text{Pinsker}}_\sigma(\psi):=\lim_{\zeta \to 0}H^{\text{Pinsker}}_\sigma(\psi,\zeta)$, where
\begin{align}
	H^{\text{Pinsker}}_\sigma(\psi,\zeta):=\sup\Big\{ 
	c\geqslant 0 \colon & \text{there are} \, \mu_1,\mu_2 \in \mathcal M_{erg}(F) 
	\,\text{so that} \;  h_{\mu_i}(F\mid\sigma)\geqslant H^{\text{Pinsker}}(\psi) -\zeta, 
	\nonumber \\ 
	& h_{\pi_*\mu_i}(\sigma)\geqslant  c  \;\text{and}\, \int\varphi_\psi\, d\mu_1< \int\varphi_\psi\, d\mu_2
	\Big\}. \label{Hestrela2-sigma}
\end{align}
We work these quantities and compare them in concrete examples of skew products associated to semigroup actions by continuous interval maps 
(see Examples~\ref{ex:comparison1} and \ref{ex:comparison2}).
We denote by $h^B(\mathbb S,I_{\mathbb S}(\psi))$ the geometric topological entropy of the set $I_{\mathbb S}(\psi)\subset X$ under the action of the semigroup 
$\mathbb S$, and by $h^{GLW}(\mathbb S,I_{\mathbb S}(\psi))$ when we observe it as a free semigroup action (cf. Subsections~\ref{sec:entropy-semigroup-annealed} and ~\ref{fibentropies} for the definitions of these entropies and relations 
between them). 
 Our next result shows that the geometric entropy of irregular sets for the semigroup action is bounded below by 
 these entropies.

\begin{maintheorem}\label{thm:Bb}
Let $X$ be a compact metric space and consider the semigroup action generated by a collection of bi-Lipschitz homeomorphisms 
$G_1=\{id,f_1, f_2, \dots, f_\kappa\}$.
Assume that $\psi: X\to \mathbb R$ 
is a continuous observable and that
\begin{itemize}
\item[(a)] the semigroup generated by $G_1$ has frequent hitting times, and 
\item[(b)] $\varphi_\psi$ is not a coboundary with respect to $F$.
\end{itemize}
Then:
\begin{enumerate}
\item  $h^{GLW}(\mathbb S,I_{\mathbb S}(\psi)) \geqslant H^{\text{Pinsker}}(\psi)$; 
\item 
$h^B(\mathbb S,I_{\mathbb S}(\psi)) \geqslant h_*(\varphi_\psi)$.
\end{enumerate}
\end{maintheorem}
This theorem will be proven in Section~\ref{top-semigroups}, where  item (2) above will be a consequence of Theorem~\ref{thm:B}
and a relation between the entropy of a free semigroup action with the entropy of the skew-product $F$. Finally, it is worth noticing 
that one can replace assumption (a) above by the weaker assumption that the semigroup $\mathbb S$ is \emph{strongly transitive}: for every open set
$U\subset X$ there exists $N=N(U)\geqslant 1$ so that $\bigcup_{g\in G_N} g(U)=X$. We refer the reader to the proof of Theorem~\ref{thm:Bb-2}
where this is made precise in the special case that a single generator in $G_1$ is strongly transitive.

\begin{remark}
The previous semigroup actions $\mathbb S$ and skew-product $F$ satisfy a very mild condition, and in general fail to satisfy any kind of specification. 
In the context of free semigroup actions satisfying the specification property, it is known that either $I_{\mathbb S}(\psi)=\emptyset$ or
$h^B(\mathbb S,I_{\mathbb S}(\psi)) =h^B(\mathbb S)$ (cf.  \cite[Theorem~1.2]{ZhuMa2021}).
\end{remark}

The information carried by Theorem~\ref{thm:Bb} is of global nature, enclosing information of the whole semigroup action $\mathbb S$. In particular it does not address any kind of information about the sequential dynamics that we now describe. 
For each $\om\in  \{1,2, \dots, \kappa\}^{\mathbb Z}$ consider the set
\begin{equation}\label{def-i-fibras}
I_\om(\psi):=X_\om \cap I_F(\varphi_\psi),
\end{equation} 
which consists of the set of points in $X_\om$ which have irregular behavior along the orbits of $\mathcal F_\om=(f_\omega^j)_{j\geqslant 1}$, and set
\begin{equation}\label{def:entrop-fibrada-irregular}
h^{path}_I(\mathbb S,\psi):= \sup_{\om\in \{1,2,  \dots, \kappa\}^{\mathbb Z}} \, h_{I_\om(\psi)}(\mathcal F_\om),
\end{equation} 
 where $h_{I_\om(\psi)}(\mathcal F_\om)$ is the relative entropy defined by the end of Subsection~\ref{fibentropies}. 
 In other words, $h_I^{fiber}(\mathbb S,\psi)$ represents the largest possible complexity of irregular sets obtained along an infinite path
 on the semigroup $G$.
The assumptions in Theorem~\ref{thm:Bb} allow one to produce Moran sets with large entropy, but whose intermediate transition times  functions 
(cf. Subsection~\ref{top-fibered-2}) 
depend on the points involved in the construction of such Cantor set and, consequently, these are not enough to estimate
the entropy of irregular sets along infinite paths in the semigroup. This is possible if one assumes instead a strong transitivity condition on one of the generators
of the semigroup. More precisely:

\begin{maintheorem}\label{thm:Bb-2}
Let $X$ be a compact metric space and consider the semigroup action generated by a collection of bi-Lipschitz homeomorphisms 
$G_1=\{id,f_1, f_2, \dots, f_\kappa\}$. Assume that $\psi: X\to \mathbb R$ is a continuous observable and that the following hold:
\begin{itemize}
\item[(a)]  there exists $1\leqslant i\leqslant \kappa$ so that $f_i$ is strongly transitive (i.e. for every open set $U\subset X$ there exists $N=N(U)\geqslant 1$ so 
	that $\bigcup_{n=1}^N f_i^n(U)=X$)
\item[(b)] $\varphi_\psi$ is not a coboundary with respect to $F$.
\end{itemize}
Then
\begin{enumerate}
\item  $h^{path}_I(\mathbb S,\psi) \geqslant H^{\text{Pinsker}}(\psi)$; 
\item the topological entropy of the set $\Sigma:=\{\om\in \Sigma_\kappa \colon h_{I_\om(\psi)}(\mathcal F_\om) \geqslant H^{\text{Pinsker}}(\psi)\}$ 
is larger or equal than $H^{\text{Pinsker}}_\sigma(\psi)$.
\end{enumerate}
\end{maintheorem}

Theorem~\ref{thm:Bb-2} is proven in Section~\ref{top-fibered}.
Notice that assumption (a) above is equivalent to the minimality of the homeomorphism $f_i$ in the generating set $G_1$. 
We opted to state the condition as a strong transitivity condition as this equivalent form may turn out to be useful for an extension
of this result to semigroups generated by non-invertible maps (see Subsection~\ref{sec:overview} for more details). Since 
$g_i$ is not assumed to be an isometry, the semigroup action is not required to have the frequent hitting times property. 
Item (2) in the previous theorem ensures that the set of points in $\Sigma_\kappa$ which generate sequential dynamical systems $\mathcal F_\omega$ 
with large irregular sets has large entropy, bounded below given by the largest possible entropy (on the shift) of the $F$-invariant probabilities   
whose fibered entropies can be used to compute $H^{\text{Pinsker}}(\psi)$.
Let us derive the following trivial consequence of Theorem~\ref{thm:Bb-2} whenever all generators coincide.

\begin{maincorollary}
Let $X$ be a compact metric space and $f: X\ \to X$ be a minimal bi-Lipschitz homeomorphism.
If $f$ is not uniquely ergodic and $\psi: X\to \mathbb R$ is a continuous observable which is not a coboundary then
$h_{I_f(\psi)}(f) \geqslant H^{\text{Pinsker}}(\psi)$.
\end{maincorollary}

\medskip
\subsection{Comparison with previous methods and approaches}\label{sec:overview}
 
Most well known methods to describe irregular sets make use of topological methods,  as some weak forms of specification or shadowing. 
In rough terms, in the additive context these are used in the following two step argument: for each continuous map $f: X\to X$ on the compact metric space $X$
satisfying the specification property and every continuous observable $\varphi: X\to \mathbb R$ one has
\begin{enumerate}
\item[(I)] if $\mu_1,\mu_2$ are ergodic probabilities so that $\int \varphi \, d\mu_1<\int \varphi \, d\mu_2$ there exist points which shadow arbitrarily long pieces of 
orbits whose time averages are close to $\int \varphi \, d\mu_1$ and $\int \varphi \, d\mu_2$, alternatively, and whose time between each shadowing process 
is a constant independing on the size of the shadowing orbits;
\item[(II)] for each convex combination $\mu$ of ergodic probabilities $\mu_1, \mu_2, \dots, \mu_s$ there exist points 
which shadow finite pieces of orbits determined by typical points for $\mu_i$, alternatively, and whose time averages converge to $\int \varphi\, d\mu$.
\end{enumerate} 
The first step guarantees that the irregular set is non-empty and can be used in a crucial way to show that this set is actually Baire generic (see e.g.
\cite{Barreira1} and references therein). The second step, combined with the fact that the times between shadowing processes are constant, 
allows to construct a Moran set formed by points with oscillatory behavior and carrying full entropy (cf. \cite{Barreira,CKS}).  
Specification can actually be used to study the saturated sets as obtained by Pfister and Sullivan ~\cite{PS} and the same holds when time lags between
shadowing processes are bounded by some constant or have sub-linear growth in comparison to the size of orbits (cf. ~\cite{LV,TV}). 
The skew-products 
\begin{equation*}
\begin{array}{cccc}
F \colon & \{1,2,  \dots, \kappa\}^{\mathbb Z} \times X & \longrightarrow  & \{1,2,  \dots, \kappa\}^{\mathbb Z} \times X \\
& (\underline{\omega},x) & \mapsto  & (\sigma(\underline{\omega}), f_{\omega_0}(x)),
\end{array}
\end{equation*}
associated to the semigroup action $\mathbb S$ with generating set $G_1=\{id, f_1,f_2, \dots, f_\kappa\}$
does not satisfy any of these properties as a consequence of \cite{BTV2} and Example~\ref{MS}.
In \cite{Tian,Tian2}, this is overcome by requiring an extra H\"older continuous regularity on the cocycle and making use of 
the exponential specification property for the base dynamical system. 

\smallskip
Here, inspired by the projective dynamics of linear cocycles, we introduce an alternative sufficient condition (satisfied whenever there exists a minimal 
sub-action by isometries, see Subsection~\ref{suff}) to study irregular points. Such condition is a quantitative estimate for hitting times between balls 
of controlled size and their intersection. 
The frequent hitting times property is unrelated to specification and its weaker versions (see Subsection~\ref{sec:vs-gluing} for a discussion), and 
following diagram presents some relations between the latter and related concepts:
\[
\begin{array}{cccc}
\text{minimal action by isometries}  &  \Rightarrow & \text{frequent hitting times}  & (\text{cf. Subsection~\ref{sec:vs-gluing}}) \\
& & \not\Uparrow \; \not\Downarrow \\
\text{ $C^0$-generic homeomorphisms}  &  \Rightarrow & \text{gluing orbit property  }  &\text{\cite{BeTV}} \\
& & \Uparrow \; \not\Downarrow  \\
\text{ uniform hyperbolicity }  &  \Rightarrow & \text{specification property  }   &\text{\cite{Bo71}} \\
& & \Downarrow  \; \not\Uparrow  \\
\text{ non-uniform hyperbolicity }  &  \Rightarrow & \text{non-uniform specification property  }  &\text{\cite{OliTian,Paulo}} \\
& &   \\
\text{partial hyperbolicity }  &  \Rightarrow &  \text{'typically' {\bf no} specification }  & \text{\cite{SVY}} \\
 &   &  \text{and {\bf no}  gluing orbit property  }  &\text{\cite{BTV2}} \\
\end{array}
\]

 \smallskip
 The assumption that the homeomorphisms $f_i$ are bi-Lipschitz is used to estimate the inner diameter of the images
of dynamic balls (cf. Lemma~\ref{le:est-balls}).  This information, combined with the frequent hitting times assumption, provides quantitative 
estimates on the amount of time needed to perform the shadowing of the finite pieces of orbits. Notably, albeit transition times are not constant and can be
much larger than the pieces of orbits that have been shadowed (these actually grow exponentially fast), these satisfy a Markov property, meaning that depend exclusively on the size of the last piece of orbit. This provides a new argument for step (I) above and, under much weaker assumptions, still to prove that the irregular sets are non-empty and Baire generic (we refer the reader to Figure~\ref{fig01} for a representation of the control over the oscillations under 
the different previous circumstances). 
 
 \smallskip
 The situation differs substantially in step (II), which is the most relevant to the construction of Moran sets with arbitrarily large entropy. 
First, while in the case of maps with specification one can make successive approximations of points which shadow very large pieces of orbits, this strategy
 falls shortly in our context due to the much larger transitions times associated. We introduce a new argument, which consists of controlling the forward images of dynamic balls for all iterates, guaranteeing that a proportion of points remains possible to iterate at each level of the construction (see $2^{nd}$Step at Subsection~\ref{MLEntropy}). 
 The construction is made by choosing points which, alternatively, shadow pieces of orbits which are typical for ergodic probabilities $\mu_1,\mu_2$ having 
 different space averages with respect to $\varphi$. A subtle but key difference is that since one cannot replicate small orbits by its repetition, the size of 
$n_k$ the finite pieces of orbits to be shadowed at level $k$ are chosen depending on the maximal transition time $K$ involving the pieces of orbits
of size $n_{k-1}$ used at the level $k-1$. Selecting points which have the same transition times at each level of approximation guarantees that 
one can construct distinguishable orbits. These choices are performed in Lemma~\ref{le:selection}.
Second, the failure of step (II), imposed by the very long transition times, impart that one can only use typical points associated to ergodic probabilities
as in step (I) for the construction of a Moran set of irregular points. This explains the lower bounds on the entropy of irregular sets provided in  
Theorems~\ref{thm:Bb} and ~\ref{thm:B}, and the fact that every point is irregular for some path in the semigroup 
(for instance, recall the first statement in Theorem~\ref{thm:abstract}).

 \smallskip
Finally, notice that although the main results have been stated for semigroups generated by homeomorphisms, 
the argument just makes use of the semigroup structure.  
The bi-Lipschitz continuity of the generators guarantees not only that the semigroup action (and the skew-product) have finite topological entropy but is also is crucial to control the inner diameter of dynamic balls and their images  (cf. Lemma~\ref{le:est-balls}).
Nevertheless, the Ces\`aro averages  and the proofs of the main results take into consideration just the forward iterates of the mappings. 
A simple modification of the argument is likely to show that the main results hold as well for local homeomorphisms which are Lipschitz continuous
and whose inverse branches are Lipschitz continuous as well, but we do not pursue this here.

\subsection{Organization of the text}\label{se:organization} 

The remainder of the  paper is organized as follows. 
\begin{itemize}
\item[$\diamond$] In Section~\ref{sec:prelim-hyp-coboundaries}  we recall basic definitions concerning uniform and partial hyperbolicity, dominated splittings and
	coboundaries.
\item[$\diamond$] In Section~\ref{subsec:entropies} we include a broad discussion about the different notions of topological and measure theoretical entropies 
	for a single map, non-autonomous dynamical systems  and semigroup actions that will be used here. We first recall the notion of fibered entropy for 
	extensions and the entropy of non-autonomous 
	dynamical systems, and prove a fibered entropy distribution principle.  Then we describe two notions of entropy of a finitely generated 
	semigroup action due to Ghys, Langevin and Walczak (sometimes known as geometric entropy) and another due to Bufetov, which 
	exploits a coding by means of a free semigroup action and makes use of a skew-product map over a full shift on finitely many symbols.
\item[$\diamond$] Section~\ref{sec:prelim} is devoted to the study of the frequent hitting times condition. 
	 We provide sufficient conditions for it to hold, and give some examples on how this is useful in the context of
	projective maps arising from linear cocycles which have some projective accessibility. 
	Furthermore, we observe that this property is much weaker and unrelated to other gluing orbit properties
	and need not occur for uniformly hyperbolic maps.
\item[$\diamond$] 	
	In Section~\ref{teoC} we start a series of results on the set of irregular points with respect to a finitely generated semigroup action by continuous maps
	on a compact metric space $X$. Here we prove Theorem~\ref{thm:IFS}, stating that under the mild frequent hitting times condition, Birkhoff irregular sets for semigroup actions are either empty or Baire generic  
	on the phase space $X$.
\item[$\diamond$] In Section~\ref{top-skew}, motivated by the study of free semigroup actions,  we consider step skew-products  
	with the frequent hitting times property, and prove that its set of irregular points in $\Sigma_\kappa \times X$ carries at least as much entropy as the 
	one of any pair of ergodic probabilities with 	
	different space averages (recall Theorem~\ref{thm:B}). This lower bound reflects the fact that under this mild 
	assumptions one should not expect the entropy denseness of the space of ergodic probability measures. 
\item[$\diamond$]  In Section~\ref{top-semigroups} we obtain lower bounds for the topological entropy of irregular points for the semigroup action, considering
	both the Ghys-Langevin-Walczak and Bufetov entropies. While Bufetov entropy measures the average number of points with irregular behavior for 
	the non-autonomous dynamical systems associated to the different concatenations of the generators of the group and can be studied in relation to the skew-product map, the lower bound
	for the Ghys-Langevin-Walczak entropy demands a careful choice of a set of points which can be separated by some element in the group.
\item[$\diamond$]  In Section~\ref{top-fibered} we consider the topological entropy of fibered irregular sets, that is, the largest possible entropy of irregular sets
	associated to all possible non-autonomous dynamical systems. There are two main differences in comparison to the results of Sections~\ref{top-semigroups} 
	and ~\ref{top-fibered}. First, in order to construct Moran subsets formed by points with irregular behavior for certain (actually a large entropy set of) 
	non-autonomous dynamical systems we require some generator in $G_1$ to be strongly transitive. Second, the strong transitivity is weaker than the 
	frequent hitting times condition, hence this furnishes an alternative mechanism for homeomorphisms that do not arise as projective linear maps.
\item[$\diamond$]  In Section~\ref{sec:linear-cocycles} we apply the theory developed to general finitely generated semigroup actions to the context for 
	projective linear cocycles. 
	In Subsection~\ref{subsec:projectivelinear} we present some preliminary results on the of projective linear maps and invariant measures for the 
	projectivized skew-product. 
	In Subsection~\ref{sec:p-abstract} 
	the non-compactness of the semigroup generated by the $SL(d,\mathbb R)$ matrices allows to relate hyperbolic behavior of the 
	concatenations of matrices along specific vectors with its recurrence to certain regions of the projective space, and use this fact to prove
	Theorem~\ref{thm:abstract}.
	Actually in the special case of $SL(3,\mathbb R)$-cocycles the absence of irregular behavior leads to a rigidity that the cocycle is cohomologous to a constant 
	matrix cocycle
	(cf. Theorem~\ref{thm:rigidity} and Subsection~\ref{sec:p-rigidity}).
	We finish this section with few examples. 
\end{itemize}

\section{Uniform hyperbolicity and specification type properties}\label{sec:prelim-hyp-coboundaries}

In this brief section, included for making the paper self-contained, 
we recall some basic concepts concerning uniform, partial and dominated splittings for linear and derivative cocycles, 
{the concept of specification and some of its weakenings},
and 
recall some characterizations of coboundaries due to Liv\v{s}ic. 

\subsection{Hyperbolicity, partial hyperbolicity and dominated splittings}\label{subsec:hyp}

Given a continuous linear cocycle $A: {\Sigma_\kappa} \to \mathrm{SL}(d,\mathbb R)$ we say that an $A$-invariant splitting
$\Sigma_\kappa \times \mathbb R^d=E \oplus F$ is \emph{dominated} if there exists an integer $k\in\mathbb{N}$ such that
$$\frac{\|A^k(\om)v\|}{\|A^k(\om)w\|}<\frac{1}{2},$$
for every $\om\in {\Sigma_\kappa}$ and every pair of unit vectors $v\in E_\om$
and $w\in F_\om$.
More generally, an $A$-invariant splitting $\Sigma_\kappa \times \mathbb R^d=E_1\oplus\cdots\oplus E_k$ is \emph{dominated} if
$(E_1\oplus\cdots\oplus E_l)\oplus(E_{l+1}\oplus\cdots\oplus E_k)$ is dominated for any $1\le l\le k-1$.
The $A$-invariant sub-bundle $E$ is \emph{uniformly contracting} (resp. \emph{expanding}) if
there are $C>0$ and $0<\lambda<1$ such that for every $n>0$ one has
$\|A^n(\om)v\|\leqslant C\lambda^n\| v\|$ (resp. $\| (A^n(\om))^{-1}(v)\|\le
C\lambda^n\| v\|$) for all $\om\in \Sigma_\kappa$
and $v\in E_\om$.
Finally, we say that an $A$-invariant splitting $TM=E^s\oplus E^c\oplus E^u$ is: (i) \emph{partially hyperbolic} (resp. \emph{strongly partially hyperbolic})
if $E^s$ and $E^u$ are uniformly contracting and uniformly expanding respectively
and at least one of them is (resp. both of them are) not trivial; and (ii) \emph{hyperbolic} 
if it is strongly partially hyperbolic and $E^c=\Sigma_\kappa\times \{0\}$.
\smallskip
These notions have a natural counterpart for $C^1$-diffeomorphisms acting on a compact Riemannian manifold $M$. Indeed, such a $C^1$-diffeomorphism
$f: M\to M$ is \emph{uniformly hyperbolic} (or an Anosov diffeomorphism) (resp. \emph{partially hyperbolic}) if the derivative cocycle $Df:TM\to TM$ on the tangent bundle $TM$ 
is uniformly hyperbolic (resp. {partially hyperbolic}).

\subsection{Specification and gluing orbit properties}\label{subsec:spec}

First we recall the concept of specification,  introduced by Bowen \cite{Bo71} in the context of hyperbolic maps, which in rough terms 
describes the ability to shadow an arbitrary number of finite pieces of orbits of arbitrary length having a prefixed number of iterates in between.
Let us be more precise. 

\smallskip
Given a continuous map $f\colon X \to X$ on a  compact metric space $X$,
we say that $f$ satisfies the \emph{gluing orbit property} if for any $\vep>0$ there exists a positive integer
$N(\vep)\geqslant 1$ such that for any points 
$x_1,\dots, x_k\in X$ and any
integers
$n_1, \dots, n_k \geqslant 0$ and $N_1, N_2, \dots, N_{k-1}\geqslant N(\vep)$
there exists a point $z \in X$ so that
$d(f^n(z),f^n(x_1)) < \vep$ for every $n=0 \ldots n_1$ 
and $d(f^{n+\sum_{1 \leq i <j}(N_i+n_i)}(z),f^n(x_j))<\vep$
 for every $2 \leqslant j \leqslant k$ and $0 \leqslant n \leqslant n_j$. 
It is clear that the latter implies that $f$ is topological mixing.

A somewhat weaker notion has been used ever since,
especially in the symbolic context, which some authors sometimes referred to as the specification property. 
In the context of topological dynamical systems this was formalized in \cite{BV} as follows.
We say that a continuous map $f\colon X \to X$ on a  compact metric space $X$,
satisfies the \emph{gluing orbit property} if: given $\vep>0$ there exists a positive integer
$N(\vep)\geqslant 1$ such that for any points 
$x_1,\dots, x_k\in X$
and times 
$n_1, \dots, n_k \geqslant 0$
there exist gluing times 
$0 \leqslant p_1, \dots, p_{k-1}  \leqslant N(\vep)$
and a point $z \in X$ so that
$d(f^n(z),f^n(x_1)) < \vep$ for every $n=0 \ldots n_1$ 
and $d(f^{n+\sum_{1 \leq i <j}(p_i+n_i)}(z),f^n(x_j))<\vep$
 for every $2 \leqslant j \leqslant k$ and $0 \leqslant n \leqslant n_j$. 

\smallskip
Even though the previous two notions can seem similar at first sight, the gluing orbit property is substantially weaker than the notion of specification, and it holds for non-mixing hyperbolic basic sets, 
minimal rotations and some partially hyperbolic maps with neutral direction.  
Hence, while the specification property implies on topological mixing and positive topological entropy, the same does not hold
for all continuous maps with the gluing orbit property. A characterization of zero entropy maps with the gluing orbit property appeared in \cite{Sun}.
We refer the reader to  \cite{BTV} and references therein for a more embracing discussion on these topological invariants.

\smallskip
Finally, we shall refer to a measure theoretical non-uniform  specification property, 
that proved to hold for invariant measures of local diffeomorphisms with no zero Lyapunov exponents (cf.~\cite{OliTian,Paulo}).
Indeed, if $f$ is a $C^1$-local diffeomorphism and  $\mu$ is an $f$-invariant probability measure,
we say that $(f,\mu)$ satisfies the \emph{non-uniform specification property}
if there exists $\de>0$ and for $\mu$-almost every $x$ and $0<\vep<\de$ there exists an integer $p(x,n,\vep)\geqslant 1$
satisfying
\begin{equation}\label{def:nuspec}
\lim_{\vep\to 0}\limsup_{n \to \infty} \frac{1}{n} p(x,n,\vep)=0
\end{equation}
and such that the following property holds:
given $x_1, \dots, x_k$ in a full $\mu$-measure set
any integers $n_1, \dots, n_k\geqslant 1$ and $p_i \geq p(x_i,n_i,\vep)$
then there exists $z \in B(x_1,n_1,\vep)$ such that
$
 f^{\sum_{j=1}^{i-1} (n_j+p_j)}(z) \in B(x_i,n_i,\vep)$
for every $2\leqslant i\leqslant k.$
In contrast to their topological counterparts, in this context the time lag between any two consecutive pieces of
orbits may depend on the sizes of the orbits, but it is required to have sub-linear grow in respect to these (cf. ~\eqref{def:nuspec}).

\subsection{Coboundaries and Liv\v{s}ic-like characterization of coboundaries}\label{subsec:Livsic}

Given a continuous map $f:X\to X$ on a compact metric space $X$, we say that a continuous observable $\varphi : X \to \mathbb R$ is 
\emph{not a coboundary} (with respect to $f$) if 
\begin{equation}\label{eq:ncohom}
\inf_{\mu\in \cM_1(f)}  \int \varphi \,d\mu < \sup_{\mu\in \cM_1(f)}  \int \varphi \,d\mu
\end{equation}
where $\cM_1(f)$ denotes the space of $f$-invariant probability measures.
In the special case that $f$ satisfies the specification property the latter is equivalent to say that $\varphi$ satisfies  \eqref{eq:ncohom}
if and only if $\varphi$ is not $C^0$-approximated by a sequence of functions $(\varphi_n)_n$ of the form
$\varphi_n = u_n - u_n\circ f +c_n$ for some $c_n\in \mathbb R$ and some continuous $u_n: X\to \mathbb R$. Any such function $\varphi_n$ is called
coholomogous to a constant and, if it is H\"older continuous then it is characterized through Birkhoff averages at periodic points by Liv\v{s}ic  theorem \cite{Livsic}:

\begin{theorem}
Let $G$ be a metric group, $f: M \to M$ be an Anosov diffeomorphism and let $A: M \to G$ be H\"older continuous cocycle. If either $G=\mathbb R$
or $G$ is an abelian group and $A,B: M \to G$ are so that $A^n(p)=B^n(p)$ whenever $f^n(p)=p$ then there exists a continuous cocycle $C: M\to G$ so that $A(x)=C(f(x)) B(x) C(x)^{-1}$
\end{theorem}

In the context of an Anosov diffeomorphism $f$, as the space of periodic probability measures is dense in the space of invariant probability measures, 
it is not hard to use the previous theorem to ensure that, if the observable $\varphi$ is H\"older continuous then condition \eqref{eq:ncohom} fails if and only if 
there exists a continuous function $\psi: M \to \mathbb R$ and $c\in \mathbb R$ so that $\varphi=\psi\circ f-\psi + c$.
The set of irregular points associated to a continuous observable $\varphi$ satisfying  ~\eqref{eq:ncohom}  and a continuous map with the specification property
is a Baire residual subset of $X$ and it has full topological entropy (cf. \cite{Barreira,CKS,Tho2012}).

\medskip
In the context of finitely generated semigroup actions common invariant measures seldom exist, hence 
the concept of coboundary is extremely rigid.  For instance, it is known that a finitely generated group action admits a common invariant probability 
if and only if the group is amenable (see e.g. \cite{OW}). For that reason, we will only refer to coboundaries with respect to individual dynamical systems 
or the skew-product determined by the semigroup action. Furthermore, even though the concepts of orbital specification for semigroup actions and 
semigroup action introduced in \cite{JMP} can probably be used as a tool to characterize irregular points for certain classes of semigroup actions 
(e.g. generated by expanding maps), the class of semigroup actions treated here do not satisfy any of the traditional gluing orbit properties 
(see Example~\ref{MS} for more details).

\section{Entropy and upper capacities}\label{subsec:entropies}

In the core of this section we relate several notions of entropy for semigroup actions and non-autonomous dynamical systems, 
and prove entropy distribution principle for fibered maps.

\subsection{Topological entropy of continuous maps and non-autonomous dynamical systems}\label{sec:single+sequential}
Let $(X,d)$ be a compact metric space and $f: X\to X$ be a continuous map. Given $\vep > 0$~and $n \in \mathbb N$, we say that $E\subset X$ is 
$(n, \vep)$-separated if for every $x \neq y\in E$ it holds
that $d^f_n (x, y)> \vep$, where $d^f_n (x, y) = \max\{d(f^j (x), f^j (y)) \colon 0\leqslant j \leqslant  n - 1\}$ denotes Bowen's distance. The sets $B_{f}(x, n,\vep) =\{ y\in X: d_n^{f}(x, y) < \vep\}$ are called Bowen dynamic balls. 
If $s(n,\vep)$ denotes the maximal cardinality of a $(n, \vep)$-separated subset of $X$ then the \emph{topological entropy} of $f$ is 
\begin{eqnarray*}
\displaystyle h_{top}(f) = \lim_{\vep \to 0} \limsup_{n \to \infty} \frac{1}{n} \log s(n,\vep).
\end{eqnarray*}
The classical variational principle for entropy ensures that 
$
\htop(f) =\sup_{\mu\in \mathcal M_{inv}(f)} \{ h_\mu(f) \}.
$

\smallskip
Similarly, Kolyada and Snoha \cite{KS} introduced the concept of topological entropy of non-autonomous dynamical systems $\mathcal F=(f_n)_{n\geqslant 1}$,
described by the iterates $F_n:=f_n \circ f_{n-1}\circ \dots \circ f_1$ for each $n\geqslant 1$, as
\begin{eqnarray*}
\displaystyle h_{top}(\mathcal F) = \lim_{\vep \to 0} \limsup_{n \to \infty} \frac{1}{n} \log s(\mathcal F,n,\vep),
\end{eqnarray*}
where $s(\mathcal F,n,\vep)$ denotes the maximal cardinality of a subset $E\subset X$  so that 
$$
\max\{d(F_j (x), F_j (y)) \colon 1\leqslant j \leqslant  n - 1\}>\vep
	\quad\text{for any distinct $x,y\in E$.}
$$
To each $\om\in \Sigma_\kappa$ we associate the (one-sided) non-autonomous dynamical system
 $\mathcal F_\om=(f_{\omega_j})_{j\geqslant 1}$.

\smallskip
The topological entropy of invariant and not necessarily compact subsets $Z\subset X$ is defined using a Charath\'eodory structure. 
Given $\vep>0$, $\alpha \in \R$ and $N\geq 1$, define 
\begin{equation*}
m_\al(f,Z,\vep,N) = \inf_{\mathcal G} \Big\{
\sum_{B(x,n,\vep) \in \mathcal G}
  e^{-\alpha n} \Big\}
\end{equation*}
where the infimum is taken over all finite or countable families
$\mathcal{G}$ of dynamic balls of length $n\geqslant N$ 
so that $\bigcup_{B(x,n,\vep) \in \mathcal G} B(x,n,\vep)$ covers $Z$. 
Let
$
m_\al(f,Z, \vep)
  = \lim_{N \to \infty} m_\al(f,Z,\vep,N)
$
and
$$
\htop(f,Z,\vep) 
= \inf{\{\alpha:
m_\al(f,Z,\vep)=0\}}
= \sup{\{\alpha:
m_\al(f,Z,\vep)=+\infty\}}.
$$
Then, \emph{topological entropy (also known as topological capacity) of $Z \subset X$} is given
by $\htop(f,Z)= \lim_{\vep \to 0} \htop(f,Z,\vep)$ and satisfies the partial variational principle
$\htop(f,Z)\geqslant \sup\left\{ h_\mu(f) \right\}$ where the supremum is over all $f$-invariant probability
measures $\mu$ such that $\mu(Z)=1$.
Moreover, in the special case that $Z\subset X$ is $f$-invariant and compact 
we have that $\htop(f,Z)=\htop(f\mid_Z)$.
We refer the reader to \cite{Pe97} for the proofs and further details in the case of a single map.
\smallskip

\medskip
In order to estimate the topological entropy of non-compact sets, we will make use of the following extremely useful 
entropy distribution principle, due to Takens and Verbitski (see \cite[Theorem~3.6]{TV03}):

\begin{theorem}\label{thm:EDP}
Let $X$ be a compact metric space, $f: X \to X$ be a continuous map and $Z\subset X$. Suppose that there exists $s\geqslant 0$ so that for any  sufficiently small $\vep>0$ one can find a Borel probability measure $\mu=\mu_\vep$ and $C(\vep)>0$ satisfying: (1) $\mu_\vep(Z)>0$, and (2) $\mu_\vep(B_f(x,n,\vep)) \leqslant C(\vep) e^{-s \,n}$ whenever
$B_f(x,n,\vep) \cap Z\neq\emptyset$. Then $\htop(f,Z) \geqslant s$.
\end{theorem}

\medskip
Finally, we recall a characterization of the Kolmogorov-Sinai metric entropy due to Katok.  
As a consequence of \cite[Theorem~I.I]{Katok}, if $\mu$ is an $f$-invariant and ergodic probability measure then for any $\rho>0$ it holds that  
$h_{\mu}(f)=\lim_{\vep \to0} \bar h_{\mu}(f,\vep,\rho)=\lim_{\vep \to0} \underline{h}_{\mu}(f,\vep,\rho)$ where
\begin{eqnarray}\label{eq entropy-Katok}
	\bar h_{\mu}(f,\vep,\rho)=\limsup_{n\rightarrow +\infty}\frac{1}{n}\log N^{\mu}_n(\vep,\rho),
	\qquad 
	\underline{h}_{\mu}(f,\vep,\rho)=\liminf_{n\rightarrow +\infty}\frac{1}{n}\log N^{\mu}_n(\vep,\rho),
\end{eqnarray}
and $N^{\mu}_n(\vep,\rho)$ denotes the minimal cardinality of an $(n,\vep)$-generating set for some set $Z\subset X$  
with measure at least $1-\rho$. It will be useful to use that a minimal cardinality $(n,\vep)$-generating set is an
$(n,\frac\vep2)$-separated set, not necessarily of maximal cardinality.

\subsection{Topological entropies of finitely generated semigroup actions}\label{sec:entropy-semigroup-annealed}

Given a compact metric space $X$ and the finite collection $G_1=\{id,f_1, f_2, \dots, f_\kappa\}$ of continuous self maps on $X$, 
consider the finitely generated semigroup $(G,\,\circ)$ generated by $G_1$, where the semigroup operation is the composition of maps. 
Setting $G_n:=\big\{ g_{n}\circ \dots g_2 \circ g_1 \circ g_0 \colon \; g_i\in G_1 \big\}$ it is clear that 
$G_n\subset G_{n+1}$ for every $n\geqslant 1$.

\subsubsection{Geometric entropy}
We first recall the \emph{geometric entropy} of a semigroup, introduced by Ghys, Langevin and Walczak ~\cite{GLW}.
Given $n \in \mathbb{N}$ and $\varepsilon > 0$, the $(n,\vep)$-Bowen ball generated by the semigroup action and centered at $x$ is defined by
$B^G_n(x,\varepsilon) = \big\{y \in X \colon \, d(g(x), g(y)) < \varepsilon, \; \forall \, g\in G_n\big\}$.
One says that two points $x, \,y \in X$ are $(n, \varepsilon)$-separated by $G$ if there exists $g \in G_n$ such that $d(g(x), \, g(y)) \geqslant \varepsilon$.
A subset $E$ of $X$ is $(n, \varepsilon)$-separated if any two distinct points of $E$ are $(n, \varepsilon)$-separated by $G$. Having fixed $n\geqslant 1$ and $\varepsilon>0$, consider
$$s(G, G_1,n, \varepsilon) = \max\,\big\{|E| : \,\, E \subset X \text{ is $(n, \varepsilon)$-separated}\big\}.$$
Since $X$ is compact, $s(G, G_1,n, \varepsilon)$ is finite for every $n \in \mathbb{N}$ and $\varepsilon > 0$.
The \emph{geometric entropy} of the free semigroup $G$ generated by $G_1$ is defined by
\begin{equation}\label{def:GLW}
h^{GLW}(\mathbb S) = \lim_{\varepsilon \, \to \, 0} \,\limsup_{n\, \to\, +\infty}\,\frac{1}{n}\,\log \,s(G, G_1,n, \varepsilon)
\end{equation}
(the limit as $\vep$ tends to zero exists by monotonicity), where we omit the dependence on $(G,G_1)$ for notational simplicity. 
It is clear from the definition that $s(G, G_1,n, \varepsilon)\geqslant 
s(\mathcal F_\om,n,\vep)$ for every $n\geqslant 1$, $\vep>0$ and $\om\in \Sigma_\kappa$ and, consequently, 
\begin{equation}\label{comp:GLW.KS}
h^{GLW}(\mathbb S)  \geqslant \htop(\mathcal F_\om), 
	\qquad \forall \om\in \Sigma_\kappa.
\end{equation}
Similarly, one can define the \emph{Ghys-Langevin-Walczak topological capacity of a set $Y\subset X$} by
\begin{equation}\label{def:GLW-Y}
h^{GLW}(\mathbb S,Y) = \lim_{\varepsilon \, \to \, 0} \,\limsup_{n\, \to\, +\infty}\,\frac{1}{n}\,\log \,s(G, G_1, Y,n,\varepsilon)
\end{equation}
where $s(G, G_1, Y,n,\varepsilon)$ denotes the largest cardinality of a set $E\subset Y$ that is $(n, \varepsilon)$-separated by $G$.
For the sake of completeness, let us mention that this notion of entropy was rediscovered independently by Hofmann and Stoyanov~\cite{HS}
in the context of continuous actions generated locally compact semigroups.

\subsubsection{Free semigroup actions}
A different approach, inspired by the context of random walks on groups, is to code the semigroup using the free semigroup on $\kappa$ generators.
In this way one consider paths on the semigroup determined by finite words in $\{1,2,  \dots, \kappa\}^{\mathbb Z}$.
More precisely, the semigroup action $\mathbb S$ can be coded by the skew-product
\begin{equation*}\label{de.Skew product}
\begin{array}{rccc}
F : & \{1,2,  \dots, \kappa\}^{\mathbb Z}  \times X & \to & \{1,2,  \dots, \kappa\}^{\mathbb Z}  \times X \\
	&\quad \quad (\om,x) & \mapsto & (\sigma(\om), f_{\omega_0}(x))
\end{array}
\end{equation*}
where the shift codes the possible paths in the free semigroup with $\kappa$ generators. The topological entropy of $F$-invariant subsets  
$Z\subset \{1,2,  \dots, \kappa\}^{\mathbb Z}  \times X$ is defined as before. However, this notion is not fitted specifically to deal with semigroup actions
as there is an artificial creation of entropy due to the shift, a fact which suggests some alternative notions of entropy. 

The first one was introduced by Bufetov in \cite{Bufetov} in the case of symmetric random walks and studied by Carvalho, Rodrigues and the second named 
author in \cite{CRV} for asymmetric random walks.
Given $\underline g= g_{i_n} \dots g_{i_1} \in G_n$, with $g_{i_j}\in G_1$, we say that a set $K \subset X$ is \emph{$(\underline g, n, \varepsilon)-$separated} if 
$
 \max_{0\,\leqslant\, j \,\leqslant \,n-1 } d(\underline{g}_{\,j}(x),\underline{g}_{\,j}(y))> \varepsilon
$ 
 for any distinct $x,y \in K$ where, for every $1 \leqslant j \leqslant n-1$, $\underline g_{\,j}$ stands for the concatenation $g_{i_{j}} \dots g_{i_2} \, g_{i_1}\in G_j$, and $\underline g_0=id$. 
The {topological entropy of the semigroup action $\mathbb{S}$} with respect to 
a random walk $\mathbb{P}$ is defined as
$$h_{\text{top}}(\mathbb{S}, \mathbb{P}) :=\lim_{\vep\to 0}\limsup_{n\to\infty}\frac1n\log
	\int_{\Sigma_\kappa^+} \, s(g_{\omega_n} \dots g_{\omega_1}, n,\vep) \, d\mathbb{P}(\omega),
$$
where
$s(g_{\omega_n} \dots g_{\omega_1}, n,\vep)$ is the maximum cardinality of $(\underline g,n,\varepsilon)-$separated sets,
and it coincides with a certain annealed topological pressure for the skew-product $F$ (see \cite[Proposition~4.1]{CRV2}).
In the special case of the symmetric random walk  $\mathbb P_{sym}$ on $G$, which corresponds to the Bernoulli probability $(\frac1{\kappa}, \frac1{\kappa}, \dots, \frac1{\kappa})^{\mathbb Z}$ on the shift, this is written as   
\begin{equation}\label{def:entropysym}
h_{\text{top}}(\mathbb{S}, \mathbb{P}_{sym}) :=\lim_{\vep\to 0}\limsup_{n\to\infty}\frac1n\log 
	\Big( \frac1{\kappa^n} \sum_{\underline g \in G_n} s(\underline g, n,\vep) \Big).
\end{equation}
In this paper we will always consider symmetric random walks.

\medskip
In analogy to the case of a single map, the complexity of non-compact subsets $Z\subset X$ can be defined using a topological capacity or a Carath\'eodory structure.
The \emph{(upper) topological capacity of the semigroup $\mathbb S$} on a subset $Z\subset X$ is defined by
\begin{equation}\label{def:uppercapacityS}
\overline{Ch_{\text{top}}}(\mathbb{S}, Z, \mathbb{P}_{sym}) :=\lim_{\vep\to 0}\limsup_{n\to\infty}\frac1n\log 
	\Big( \frac1{\kappa^n} \sum_{\underline g \in G_n} s(\underline g, Z, n,\vep) \Big),
\end{equation}
where $s(\underline g, Z, n,\vep)$ denotes the largest cardinality of 
a $(\underline g,n,\varepsilon)$-separated subset of $Z$.
In the particular case that $X$ is compact, this coincides with the definition of entropy in ~\eqref{def:entropysym} (see \cite{JMW} for more details).
For notational simplicity, we will denote the previous expression simply by $\htop(\mathbb{S}, Z,\mathbb{P}_{sym})$. 

\medskip
The \emph{annealed topological entropy of the semigroup} of a subset $Z\subset X$ was introduced 
in \cite{JMW} as follows: given $\vep>0$, $\alpha \in \R$, $N\geq 1$ and $\underline g\in G_N$,  defining 
\begin{equation*}
m_\al(\underline g,Z,\vep,N) = \inf_{\mathcal G_{\underline g}} \Big\{
\sum_{B_{\tilde{\underline g}}(x,n,\vep) \in \mathcal G_{\underline g}}
  e^{-\alpha n} \Big\}
\end{equation*}
where the infimum is taken over all finite or countable families
$\mathcal G_{\underline g}$ of dynamic balls $B_{\tilde{\underline g}}(x,n,\vep)$ determined by elements
$\tilde{\underline g}\in G_n$ with $n\geqslant N$ so that $\bigcup_{B_{\tilde{\underline g}}(x,n,\vep) \in \mathcal G} B_{\tilde{\underline g}}(x,n,\vep)$ 
covers $Z$, 
and
\begin{equation*}
m_\al(\mathbb S, Z,\vep,N) = \frac1{\kappa^N} \sum_{\underline g\in G_N} m_\al(\underline g,Z,\vep,N),
\end{equation*}
the limit 
$
m_\al(\mathbb S,Z, \vep)
  = \lim_{N \to \infty} m_\al(\mathbb S,Z,\vep,N)
$
does exist, and 
$$
h^B(\mathbb S,Z,\vep) 
= \inf{\{\alpha:
m_\al(\mathbb S,Z,\vep)=0\}}
= \sup{\{\alpha:
m_\al(\mathbb S,Z,\vep)=+\infty\}}.
$$
Then, \emph{topological entropy of $Z \subset X$} is given
by $h^B(\mathbb S,Z)= \lim_{\vep \to 0} h^B(\mathbb S,Z,\vep)$. 
It holds that 
$$
h^B(\mathbb S,Z) \leqslant \overline{Ch_{\text{top}}}(\mathbb{S}, Z, \mathbb{P}_{sym}) \leqslant h_{\text{top}}(\mathbb{S}, \mathbb{P}_{sym})
$$
for every $Z\subset X$, that these three notions coincide when $Z=X$ and that the first two notions coincide 
for any compact and $G$-invariant subset $Z$ (cf. \cite[Theorem 3.9]{JMW})
Moreover, 
although the original notion 
has been defined using open coverings, a standard argument 
shows this Carath\'eodory structure can be defined using dynamic balls (see e.g. \cite{Pe97}), an object that suits our purposes better.

\subsection{Fibered topological entropies}\label{fibentropies}

Given $\vep>0$ and $\om\in \{1,2, \dots, \kappa\}^{\mathbb Z}$ consider the metric in $X_\om$ given by
$
d_{\om}^n(x,y)=\max_{0 \leqslant j \leqslant n-1}d(g_\om^j(x),g_\om^j(y))
$
for $x,y \in X_\om$ and $n\geqslant 1$. The dynamic ball determined by $x\in X_\omega$, $n\geqslant 1$ and $\vep>0$ is the set
$$
B_{\om}(x,n,\vep)=\Big\{y\in X_\om: d_{\om}^n(x,y)<\vep\;\;\mbox{for all}\;\; 0\leqslant j<n.\Big\}
$$
A subset $K \subset X_{\om}$ is called \emph{$(\om,n,\vep)$-separated} if $d_{\om}^n(x,y) > \vep$
for any distinct $x,y \in K$.

We define the \emph{fibered topological entropy for the semigroup} on the compact metric space $X$  
as
\begin{equation}\label{eq:def-fibered-ent}
h^{path}(\mathbb S) 
=
\sup_{\om\in \{1,2,  \dots, \kappa\}^{\mathbb Z}} \, \htop(\mathcal F_\om),  F_\om) 
\end{equation}
where 
$
\htop(\mathcal F_\om) =\lim_{\vep\to 0} \limsup_{n\to\infty} \frac1n \log s(g_\om, n,\vep)
$
denotes the topological pressure of sequential dynamical system $\mathcal F_\om=(g_{\om}^j)_{j\geqslant 1}$
introduced in \cite{KS}.
Related measure theoretical notions of fibered entropy appeared in early studies by Pinsker~\cite{Pi}, Abramov and Rokhlin \cite{AR,Ro},
Ledrappier and Walters~\cite{LW}
 on the entropy of skew-products and later associated to the context of random dynamical systems (see e.g.~\cite{Ki86}).
Indeed, it is known that 
\begin{equation}\label{eq:AB}
h_\mu(F)=h_{\pi_*\mu}(\sigma)+ h_\mu(F \!\mid\! \sigma)
\end{equation}
for each $F$-invariant probability measure $\mu$, where 
the relative entropy is defined by the supremum $h_\mu(F \!\mid\! \sigma)=\sup_\xi h_\mu(F \!\mid\! \sigma,\xi)$ over finite partitions $\xi$ of 
$\{1 , 2, \dots, \kappa\}^{\mathbb Z} \times X$,  where
$$
h_\mu(F \!\mid\! \sigma,\xi)=\lim_{n\to\infty} \frac1n H_\mu\Big(\bigvee_{j=0}^{n-1} F^{-j} \xi \mid \pi^{-1}(\mathcal E)\Big)
$$
and $\mathcal E$ denotes the partition of $\{1 , 2, \dots, \kappa\}^{\mathbb Z}$ into points.

\begin{remark}
The expression ~\eqref{eq:AB} is known as the Abramov-Rokhlin formula, although it is
sometimes referred as Pinsker formula in the special case that $\sigma$ is invertible (cf. \cite{Pi,Ro}).
It was proven for skew-products with non-invertible fiber maps by Bogensch\"utz and Crauel~\cite{BC}.
\end{remark}

\begin{remark}\label{rmk:random}
Whenever $\mathbb P$ is a $\sigma$-invariant probability and $\mu$ is an $F$-invariant probability whose marginal $\pi_*\mu$ in $\Sigma_\kappa$
coincides with $\mathbb P$, the value $\int \htop(\mathcal F_\om)\, d\mathbb P$ is the (quenched) topological entropy of $\mu$
for the random dynamical system driven by $\mathbb P$ (see \cite{Ki86}). Moreover, $h_\mu(F \!\mid\! \sigma)=h_\mu(\mathcal F)$ 
where $h_\mu(\mathcal F)=\sup_{\xi} h_{\mu}(\mathcal F,\xi)$,  
the supremum is taken over all finite Borel partitions $\xi$ of $X$, 
\begin{equation*}\label{e.entropy-def-conditional2fibered}
h_{\mu}(\mathcal F,\xi)=\lim_{n\to +\infty}\frac{1}{n}\int H_{\mu_\om} \Big(\bigvee_{k=0}^{n-1}(f^k_\om)^{-1}\xi \Big) d\mathbb P(\om),
\end{equation*}
and $\mu=(\mu_\om)_\om$ is the almost everywhere defined Rokhlin disintegration of $\mu$ on its sample measures
(see e.g. ~\cite{Bo92,BC}). The variational principle for entropy states that
$$\int \htop(\mathcal F_\om)\, d\mathbb P(\om) =\sup\Big\{ h_{\mu}(\mathcal F) \colon \mu\in \mathcal M_{inv}(F), \pi_*\mu=\mathbb P \Big\}$$
(cf. \cite[Theorem~3.1]{Li01}). 
We will consider the previous equality when $\mathbb P=\pi_*\mu$ in order to compare the fibered entropies of
$F$-invariant probabilities $\mu$.
In particular, when the $F$-invariant measure $\mu$ changes then the entropies $h_\mu(F\mid\sigma)$ are not related to any fixed 
random dynamical system. 
Nevertheless, 
$
\int \htop(\mathcal F_\om) \, d\pi_{*}\mu \geqslant h_\mu(F \!\mid\! \sigma)
\;\text{for every}\, \mu\in \mathcal M_{inv}(F).
$
This, together with the ergodic decomposition theorem, ensures that 
\begin{equation}\label{eq:def-htop-conditional}
h^{path}(\mathbb S) 
	\geqslant  \sup_{\mu\in \mathcal M_{inv}(F)} \, h_\mu(F\mid\sigma)
	= \sup_{\mu\in \mathcal M_{erg}(F)} \, h_\mu(F\mid\sigma)=: \htop(F\mid\sigma),
\end{equation}
an expression which relates the fibered entropy of the semigroup with that generated by $F$-invariant probabilities. 
\end{remark}

\medskip
Let us now recall another notion of measure theoretical entropy for non-autonomous dynamical systems, which is adaptable to general probabilities,  
but for which we will be consider with respect to the sample measures of $F$-invariant probabilities. Let us be more precise.
Assume that $\mu$ is an $F$-invariant and ergodic probability, that $0<\delta<1$, $\vep>0$ and $\om\in \{1,2, \dots, \kappa\}^{\mathbb Z}$.
By compactness of $X_\om$, for each $Y\subset X_\om$ the largest cardinality $s(\om,n, \vep, Y)$ 
of an $(\om,n,\vep)$-separated subset of $Y$ is finite. Now, take $s(\om,n, \vep,\delta):=\min\{s(\om,n, \vep, Y): \mu(Y)\geqslant 1-\delta\}$ and define 
\begin{equation}\label{eq:def-d-Katok}
h_{\mu}^{\delta}(\ud\omega):= \lim_{\vep \rightarrow 0}\liminf_{n\rightarrow \infty} \frac{1}{n}\log s(\om,n, \vep, \delta),
\end{equation} 
which is clearly bounded above by $\htop(\mathcal F_\om)$.
A related notion of entropy, using a Shannon-McMillan-Breiman type formula appeared in \cite{Bis2018}.
The \emph{Katok $\delta$-entropy} of $\mu\in \mathcal M_{inv}(F)$
is defined by
$$
h_{\mu}^{\delta}({F})=\int h_{\mu}^{\delta}(\ud\omega) \, d(\pi_*\mu)(\om).
$$ 
In the special case that $\mu$ is ergodic one has that 
\begin{equation}\label{eq:comparison-entropies}
h_{\mu}^{\delta}({F})=h_{\mu}^{\delta}(\ud\omega) = h_{\mu}(\mathcal F_\om) = h_\mu(F\!\mid \!\sigma) \quad \text{for $(\pi_*\mu)$-a.e. $\om$}
\end{equation} 
coincides with the quenched topological entropy of the random dynamical system
driven by $\mathbb P=\pi_*\mu$ (cf. \cite[Theorem~A]{LiTang}). In particular we conclude that
\begin{equation}\label{eq:Pinsker-Katok}
H^{\text{Pinsker}}(\psi)=H^{Katok}(\psi),
\end{equation} 
 where the left hand-side was defined in ~\eqref{Hestrela2} and the second term is defined by
\begin{align}
	H^{Katok}(\psi):=\sup\Big\{ 
	c\geqslant 0 \colon & \forall \vep>0 \,\text{there exist} \, \mu_1,\mu_2 \in \mathcal M_{erg}(F) 
	\,\text{so that}\, \nonumber \\ 
	& h_{\mu_i}^{\delta}(F)\geqslant c-\vep  \;\text{and}\, \int\varphi_\psi\, d\mu_1< \int\varphi_\psi\, d\mu_2
	\Big\}. \label{Hestrela}
\end{align}
\medskip
In the case of non-compact subsets $Y\subset X_\om$, 
Stadlbauer, Suzuki and the second named author \cite[Appendix~A]{SSV} introduced a notion of relative entropy using Carath\'eodory structures in the realm of random dynamical systems.
Let us recall this notion.
Fix $\vep>0$ and $\om\in \Omega$ and consider the sets
$\cI_\om^n= X_\om\times\{n\}$ and $\cI_\om =X_\om \times \N$. For
every $\alpha \in \R$ and $N\geqslant 1$, define
\begin{equation}\label{eq.alpham}
m_\al(\om,f,\phi,\Lambda,\vep,N) 
	=  \inf_{\mathcal{G}_\om} \Big\{ 
	\sum_{{(x,n)} \in \mathcal{G}_\om}
  e^{-\alpha n(x)+ S_{n(x)} \phi_\om({\mathrm B_\om(x,n(x),\vep)})} \Big\}, 
\end{equation}
where the infimum is taken over all finite or countable families
$\cG$ of the form $\mathcal{G}_\om\subset \cup_{n \geqslant N}\cI_\om^n$ 
such that the collection
of sets $\{B_\om(x,n,\vep): (x,n)\in \mathcal G_\om\}$ cover $\Lambda_\om$
and such that \eqref{eq.alpham} is measurable in $\omega$.
As before, 
$
S_{n} \phi_\om({\mathrm B_\om(x,n,\vep)})
	=\sup\{ \sum_{j=0}^{n-1}\phi_{\theta^j(\om)} (g_\om^j(y)): y \in B_\om(x,n,\vep)\}.
$
Using that the previous sequence is increasing on $N$, the limit
\begin{equation}\label{eq.press1}
m_\alpha(\om, f,\phi,\Lambda, \vep)
  = \lim_{N\to+\infty}m_\al(\om,f,\phi,\Lambda,\vep,N) 
\end{equation}
does exist. If $\alpha$ is such that \eqref{eq.press1} is finite and $\beta>\al$ then   
$m_\beta(f,\phi,\Lambda,\vep,N)  \leqslant e^{-(\beta-\al)N} m_\al(f,\phi,\Lambda,\vep,N) $ tends to zero as $N\to\infty$, hence
$m_\beta(f,\phi,\Lambda, \vep)=0$. For that reason we define
$$
\pi_\phi(\om,f,\Lambda,\vep) 
	= \inf{\{\alpha: m_\al(\om,f,\phi,\Lambda, \vep)=0\}},
$$
which is decreasing on $\vep$. 
The fibered \emph{relative pressure} of the set $\La$ with respect to 
$(f,\phi)$ is defined by
\begin{equation}\label{eq.random.pressure}
 \pi_\phi(\om,f,\Lambda)
 	= \lim_{\vep \to 0} \pi_\phi(\om,f,\Lambda,\vep).
\end{equation}
The previous limit exists as $\vep \mapsto \pi_\phi(\om,f,\Lambda,\vep)$ is non-increasing.
 In the special case that $\phi\equiv 0$ we write simply 
$$
h_{\Lambda_{\om}}(\mathcal F_\omega) :=\pi_0(\om,f,\Lambda_\om).
$$
It is not hard to check that $h_{\Lambda_{\om}}(\mathcal F_\omega) \leqslant \htop(\mathcal F_\om) $ for every $\om \in \{1 , 2, \dots, \kappa\}^{\mathbb Z}$.

\medskip
In general, upper bounds are much easier to obtain as these rely on choices of certain coverings of $\Lambda$ using dynamic balls at a certain scale. 
The next result provides an entropy distribution principle for fibered sets, which is a simple modification of  \cite[Theorem 3.6]{TV03} 
and \cite[Proposition 2.6]{Tho2010}  for sequential dynamical systems and serves as a criterion to provide lower bounds on these fibered entropies.

\begin{proposition}[Fibered entropy distribution principle]\label{pro. princ dist. entropy fiber}
 Let  $\om\in \{1,2, \dots, \kappa\}^{\mathbb Z}$ and $\Lambda_\om\subset X_\om$ be an arbitrary Borel set.
 Suppose there exist $\vep>0$, $\alpha \geqslant 0$, $K_{\ud\omega}(\vep)>0$ and a sequence of probability measures $(\mu_{\om,k})_{k\geqslant 1}$ on $X_\om$ such that 
\begin{equation}\label{eq:EDP-N}
	\limsup_{k\rightarrow \infty}\mu_{\ud\omega, k}(B_{\ud\omega}(x,n,\vep))\leqslant K_{\ud\omega}(\vep)\, e^{-\alpha n}
\end{equation}
	for sufficiently large $n$ and every ball $B_{\ud\omega}(x,n,\vep)$ which has non-empty intersection with $\Lambda_{\om}$. 
	If there exists a weak$^*$ accumulation point $\nu_{\ud\omega}$ of the probability measures $(\mu_{\ud\omega,k})_{k\geqslant 1}$ so that 
	$\nu_{\ud\omega}(\Lambda_{\ud\omega})>0$ then 
	$h_{\Lambda_{\om}}(\mathcal F_\omega)\geqslant \alpha$.
\end{proposition}
 
\begin{proof}
	Let $\vep>0$, $\alpha \geqslant 0$, $K_{\ud\omega}(\vep)>0$ be as above, and assume that $(\mu_{\ud\omega,k_j})_{j\geqslant 1}$ is weak$^*$ convergent 
	to $\nu_{\om}$. Assume also that  $\nu_{\ud\omega}(\Lambda_{\ud\omega})>0$. Fix $N_\om\geqslant 1$ large so that ~\eqref{eq:EDP-N} holds for all $n\geqslant N_\om$.
	Consider an arbitrary countable family $\mathcal{G}_\om\subset \cup_{n \geqslant N_\om}\cI_\om^n$ such that the collection
$\{B_\om(x,n(x),\vep): (x,n(x))\in \mathcal G_\om\}$ covers $\Lambda_\om$.	
Extracting a subcollection of sets, if necessary, one may assume without loss of generality that 
$B_{\om}(x,n(x),\vep)\cap \Lambda_{\om}\neq \emptyset$ for every $(x,n(x))\in \mathcal G_\om$.
Therefore, by the Portmanteau theorem on the characterization of weak$^*$ convergence one concludes that 
\begin{align*}	
\sum_{{(x,n)} \in \mathcal{G}_\om} e^{-\alpha n(x)}	
	& \geqslant  K_{\om}(\vep)^{-1} \sum_i \limsup_{k\rightarrow \infty}\mu_{\om,k}(B_{\om}(x, n(x),\vep)) \\
		& \geqslant   K_{\omega}(\vep)^{-1} \sum_i \nu_{\om}(B_{\om} (x, n(x),\vep) )\\
		&\geqslant K_{\om}(\vep)^{-1} \nu(\Lambda_{\om})>0		
\end{align*}
As $\mathcal G_\om$ was chosen arbitrary this implies that $m_\al(\om,f,0,\Lambda_\om,\vep,N) \geqslant K_{\om}(\vep)^{-1} \nu(\Lambda_{\om})>0$
for every $N\geqslant N_\om$ and, consequently,
$h_{\Lambda_{\om}}(\mathcal F_\omega)\geqslant \alpha$.
\end{proof}

\begin{remark}\label{rmk:EDP-RDS}
We observe that 
Proposition~\ref{pro. princ dist. entropy fiber} is of independent interest in the special case of random dynamical systems, modeled by the skew-product $F$. Indeed, 
it allows to estimate the entropy of equivariant sets contained in fibers $X_\om$ associated to a $\mathbb P$-generic point $\om$ 
provided the existence of $F$-invariant probabilities $(\mu_k)_{k\geqslant 1}$ whose disintegrations $(\mu_{\om,k})_{k\geqslant 1}$, defined almost everywhere 
by Rokhlin theorem as a measurable family $\om \mapsto \mu_{\om,k}$ of probabilities so that $\supp \mu_{\om,k}\subset X_\om$ and 
$\mu=\int \mu_{\om,k} \, d\mathbb P(\om)$, satisfy ~\eqref{eq:EDP-N} at a certain scale $\vep$ and every large $n$.
\end{remark}

\subsection{Some examples}

In the remainder of this section we compare the quantities $h_*(\varphi_\psi)$, $H^{\text{KS}}(\psi)$ $H^{\text{Pinsker}}(\psi)$ and 	$H^{\sigma}(\psi)$ defined by 
~\eqref{hestrela}-\eqref{Hestrela3}, using a couple of examples. For the sake of simplicity we shall consider semigroups generated by certain continuous maps on the interval $[0,1]$. 
\begin{example}\label{ex:comparison1}
Let $G$ be the semigroup action generated by two continuous and Markov interval maps $f_1, f_2 : [0,1] \to[0,1]$ as described in the following figure.
Let $\psi: [0,1] \to\mathbb R$ be a continuous observable.
\begin{figure}[htb]
	 \includegraphics[scale=.5]{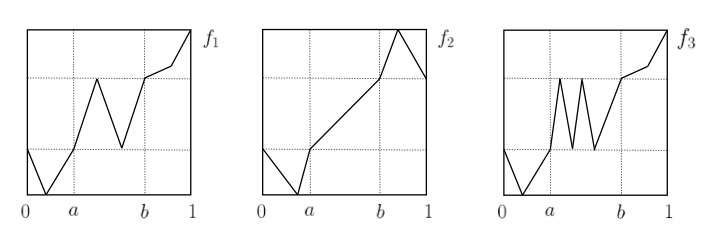}
 \caption{Generators $G_1=\{f_1,f_2\}$}
\end{figure}
Assume further that $\psi$ is not cohomologous to a constant neither with respect to $f_1$ or $f_2$ 
(in particular $\varphi_\psi$ is not cohomologous to a constant with respect to the step skew-product $F$ on $\Sigma_2^+\times [0,1]$ determined by $f_1$ and $f_2$. Assume, additionally, that $\psi\mid_{[a,1]}\equiv 1$. 

Firstly we compute the entropy of the skew-product $F$. 
The map $F\mid_{\Sigma_2^+ \times [0,a]}$ is a distance expanding map and is semiconjugate to a full shift with $4$ symbols, hence 
$\htop(F\mid_{\Sigma_2^+ \times [0,a]})=\log 4$. 
Proceeding with a similar description of the three transitive components and Markov property of the skew-product it is easy to check that
$\htop(F)=\log 4$.  
As $F\mid_{\Sigma_2^+ \times [0,a]}$ satisfies the specification property then the space of ergodic probability measures is entropy dense, hence
the assumptions on $\psi$ guarantee that there exist $F$-invariant and ergodic probability measures 
$\mu_1,\mu_2$ with entropy arbitrarily close to $\log 4$ so that $\int \varphi_\psi d\mu_1 \neq \int \varphi_\psi d\mu_2$, hence 
$$
\htop(F)=\log 4=h_*(\varphi_\psi).
$$

Secondly, let us observe now the non-autonomous dynamical systems generated by concatenations of maps $f_1$ and $f_2$. The restriction of each of
these maps to any of its invariant intervals $\{[0,a], [a,b], [b,1] \}$ defines a continuous, piecewise affine distance expanding and Markov map with degree at most three. Hence,  
$$
\htop(\mathcal F_\om) \leqslant \log 3  \quad \text{for every } \om\in \Sigma_2^+.
$$
Actually,
$
H^{KS}(\psi) = \log 3,
$
where the maximum is attained for any $\om$ in the pre-orbit of the fixed point $(2,2,2, \dots)\in \Sigma_2^+$.
\medskip

Finally, let us now draw our attention to the Pinsker entropy and the contribution on the shift of measures with large fibered entropy. Let $\mu_1, \mu_2$ be two
$F$-invariant and ergodic probabilities so that $\int\varphi_\psi\, d\mu_1< \int\varphi_\psi\, d\mu_2$. 
As $\psi$ is constant on the interval $[a,b]$, both probabilities cannot be simultaneously supported on the invariant domain $\Sigma_2^+\times [a,b]$ and, using ~\eqref{eq:comparison-entropies}, we deduce that 
$h_{\mu_i}(F\mid\sigma) \leqslant \log 2$ for either $i=1$ or $i=2$. We claim that $H^{\text{Pinsker}}(\psi)=\log 2$. Indeed, this follows as a standard consequence of the specification property for $F\mid_{\Sigma_2^+\times [0,a]}$ (which implies that ergodic measures are entropy dense) combined with the fact that the unique maximal entropy $\mu^1_{max}$ measure for $F\mid_{\Sigma_2^+\times [0,a]}$ projects 
over the Bernoulli measure $(\frac12,\frac12)^{\mathbb N}$ on $\Sigma_2^+$ and, consequently,
$
h_{\mu^1_{max}}(F\mid\sigma)=h_{\mu^1_{max}}(F) - h_{\pi_*\mu^1_{max}}(\sigma)=\log4 -\log 2=\log 2.
$
Altogether,
$$
H^{\text{Pinsker}}(\psi)= \log 2  <  \log 3 = H^{KS}(\psi).
$$
Furthermore, assume that the unique maximal entropy $\mu^2_{max}$ measure for $F\mid_{\Sigma_2^+\times [a,b]}$ is such that  
$\int\varphi_\psi\, d\mu^1_{max}< \int\varphi_\psi\, d\mu^2_{max}$ (otherwise use entropy denseness of ergodic measures to produce 
an ergodic measure arbitrarily close to $\mu^2_{max}$ with fibered entropy arbitrary close to $\log 3$). As this measure projects over
the Bernoulli measure $(\frac12,\frac12)^{\mathbb N}$ as well we conclude that 
$$
H^{\text{Pinsker}}_\sigma(\psi)=\log 2=\htop(\sigma).
$$
\end{example}

\medskip
\begin{example}\label{ex:comparison2}
Let $G$ be the semigroup generated by the continuous and Markov interval maps $f_2,f_3: [0,1] \to [0,1]$ described in the previous figure, and let 
$\hat F:\Sigma_2^+\times [0,1] \to \Sigma_2^+\times [0,1]$ denote the step skew-product determined by them. 
Let $\psi: [0,1] \to\mathbb R$ be a continuous observable so that  $\psi\mid_{[a,b]}\equiv 1$ and that
is not cohomologous to a constant neither with respect to $f_2$ or $f_3$ (in particular $\varphi_\psi$ is not cohomologous to a constant with respect to $\hat F$. 
It is easy to check that 
$\htop(\hat F)=\log 6$. Moreover, as $\psi$ is constant over the interval $[a,b]$, given two $\hat F$-invariant and ergodic probabilities $\mu_1, \mu_2$ so that 
$\int \varphi_\psi d\mu_1 \neq \int \varphi_\psi d\mu_2$ at least one of them is supported either on $\Sigma_2^+\times [0,a]$ or $\Sigma_2^+\times [b,1]$.
Using the specification property for the restriction of $\hat F$ to any of these invariant domains, one gets that
$$
\htop(\hat F)=\log 6 > \log 4 \geqslant h_*(\varphi_\psi).
$$
and $H^{\text{Pinsker}}_\sigma(\psi)=\log 2=\htop(\sigma).$
\end{example}

\medskip
Finally, it is worthwhile to point out that the entropy estimates in Examples~\ref{ex:comparison1} and ~\ref{ex:comparison2} involve counting elements in a 
Markov partition and the fact that all compositions are also Markov. Hence, it is likely that similar examples can be produced 
involving higher dimensional piecewise expanding maps or even Axiom A diffeomorphisms on a compact Riemannian manifold.

\color{black}

\section{The frequent hitting times property}\label{sec:prelim}
In this section we not only provide sufficient conditions to ensure that
single dynamics, semigroup actions and skew-products satisfy the notion of frequent hitting times,
as we discuss the relation between it, some weak versions of specification and hyperbolicity.

\subsection{Sufficient condition for the frequent hitting times property}\label{suff}

Our aim here is to provide a checkable condition which implies on the frequent hitting times property, and to provide 
several examples that illustrate its applicability.
Assume that the semigroup action generated by a collection $G_1=\{id,f_1, f_2, \dots, f_\kappa\}$ of continuous maps on a compact metric space $X$ satisfying the following 
strong transitivity 
property: for every $\vep>0$ there exists $K(\vep)\geqslant 1$ such that: 
\begin{equation}\label{eq:sufficient-condition}\tag{C}
\forall x\in X\, \exists \, {\underline{\omega}}=\om_x \in \Sigma_\kappa 
	\qquad \bigcup_{j=0}^{K(\vep)} \, f_{\underline{\omega}}^j(B(x,\vep))=X.
\end{equation}
In the special case that the continuous maps $f_i$ are actually homeomorphisms 
it will be often more convenient to express the previous condition (C) using their inverses:
for every $\vep>0$ there exists $K(\vep)\geqslant 1$, and for each $x\in X$ there exists ${\underline{\omega}}\in \Sigma_\kappa$
such that 
\begin{equation}\label{eq:sufficient-condition-pre-images}\tag{C'}
\forall x\in X\, \exists\, {\underline{\omega}}=\om_x\, \in \Sigma_\kappa 
	\qquad \bigcup_{j=0}^{K(\vep)} \, f_{{\underline{\omega}}_0}^{-1}\circ  \dots \circ  f_{{\underline{\omega}}_1}^{-1} \circ f_{{\underline{\omega}}_{K(\vep)}}^{-1}(B(x,\vep))=X.
\end{equation} 
Conditions (C) and (C') are clearly equivalent for semigroup actions generated by isometries. Actually, we have the following:

\begin{lemma}\label{prop:sufficiency}
Let $X$ be a compact metric space and consider the semigroup action $\mathbb S$ generated by a collection of 
homeomorphisms $G_1=\{id,f_1, f_2, \dots, f_\kappa\}$ satisfying property (C). 
If the smallest sub-semigroup action satisfying property (C) is 
generated by a collection $\widetilde{ G_1}\subset G_1$ of isometries 
then the semigroup action $\mathbb S$ 
has {frequent hitting times}.
\end{lemma}

\begin{proof}
Assume there exists a collection $\widetilde{ G_1}\subset G_1$ of isometries generating a sub-semigroup action
satisfying (C), and for any $\varepsilon>0$ let $K(\varepsilon)\geqslant 1$ be given according to this property.
 We claim that the semigroup action generated by $G_1$ has {frequent hitting times}.
Fix arbitrary balls $(B_1,B_2)$ with $B_1=B(x_1,\vep)$ and $B_2$ of radius $0<r_2<\vep/2$.
If $\hat B_1 \subset B_1$ is the ball of radius $\vep/3$ centered at $x_1$
then there exists $0\leqslant p \leqslant K(\vep/3)$, a finite word ${\underline{\omega}}\in \{1,2, \dots,\kappa\}^p$ and elements $f_{\omega_i}\in \widetilde{G_1}$ for every
$0\leqslant i \leqslant p$ so that $f_{\underline{\omega}}^p(\hat B_1) \cap B_2 \neq\emptyset$. 
By the hypothesis, the map
$f_{\underline{\omega}}^p$ is an isometry. Therefore, if
$r_2<\vep/3$ then $f_{\underline{\omega}}^n(B_1) \supset B_2$. Otherwise, $\vep/3\leqslant r_2 <\vep/2$ and $f_{\underline{\omega}}^p(B_1) \cap B_2$ contains a ball of radius 
$\frac\vep3  = \frac\vep2 \frac23  > r_2 \frac23 > \frac{r_2}2$. 
Altogether, this shows that $G_1$ satisfies the frequent hitting time property with transition time 
$K(\vep/3)$, and proves the lemma.
\end{proof}

In other words, the previous lemma ensures that the frequent hitting times property 
hold for semigroup actions which contain some minimal action by isometries.
Recall that a homeomorphism $f: X\to X$ is \emph{minimal} if all points have a dense orbit, and  
a semigroup action generated by a collection of homeomorphisms $G_1$ is \emph{minimal} if for every $x\in X$ there
exists ${\underline{\omega}}\in \Sigma_\kappa$ so that $X\,=\, \overline{\{ f_{\underline{\omega}}^n(x) \colon n\geqslant 0 \}}$
(cf. \cite{HN}).
The previous lemma can be reformulated to provide some classes of examples, defined in terms of minimality, which satisfy the frequent hitting times property.

\begin{corollary}\label{cor:rotations}
Let $X$ be a compact metric space. The semigroup action generated by the collection $G_1$ of continuous maps 
has frequent hitting times in the following cases:
\begin{enumerate}
\item $G_1=\{f\}$, where $f: X\to X$ is a minimal isometry;
\item $G_1=\{f\}$, where $f: \mathbf S^1 \to \mathbf S^1$ is an irrational rotation; 
\item $G_1=\{id,f_1, f_2, \dots, f_\kappa\}$ and there exists a sub-collection 
	$\widetilde{G_1}\subset G_1$ formed by isometries which generate a minimal semigroup action. .
\end{enumerate}
\end{corollary}

 \begin{proof} 
 Items (1) and (2) are immediate consequences of item (3), hence we are left to prove the latter.
Moreover, using Lemma~\ref{prop:sufficiency} it is enough to show that
the collection $\widetilde{G_1}$ generates a sub-semigroup satisfying property (C).
The argument is identical to the one used in the proof of \cite[Lemma~5.1]{Sun0} adapted to the semigroup actions context, which we include
for completeness. Indeed, by minimality of the sub-semigroup action, 
for any $x\in X$ there exists ${\underline{\omega}}\in \Sigma_\kappa$ so that $f_{\omega_i}\in \widetilde{G_1}$ for every $i\geqslant 0$
and 
$
X\,=\, \overline{\{ f_{\underline{\omega}}^n(x) \colon n\geqslant 0 \}}.
$
Thus, for every $\vep>0$ we obtain that 
$
\bigcup_{n\geqslant 0} (f_{\underline{\omega}}^n)^{-1} (B(x,\vep)) \,=\, X.
$
Moreover, by compactness of $X$, 
there exists $K_x(\vep)\geqslant 1$ so that
$
\bigcup_{n=0}^{K_x(\vep)} (f_{\underline{\omega}}^n)^{-1} (B(x,\vep)) \,=\, X.
$
Now, using the uniform continuity of the finite collection of maps $f_{\omega_i}$, there exists $\delta_x>0$ so that 
$$
\bigcup_{n=0}^{K_x(\vep)} (f_{\underline{\omega}}^n)^{-1} (B(y,\vep)) \,=\, X
	\quad\text{for every $y\in B(x,\delta_x)$}.
$$
Extracting a finite sub cover $\{B(x_i,\delta_i): 1\leqslant i \leqslant s\}\subset \{B(x,\delta_x):x\in X\}$, we conclude that
(C') holds with
$K(\vep)=\max_{1\leqslant i \leqslant s} K_{x_i}(\vep)$.
Furthermore, 
as each elements $f_{\omega_i}$ is an isometry, we conclude that property (C) holds for the sub-semigroup action generated by $\widetilde{G_1}$,
which completes the proof of the corollary.
 \end{proof}
 
\begin{remark}\label{rmk:irrational-rot}
In the special case that $f: \mathbf S^1 \to \mathbf S^1$ is an irrational rotation  
it is easy to check from the proof of Lemma~\ref{prop:sufficiency} and simple estimates on the minimal covering time  
(cf. \cite{BTV}) that the frequent hitting times property holds with 
$K(\vep)=[\frac3\vep]+1$ .
In particular, the transition time may be exponentially large on balls with exponentially small radius.
\end{remark}

\begin{remark}\label{rmk:equicontinuous}
In the case of a single homeomorphism, one can replace the assumption that the dynamics $f$ is an isometry 
by equicontinuity of the family $\{f^n\}_{n\in \mathbb Z}$, as in this context $d_f(x,y):=\sup_{n\in \mathbb Z} d(f^n(x),f^n(y))$ defines a metric
on $X$ with respect to which $f$ becomes an isometry.
\end{remark}

\begin{remark}\label{rmk:def.hitting}
The frequent hitting times property can be rephrased as follows: for any $\varepsilon>0$ there exists $K(\varepsilon)\geqslant 1$ 
so that for any balls $(B_1,B_2)$ such that $|B_1|=\vep$ there exists $0\leqslant p \leqslant K(\vep)$, $\om\in \{1,2,\dots, \kappa\}^p$ and a ball $B_2'\subset B_2$ of radius 
$r\geqslant \min\{|B_1|/4,|B_2|/2\}$ so that $g_\om^p(B_1) \supset B_2'$. Indeed, if $|B_2|>\vep/2$ and one takes $\hat B_2\subset B_2$ a ball of radius $|B_1|/2$
then there exists $B_2'\subset \hat B_2$ of radius 
$|\hat B_2|/2=|B_1|/4$ so that $g_\om^p(B_1) \supset B_2'$.
\end{remark}

Let us provide a simple example which fits in the context of item (3) in the corollary.

\begin{example}\label{ex:T3}
Consider the torus $\mathbb T^3=\mathbb R^3/\mathbb Z^3$, let $v_1,v_2,v_3\in \mathbb R^3$ be a base of $\mathbb R^3$ such that the components of each vector $v_i$ are rationally independent. For each $i=1,2,3$ consider the translation
$f_i(x)=x+v_i (mod \, \mathbb Z^3)$, $x\in \mathbb T^3$, which is clearly an isometry. 
Even though each individual translation may fail to be transitive, the choice of the vectors $v_i$ guarantees that  the semigroup action generated by $\{f_1,f_2, f_3\}$ is minimal. According to Corollary~\ref{cor:rotations},
the semigroup action generated by $G_1=\{ f_1,f_2,f_3\} $ has frequent hitting times. 
\end{example}

\subsection{Projective linear maps}\label{ex:projection}

In opposition to the case of the torus, where a single irrational translation have dense orbits, minimal semigroup action in arbitrary
manifolds often require some minimum number of generators (see \cite{HN} and references therein for a discussion on this topic). 
This is precisely the case for continuous maps on spheres, which appear naturally in applications to linear cocycles as these are 
double covering of real projective Euclidean spaces.

\begin{example}\label{ex:S2}
Let ${\mathbf S}^2 \subset \mathbb R^3$ be the unit sphere an
d, given $\alpha, \beta\in \mathbb R\setminus \mathbb Q$ consider the semigroup action on ${\mathbf S}^2$ generated by the collection $G_1=\{f_0,f_1\}$ 
of elements in $SO(3)$ defined as
$$
f_1= 
\left(
\begin{array}{ccc}
\cos \alpha & -\sin \alpha & 0 \\
\sin \alpha & \cos \alpha  & 0 \\
0 & 0 & 1
\end{array}
\right)
\quad\text{and}\quad
f_2= 
\left(
\begin{array}{ccc}
1 & 0 & 0 \\
0 & \cos \beta & -\sin \beta \\
0 & \sin \beta & \cos \beta   
\end{array}
\right)
$$
The maps $f_1$ and $f_2$ are irrational rotations around the $z$-axis and the $x$-axis, respectively. 
It is clear from the use of polar coordinates 
and the fact that the rotations are irrational that the semigroup action generated by the collection $G_1=\{f_1,f_2\}$ of isometries is minimal. 
Corollary~\ref{cor:rotations} implies that such semigroup action has frequent hitting times. 
\end{example}

\begin{example}\label{ex:symplectic}
Let $Sp(2\ell,\mathbb R)$ denote the Lie group of $2\ell \times 2\ell$-symplectic matrices. 
Assume that $A\in sp(4,\mathbb{R})$ is a $4\times 4$ symplectic matrix corresponding to a {generic center}, i.e. 
$A$ has two pairs of non-real, complex conjugate eigenvalues with norm one.
By continuity of the isolated eigenvalues, 
there exists an open neighborhood $\mathscr U\subset sp(4,\mathbb{R})$ of the matrix $A$ 
formed by symplectic matrices having distinct unitary complex eigenvalues and, up to conjugacy by symplectic matrices, each  $B\in \mathscr U$
is determined by their eigenvalues
in the canonical symplectic basis 
where the $2$-form can be written as $e_1\wedge \hat{e}_1+e_2\wedge \hat{e}_2$, 
one can write
$$
B
	= \begin{pmatrix}a&0&-b&0\\0&c&0&-d\\b&0&a&0\\0&d&0&c\end{pmatrix}
$$
where $a^2+b^2=1$, $c^2+d^2=1$ and $b,d\neq 0$. Up to conjugacy by a symplectic change of basis matrix, the previous matrices are in one-to-one correspondance with the previous parameterization of its eigenvalues $a\pm {\bf i}\, b$ and $c\pm {\bf i}\, d$. The projectivization of $B$ is
an isometry. Moreover, there exists a Baire generic and full Haar measure subset $\mathcal R\subset \mathscr U$
so that the eigenvalues of each $B\in \mathcal R$ do not satisfy any of the equations 
$
(a+bi)^m \, (c+di)^n=1,
$
with $m,n\in \mathbb Z$ and, consequently,
the projectivization of $B$ is a minimal isometry on $\mathbf P\mathbb R^4$.
The previous reasoning can also be action on the central direction of to partially hyperbolic $sl(2\ell,\mathbb R)$ matrices ($\ell>2$)
having a generic center in its 4-dimensional center subbundle. 
\end{example}

\subsection{Relation with other gluing orbit properties and hyperbolicity}\label{sec:vs-gluing}

Here we provide some examples to illustrate that the frequent hitting times property is significantly different from
the gluing orbit and specification properties. Indeed, The specification property is frequently associated to hyperbolic dynamics and
the gluing orbit property can arise in both hyperbolic, partially hyperbolic and non-uniformly hyperbolic dynamics 
(we refer the reader \cite{BV,BTV2} for definitions and precise statements). However, as observed in Subsection~\ref{suff}, the frequent hitting times property
is guaranteed for minimal semigroup actions by isometries.
The next examples illustrate that it may occur even if all others notions fail.

We say that the semigroup action generated by $G_1=\{id,f_1,...,f_\kappa\}$ satisfies the \emph{gluing orbit property} if for any 
$\varepsilon>0$ there exists $K(\vep)\geqslant 1$ so that for any points $(x_1,{\underline{\omega}}_1), \dots (x_r,{\underline{\omega}}_r)$ in $X\times \Sigma_\kappa$ 
there exist $x\in X$, $0\leqslant p_i\leqslant K(\vep)$  and $\theta_i \in \{1, 2, \dots, p\}^{p_i}$ ($i=1\dots r-1$) so that
$d(f^j_{{\underline{\omega}}_1}(x) ,f^j_{{\underline{\omega}}_1}(x_1)) < \varepsilon$ for every $j=1,...,n_1$ and
$$
d( f^j_{{\underline{\omega}}_\ell} \dots f_{\theta_2}^{p_2} {f}^{n_2}_{{\underline{\omega}}_2}(x) f_{\theta_1}^{p_1} {f}^{n_1}_{{\underline{\omega}}_1}(x),\, f^j_{{\underline{\omega}}_\ell}(x_j))<\varepsilon
$$
for every $j=2,...,n_\ell$ and $\ell=1,...,k$.
The first example  
shows that the frequent hitting times property 
may fail among hyperbolic dynamics.

\begin{example}\label{ex:Anosov}
Consider the linear Anosov automorphism $f_A$ on $\mathbb T^2=\mathbb R^2/\mathbb Z^2$ induced by $$A=\left(\begin{array}{cc}
2& 1 \\ 1 & 1\end{array}\right)$$
and let $q: \mathbb R^2 \to \mathbb T^2$ be the quotient map. In particular, $f_A$ satisfies the gluing orbit property  \cite{Bo71}.
Note that $p=(0,0)$ is a hyperbolic fixed point for $f$, and that the $A$-invariant hyperbolic splitting $\mathbb R^2=E^u\oplus E^s$ 
is such that the invariant manifolds at $p$ are simply $W^*(p)=q(E^*)$ where $*\in \{s,u\}$. Let $\lambda>1$ be the leading eigenvalue for
$A$ (and, since $A$ is volume preserving, $\lambda^{-1}$ is the contracting eigenvalue for $A$). 
Given $\vep>0$, the set $f_A^n(B(p,\vep))$ is the projection by $q$ of an ellipsis in $\mathbb R^2$ whose larger and smaller axis have length
$\lambda^{n}\vep$ and $\lambda^{-n}\vep$, respectively, wrapped in the torus. 
In consequence, the inner diameter of the set $f_A^n(B(p,\vep))$ decreases exponentially fast with $n$, 
hence the frequent hitting times property does not hold.
\end{example}

\begin{example}\label{MS}
Consider the semigroup action $\mathbb S$ generated by the homeomorphisms $G_1=\{id, f_1,f_2\}$ on the circle $\mathbf S^1=\mathbb R/\mathbb Z$ where $f_0$  is a north-pole south-pole diffeomorphism and $f_1$ a irrational rotation (see Figure~5.1 below). 
By \cite{BTV} and Corollary~\ref{cor:rotations}, the map $f_1$ satisfies both the gluing orbit and the frequent hitting times properties. In particular the semigroup action $\mathbb S$ has frequent hitting times. 
\begin{figure}[h]
\begin{center}
        \includegraphics[scale = 0.35]{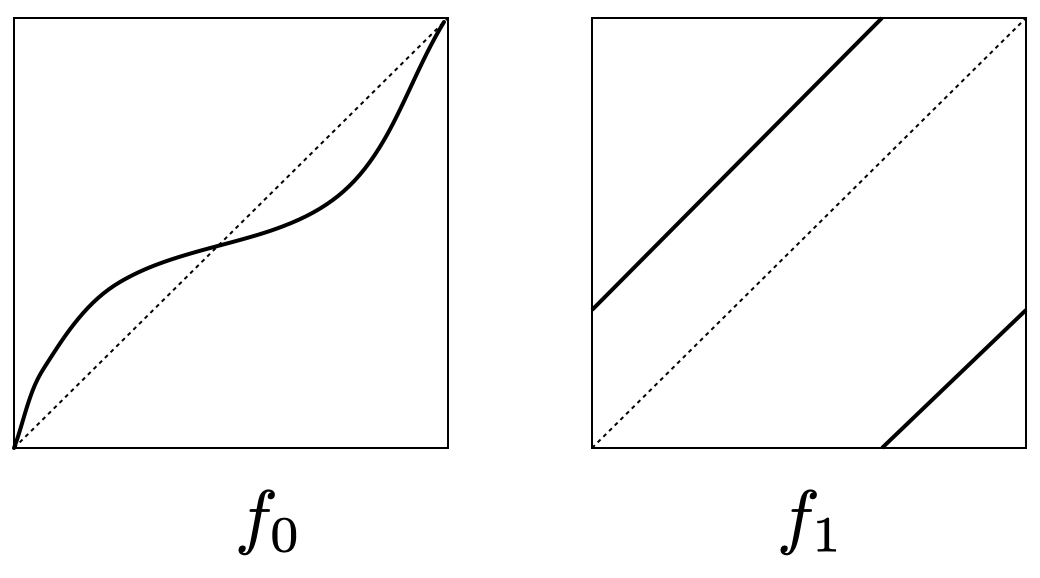}
	\vspace{-.2cm}
        \caption{Iterated function system with two generators $f_0,f_1$ on the circle $\mathbf S^1$: a north-pole south-pole homeomorphism and an irrational rotation $f_1$}
\end{center}
\end{figure}
We proceed to show that such semigroup action does not satisfy a gluing orbit property nor any non-uniform specification property.
First, since $f_1$ is an irrational rotation then the orbit of $\frac12$ cannot
contain $0$ (otherwise a composition of these rotations would be a rotation of rational angle $1/2$). 
For any $\delta>0$ let $k(\delta)\geqslant 1$ be the largest integer so that $d(f_1^j(\frac12),1)>2\delta$ for every 
$0\leqslant j \leqslant k(\delta)$. Clearly $k(\delta) \to +\infty$ as $\delta\to 0$.
Fix $\vep>0$, $x_1=\frac12$ and $x_2=0$. By the uniformly contracting (resp. expanding) behavior in a neighborhood of $x_1$ (resp. $x_2$), 
there exists $\lambda\in (0,1)$ so that
$f_0^n(B_{f_0}(x_1,n,\vep)) \subset B(x_1, \lambda^{n}\vep)$ and $B_{f_0}(x_2,n,\vep) \subset B(x_2, \lambda^{n}\vep)$ for every small $\vep>0$ and every $n\geqslant 1$.  Observe that the contracting behavior in a neighborhood of $x_1$ guarantees that $f_0^{(j+1)+n}(B_{f_0}(x_1,n,\vep))\subset f_0^{j+n}(B_{f_0}(x_1,n,\vep))$ and $f_0^{j+n}(B_{f_0}(x_1,n,\vep)) \cap B_{f_0}(x_2,n,\vep)=\emptyset$, for every $j\geqslant 0$.
Hence, any possible orbit between the sets $B_{f_0}(x_1,n,\vep)$ and $B_{f_0}(x_2,n,\vep)$ need to arise from a concatenation of the map $f_1$
for a certain number of times. By construction,
$$
f_1^{j} f_0^{n}(B_{f_0}(x_1,n,\vep)) \cap B_{f_0}(x_2,n,\vep) =\emptyset
	\quad\text{for every $0\leqslant j \leqslant k(\lambda^n \vep)$}
$$
and, consequently, not only any such transition time depends on the size $n$ of the orbit of $x_1$ as it grows exponentially fast with $n$. Indeed, 
by \cite{BTV}
the minimal integer $\lfloor \frac1\delta\rfloor -1 \leqslant K_0(\delta)\leqslant \lfloor \frac1\delta\rfloor +1$ 
so that $\bigcup_{j=0}^{K_0(\delta)} f_1^j(B(z,\delta))=\mathbf S^1$ for every $z\in \mathbf S^1$
is bounded above by $2 K_1(\delta)$ where $K_1(\delta)\geqslant 1$ is the smallest integer so that $f_1^j(B(z,\delta)) \cap B(z,\delta)=\emptyset$
for every $z\in \mathbf S^1$ (as $f_1$ is a minimal isometry the integers $K_0(\delta), K_1(\delta)$ are independent of the point $z$).
As $0$ and $\frac12$ are in diametrically opposite positions one has that $k(\delta)\geqslant \lfloor \frac12 K_1(\frac\delta2)\rfloor \geqslant 
\lfloor \frac1{\delta}\rfloor -1$ for every small $\delta>0$, which ensures that 
$\liminf_{n\to\infty} \frac1n \log k(\lambda^n \vep)\geqslant -\log \lambda >0$. In particular the semigroup does not satisfy a non-uniform specification property, 
ie, the transition times cannot be taken assumed to grow in a sub-linear way in comparison to the size of shadowed orbits.
\end{example}

\begin{remark}
The previous example has similar implications for the corresponding skew-product maps. Indeed, assume that 
$F: \mathbb T^2\times \mathbf S^1 \to \mathbb T^2\times \mathbf S^1$ is a 
$C^1$ partially hyperbolic  diffeomorphism of the form $F(x,y)=(A(x),g(x,y))$ where $A: \mathbb T^2 \to 
\mathbb T^2$ is  an Anosov diffeomorphism and that there exists $p_j \in Per(A)$ of period $\pi(p_j)\geqslant 1$
such that $F^{\pi(p_j)}(p_j,\cdot):\mathbf S^1 \to \mathbf S^1$ coincides with the map $f_j$ ($j=0,1$). 
Since $F$ has periodic points of different index then it does not satisfy the gluing orbit property \cite{BTV2}. 
While non-uniform specification holds for typical points of hyperbolic measures \cite{OliTian}, the
previous example suggests that a global non-uniform specification property fails dramatically.
\end{remark}

In the remainder of this subsection we discuss the strength of the frequent hitting times in the process of shadowing finite pieces of orbits, with the specification, gluing orbit property and non-uniform specification. In fact, instead of being bounded, or even having sublinear growth with respect to the size of the finite pieces of orbits, transition times may be much larger than the size of the finite pieces of orbits: given a finite collection $(x_i,n_i)_{1\leqslant i \leqslant \ell}$, the frequent hitting times property guarantees that that there exist $p_i\geqslant 1$, depending on $\vep$, ($1\leqslant i \leqslant \ell-1$) and $x\in X$ so that 
\begin{equation}
0 < n_1 \leqslant n_1 + p_1 \leqslant n_2 + (n_1 + p_1) \leqslant \dots \leqslant n_\ell + \sum_{i=0}^{\ell-1} (n_i+p_i)
\end{equation}
where 
\begin{enumerate}
\item during the shadowing intervals $[\sum_{i=0}^{j-1} (n_i+p_i), n_j+\sum_{i=0}^{j-1} (n_i+p_i)]$ the orbit of $x$ remains
	$\vep$-close to the orbit of the point $x_j$
\item the transition interval $[n_j+\sum_{i=0}^{j-1} (n_i+p_i), n_j+p_j+\sum_{i=0}^{j-1} (n_i+p_i)]$ has size $p_j$
\item $p_j$ may have exponential growth with respect to $n_j$.
\end{enumerate}
\begin{figure}[htb]\label{fig0}
\begin{center}
  \includegraphics[scale=.72]{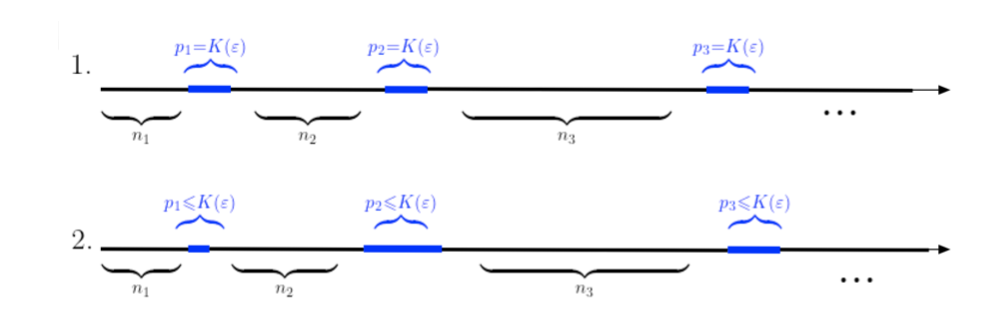}
    \includegraphics[scale=.72]{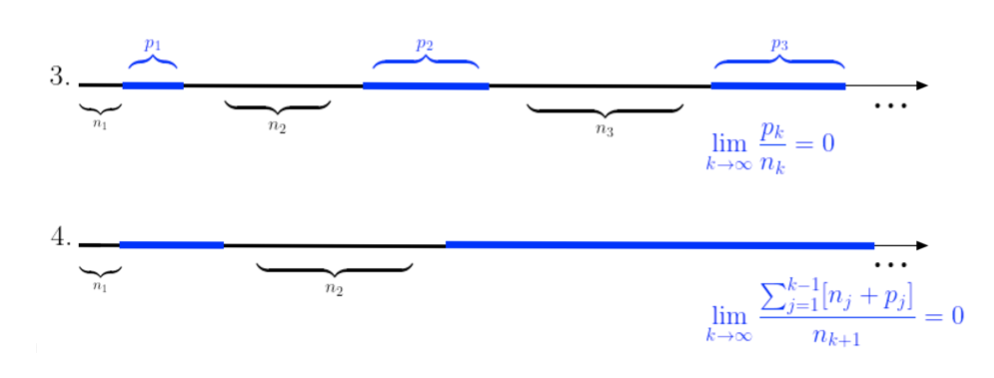}
\caption{Time lags and oscillating behavior between shadowing finite pieces of orbits under: 
(1) specification, (2) gluing orbit property, (3) non-uniform specification, and (4) frequent hitting time property.}
\label{figure}
\end{center}
\end{figure}
Yet, while $p_i$ may be much larger than $n_i$ the crucial fact is that it has memory loss,
meaning it depends exlusively on it and not on the size of future pieces of orbits. 
This makes possible to choose an adequate size 
$n_{j+1}$ of an orbit and bookkeeping the previous transition times. 
The effect obtained in the construction of points with oscillating time averages is illustrated in Figures 5.2 and 5.3.
We refer the reader to Section~\ref{teoC} and expressions \eqref{def:nk} for more details.

\begin{figure}[hbt]\label{fig01}
\begin{center}
  \includegraphics[scale=.25]{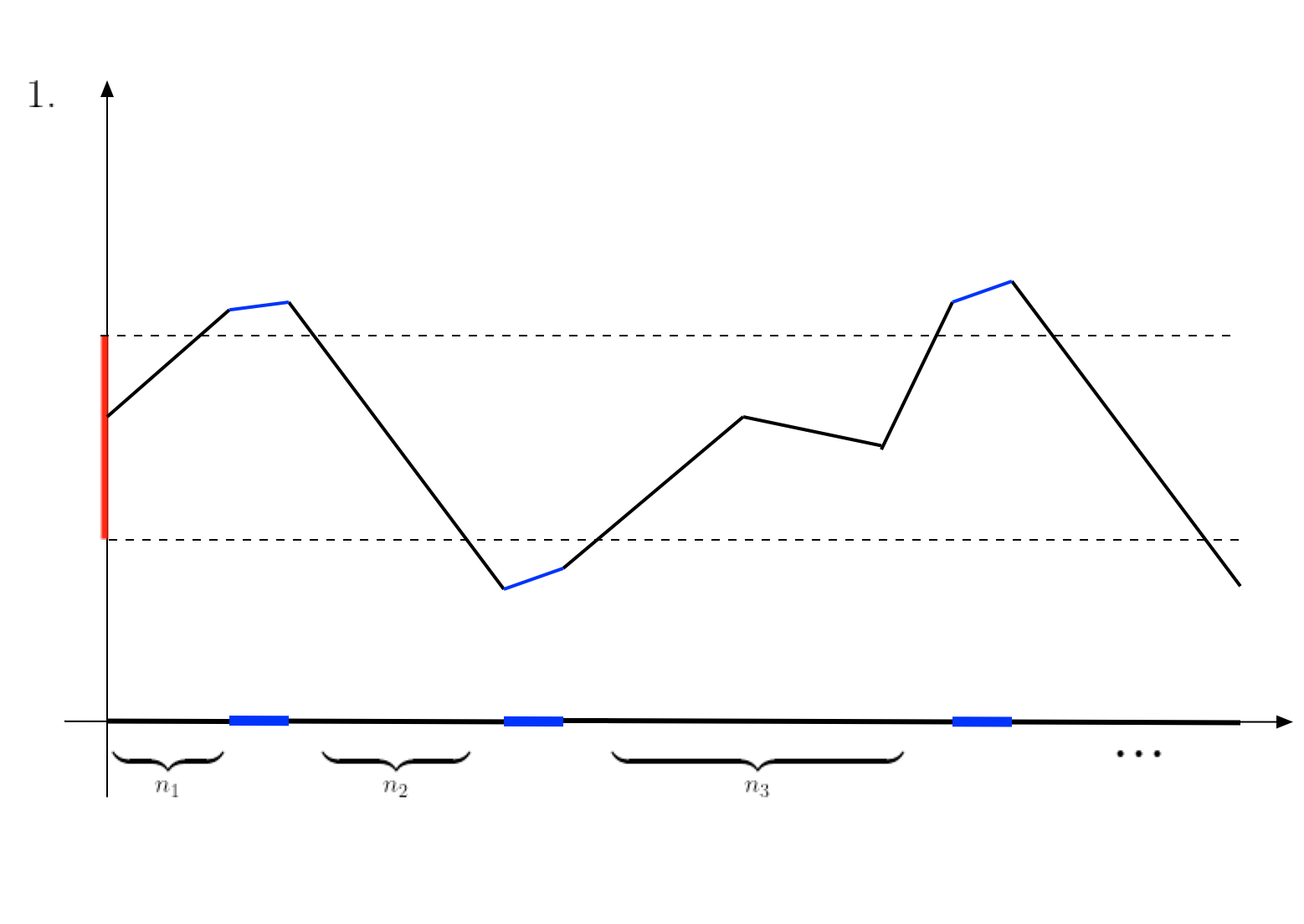}
    \includegraphics[scale=.25]{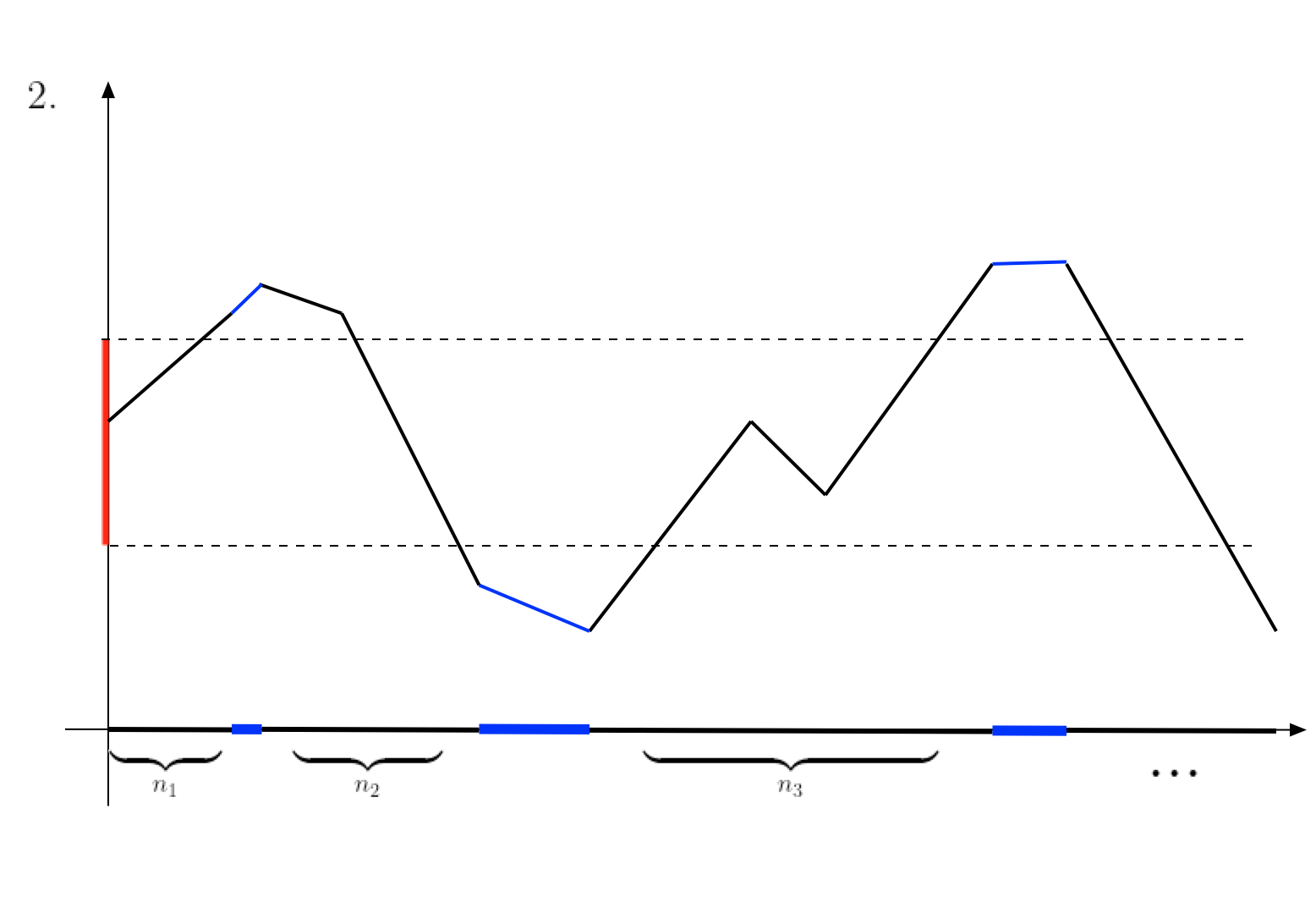} 
    \\
      \includegraphics[scale=.25]{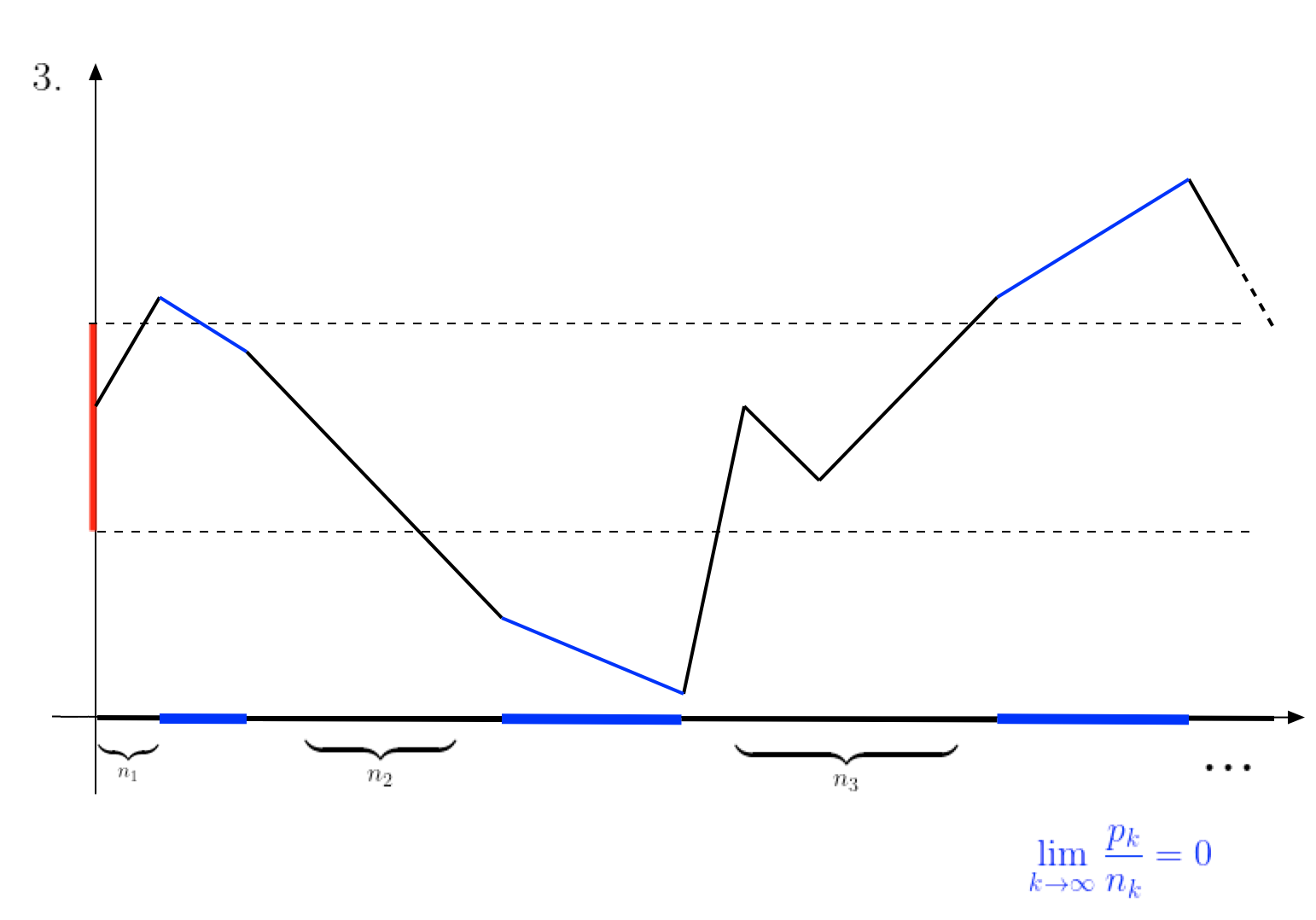}
        \includegraphics[scale=.25]{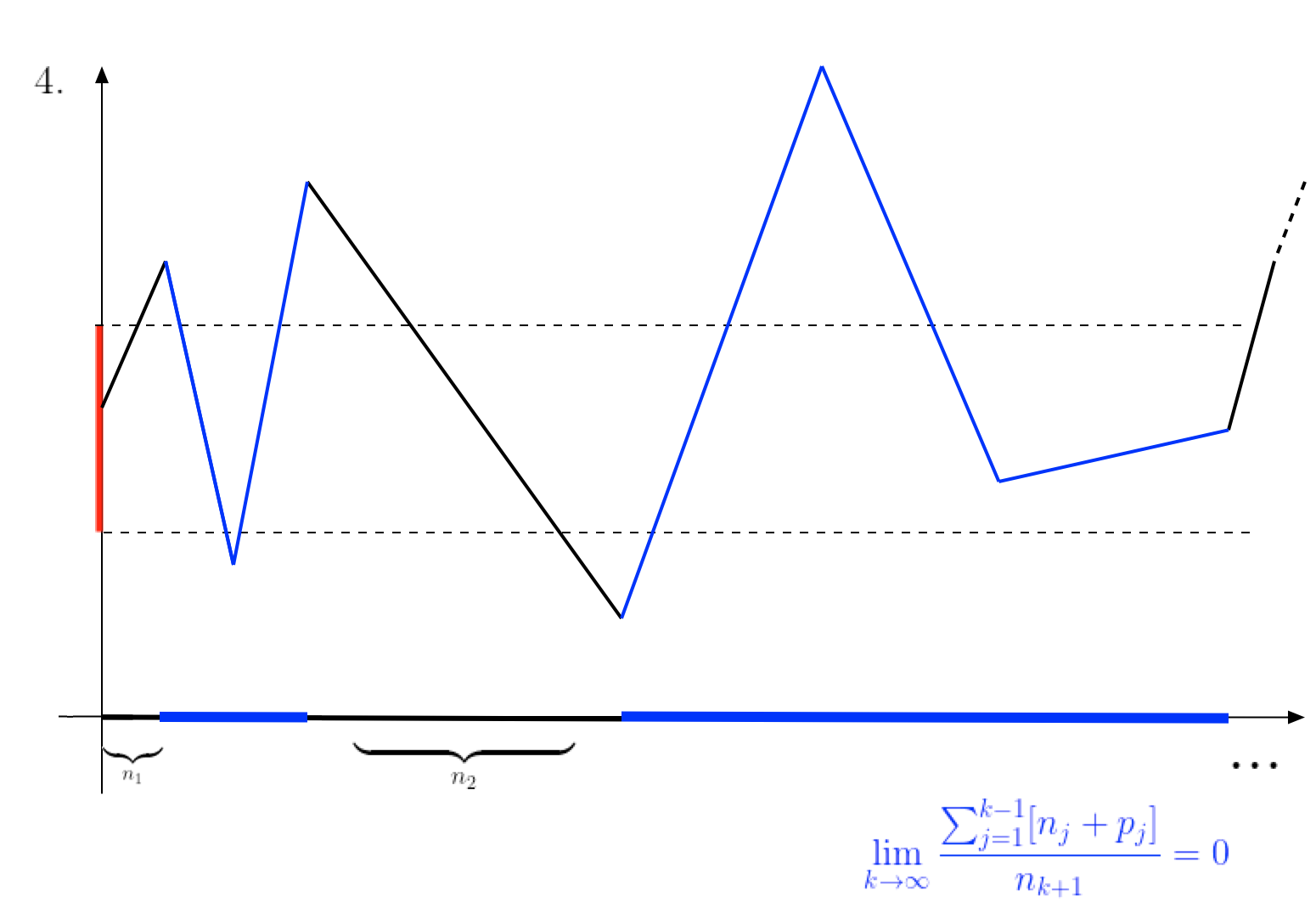}
\caption{Illustration of the oscillating behavior of time averages along the shadowing of finite pieces of orbits for maps satisfying: 
(1) specification, (2) gluing orbit property, (3) non-uniform specification, and (4) frequent hitting time property.}
\label{figure}
\end{center}
\end{figure}

\section{The set of irregular points for semigroup actions is Baire generic }\label{teoC}

This section is devoted to the proof of Theorem~\ref{thm:IFS}. In order to ease the presentation, we first explain how the assumptions in the 
theorem allow one to construct points in $X$ and an infinite path on the semigroup along which such point has divergent Birkhoff averages for the continuous
observable $\psi$ (Subsection~\ref{subsec1}). Exploring these ideas, we then prove that the set of non-typical points for the skew-product $F$ and the 
observable $\varphi_\psi$ forms a  Baire generic subset of $\Sigma_\kappa\times X$ (Subsection~\ref{subsec2}).
The proof of the theorem is completed in Subsection~\ref{subsec3}. 
\color{black}

\subsection{The set of irregular points is non-empty}\label{subsec1}

Assume that $G_1=\{id,f_1, f_2, \dots, f_\kappa\}$ is a collection of bi-Lipschitz homeomorphisms on a compact metric
space $X$ and that $\psi: X\to \mathbb R$ is a continuous observable which is not a coboundary with respect to some
of the elements in $G_1$. 
Assume, without loss of generality, that 
$$
\inf_{\mu\in \cM_1(f_\kappa)}  \int \psi \,d\mu < \sup_{\mu\in \cM_1(f_\kappa)}  \int \psi \,d\mu.
$$
Let $\mu_1,\mu_2$ be $f_\kappa$-invariant and ergodic probability measures such that  $\int \psi \,d\mu_1 
<  \int \psi \,d\mu_2$
and take $x_i \in B(\mu_i)$ be points in the ergodic basin of attraction with respect to $f_\kappa$, for $i=1,2$.  In particular, there are integers $1\leqslant n_1< n_2$ so that
\begin{equation}\label{eq:choice-vep001}
\Big| \frac1n \sum_{j=0}^{n-1} \psi(f_\kappa^j(x_i)) - \int \psi\, d\mu_i\Big| < \frac{\int \psi \,d\mu_2-\int \psi \,d\mu_1}{8}
\quad \text{for every}\, n\geqslant n_i
\end{equation}
for $i=1,2$. 
By uniform continuity of $\psi$, there one may choose $\vep_0>0$ such that 
\begin{equation}\label{eq:choice-vep}
|\psi(x)-\psi(y)|< \frac{\int \psi \,d\mu_2-\int \psi \,d\mu_1}{8}
	\quad\text{for every $x,y\in X$ so that $d(x,y)<\vep_0$.}
\end{equation}
To ease the notation, for each $x\in X$, $\vep>0$ and $n\geqslant 1$ denote by $B_{f_\kappa}(x,n,\vep)$ the dynamic ball taken with respect to $f_\kappa$, that is,
$$
B_{f_\kappa}(x,n,\vep):=
	\big\{y\in X \colon d(f_\kappa^j(y),f_\kappa^j(x)) \leqslant \vep \; \text{for every}\;  0\leqslant j \leqslant n-1\big\}.
$$
Since the semigroup action satisfies the frequent hitting time property, by Remark~\ref{rmk:def.hitting},
for every $\xi>0$ there exists $K(\xi)\geqslant 1$ so that for any balls $B_1,B_2$ in $X$ such that $|B_1|=\xi$ there exists 
$0\leqslant p \leqslant K(\xi)$, $\om\in \Sigma_\kappa$ and a ball $B_2'\subset B_2$ of radius 
$r\geqslant \min\{|B_1|/4,|B_2|/2\}$ so that $g_\om^p(B_1) \supset B_2'$. 

\medskip
The idea is to use the frequent hitting time property to create points with oscillating time averages by making them
shadow pieces of orbits associated with points with distinct time averages. 
The following auxiliary lemma provides a control on the size of dynamic balls and images of dynamic balls that 
will be instrumental in the argument. 

\begin{lemma}\label{le:est-balls}
If $f$ is a bi-Lipschitz homeomorphism and $L>1$ is a Lipschitz constant for both $f$ and $f^{-1}$ then
\begin{equation}\label{le:dynballs-vs-balls}
B_f(x,n,\vep) \supset B(x, L^{-n} \vep) 
	\quad\text{and}\quad
f^n(B_f(x,n,\vep)) \supset  B(f^n(x), L^{-2n} \vep) 
\end{equation}
for any $\vep>0$ and $n\geqslant 1$.
\end{lemma}

\begin{proof} If $x,y\in X$ and $d(x,y)<L^{-n}\vep$ then 
$
d(f^j(x), f^j(y)) \leqslant L^j d(x,y) \leqslant L^{j-n} \vep <\vep
$
for every $0\leqslant j \leqslant n-1$. This proves that $B_f(x,n,\vep) \supset B(x, L^{-n} \vep)$. Analogously, if 
$y\in B(f^n(x), L^{-2n} \vep)$ then $d(f^{-j}(y), f^{n-j}(x)) \leqslant L^j d(y, f^n(x)) \leqslant L^{j-2n} \vep< L^{-n}\vep$ for every 
$0\leqslant j \leqslant n-1$. 
This ensures that 
$$
B(f^n(x), L^{-2n} \vep) \subset f^n(B(x, L^{-n}\vep)) \subset f^n(B_f(x,n,\vep)),
$$
which finishes the proof of the lemma.
\end{proof}

\begin{remark}\label{rmk:localh}
The previous lemma holds more generally of $X$ is a compact and connected metric space and $f$ is 
a Lipschitz continuous local homeomorphism whose inverse branches are Lipschitz continuous. Indeed, in this context
there exists $\delta>0$ so that for every $x\in X$ there are well defined inverse branches 
$f_i^{-1}: B(x,\delta) \to f_i^{-1}(B(x,\delta))$,
and the conclusion of Lemma~\ref{le:est-balls} holds for every small $0<\vep<\delta$.
\end{remark}

In view of the estimates ~\eqref{le:dynballs-vs-balls}, one can see the frequent hitting times property will
be enough to guarantee that the transition times may grow exponentially fast with the size of the finite pieces 
of orbits to be shadowed. This can be observed in the following schematic Figure~\ref{fig-images} below.
\begin{figure}[htb]\label{fig-images}
\begin{center}
        \includegraphics[scale = 0.4]{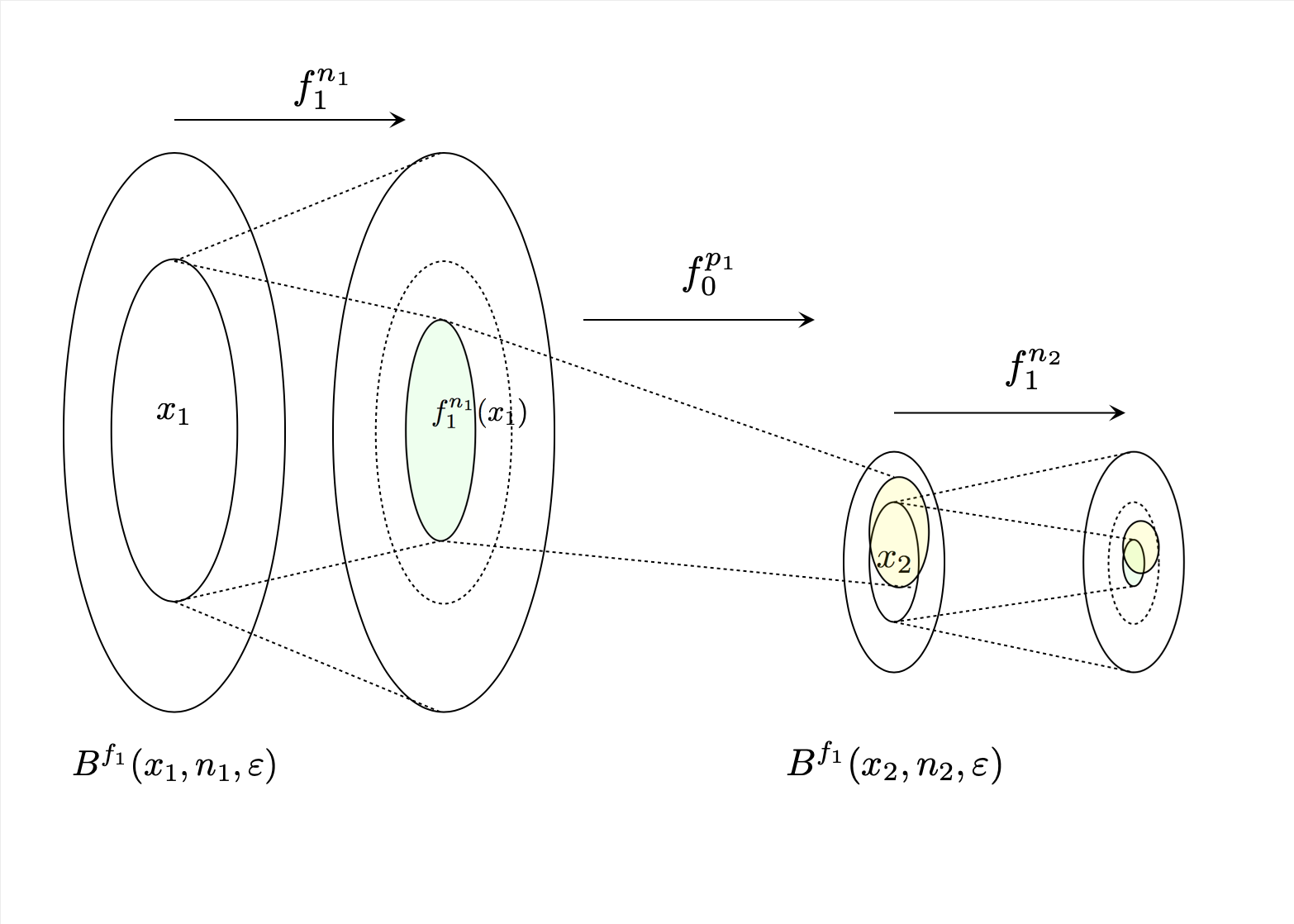}
        \caption{On the left: the image by $f_\kappa^{n_1}$ of the ball $B(x_1, L^{-n_1}\vep)$ contains a ball of radius $L^{-2n_1}\vep$ in green color. On the center: the frequent hitting times property ensures that the image of the green region by a suitable iterate of $f_0$ covers a ball of
        radius $L^{-n_2}\vep/2$}
\end{center}
\end{figure}

The difficulty now reduces to guarantee that the previous property is still enough to ensure the divergence of the time averages.
For that purpose we will make the following choice of constants.
Let us consider a strictly increasing sequence of positive integers $(n_k)_{k\geqslant 3}$ such that 
\begin{equation}\label{def:nk}
n_{j+1} > 2 n_j + \frac{2\log 2}{\log L} \;\text{for all $j\geqslant 2$}
\quad \text{and}\quad 
\lim_{k\to\infty} \frac{\sum_{j=0}^{k} \big[n_j + K(\frac{L^{-2j}\vep}2 )\big]}{n_{k+1}}=0,
\end{equation}
and for each $k\geqslant 1$ set 
$$
z_k=
\begin{cases}
x_1, \text{\; if $k$ odd} \\
x_2, \text{\; if $k$ even.}
\end{cases}
$$
The following proposition guarantees that there exist points which shadows the finite size orbits of the prescribed points $x_1$ and $x_2$, alternately, and that these have divergent time-averages. More precisely:

\begin{proposition}\label{prop:divergent.exist}
There exists $x\in X$ and ${\underline{\omega}}\in \Sigma_\kappa$ so that:
\begin{itemize}
\item[(a)] $x\in B_{f_\kappa}(x_1,n_1,\vep)$
\item[(b)] for every $k\geqslant 1$ there exists $0\leqslant p_k\leqslant K(\frac{L^{-2k}\vep}2 )$ and a finite
	word $\om^{(k)}:= \omega^{(k)}_{i_0} \omega^{(k)}_{i_1}  \dots \omega^{(k)}_{i_{p_k}}  \in\{1,2, \dots, \kappa\}^{p_k}$ such that
	\begin{equation}\label{eq:control-balls}
	f_{\underline{\omega}}^{\sum_{i=1}^k (n_i+p_i)}(x):=f_{\om^{(k)}}^{p_k} f_\kappa^{n_k} f_{\om^{(k-1)}}^{p_{k-1}} \dots 
		f_\kappa^{n_2} f_{\om^{(1)}}^{p_1}  f_\kappa^{n_1} (x) \in  B_{f_\kappa}( z_k,n_k,\vep) \quad \forall k\geqslant 2,
	\end{equation}
	where 
	\begin{equation}\label{eq:sequence}
	{\underline{\omega}} = ( \underbrace{\kappa, \kappa, \dots \kappa}_{n_1 }, 
					\omega^{(1)}_{i_0} \omega^{(1)}_{i_1}  \dots \omega^{(1)}_{i_{p_1}}, 
					\underbrace{\kappa, \kappa, \dots \kappa}_{n_2 }, 
					\omega^{(2)}_{i_0} \omega^{(2)}_{i_1}  \dots \omega^{(2)}_{i_{p_2}}, 
					\underbrace{\kappa, \kappa, \dots \kappa}_{n_3 }, \ldots).
	\end{equation}
	\end{itemize} 
Moreover,  
\begin{equation}
\label{eq:seq-nontypical}
\liminf_{n\to\infty} \frac1n\sum_{j=0}^{n-1} \psi(f_{\underline{\omega}}^j(x)
	< \limsup_{n\to\infty} \frac1n\sum_{j=0}^{n-1} \psi(f_{\underline{\omega}}^j(x).
\end{equation}
\end{proposition}

\begin{proof}
Set $B_1:=B_{f_\kappa}(x_1,n_1,\vep)$. By Lemma~\ref{le:est-balls} one has that
$f_\kappa^{n_1}(B_1)$ contains the ball of radius $L^{-2n_1}\vep$ around $f_\kappa^{n_1}(x_1)$.
Now, since the semigroup action has frequent hitting times and  $L^{-n_2}\vep < \frac{L^{-2n_1}\vep}4$ (recall the first inequality in ~\eqref{def:nk})
there are integers $K_1=K(\frac{L^{-2n_1}\vep}2)\geqslant 1$ and $0\leqslant p_1 \leqslant K_1$, a finite word 
$\om^{(1)}:=\omega^{(1)}_{i_0} \omega^{(1)}_{i_1}  \dots \omega^{(1)}_{i_{p_1}} \in\{1,2, \dots, \kappa\}^{p_1}$ 
and a ball $B_2$ of radius $\frac{L^{-n_2}\vep}2$ so that 
$$
B_2 \subset f_{\om^{(1)}}^{p_1}\big(B(f_\kappa^{n_1}(x_1),L^{-2 n_1}\vep)\big) \cap B(x_2,L^{-n_2}\vep)
	\subset f_{\om^{(1)}}^{p_1}(B(f_\kappa^{n_1}(x_1),L^{-2 n_1}\vep)) \cap B_{f_\kappa}(x_2,n_2,\vep)
$$
(recall Figure~\ref{fig-images}). By Lemma~\ref{le:est-balls}, 
$f_\kappa^{n_2}(B_2)$ contains a ball of radius $\frac{L^{-2n_2}\vep}2$. Using that  $L^{-n_3}\vep < \frac{L^{-2n_2}\vep}4$, there exist
$K_2=K(\frac{L^{-2n_2}\vep}2)\geqslant 1$ and $0\leqslant p_2 \leqslant K_2$, 
a finite word 
$\om^{(2)}:=\omega^{(2)}_{i_0} \omega^{(2)}_{i_1}  \dots \omega^{(2)}_{i_{p_2}} \in\{1,2, \dots, \kappa\}^{p_2}$ 
and a ball $B_3$ 
of radius $\frac{L^{-n_3}\vep}2$ 
so that 
$$
B_3 \subset f_{\om^{(2)}}^{p_2}\big( f_\kappa^{n_2}(B_2)\big) \,\cap\, B(x_3,L^{-n_3}\vep)
	\subset f_{\om^{(2)}}^{p_2}( f_\kappa^{n_2}(B_2))   \cap B_{f_\kappa}(x_3,n_3,\vep).
$$
Proceeding recursively, we conclude that for every $s\geqslant 1$ there exists $0\leqslant p_{s-1} \leqslant K(\frac{L^{-2n_{s-1}}\vep}2)$, a finite word 
$\om^{(s)}:=\omega^{(s)}_{i_0} \omega^{(s)}_{i_1}  \dots \omega^{(s)}_{i_{p_s}} \in\{1,2, \dots, \kappa\}^{p_s}$  
and a ball $B_s$ of radius $\frac{L^{-n_s}\vep}2$ contained in 
$f_{\om^{(s)}}^{p_{s-1}} (f_\kappa^{n_{s-1}}(B_{s-1})) \cap B_{f_\kappa}(x_s,n_s,\vep)$.
Moreover, if ${\underline{\omega}}\in \Sigma_\kappa$ is defined by ~\eqref{eq:sequence},
we claim that any point
$$
x\in B_1 \cap \bigcap_{k> 1}  \big(f_{\underline{\omega}}^{\sum_{j=1}^{k} (n_j +p_j)}\big)^{-1} (B_k).
$$
is irregular in the sense of ~\eqref{eq:seq-nontypical}
(the intersection is non-empty because the right hand-side above consists of an intersection 
of compact and nested sets).
Indeed, according to ~\eqref{eq:control-balls}, for each $k\geqslant 1$ one has that 
$y_k:=f_{\underline{\omega}}^{\sum_{i=1}^{k} (n_i+p_i)}(x)\in B_{f_\kappa}(z_k,n_k,\vep)$.
In consequence, recalling ~\eqref{eq:choice-vep001}, \eqref{eq:choice-vep} and ~\eqref{def:nk}, if $k\geqslant 1$ is even then
\begin{align*}
\Big| \sum_{j=0}^{n_{k+1}+\sum_{i=1}^{k} (n_i+p_i)} \psi(f_{\underline{\omega}}^j(x)) 
	- n_{k+1} \, \int \psi \, d\mu_1 \Big| 
	& 	\leqslant \Big| \sum_{j=0}^{n_{k+1}+\sum_{i=1}^{k} (n_i+p_i)} \psi(f_{\underline{\omega}}^j(x))
	- \sum_{j=0}^{n_{k+1}} \psi(f_{\underline{\omega}}^j(y_k)) \Big| 
	  \\
	& + \Big| \sum_{j=0}^{n_{k+1}} \psi(f_{\underline{\omega}}^j(y_k)) 
	- n_{k+1} \, \int \psi \, d\mu_1 \Big|
	\\
	& \leqslant  \|\psi\|_\infty \sum_{i=1}^{k} (n_i+p_i) 
	+ n_{k+1} \frac{\int \psi \,d\mu_2-\int \psi \,d\mu_1}{4}. 
\end{align*}
In consequence, if $k\geqslant 1$ is even and large then 
\begin{align*}
\frac1{n_{k+1}+\sum_{i=1}^{k} (n_i+p_i)}\sum_{j=0}^{n_{k+1}+\sum_{i=1}^{k} (n_i+p_i)} \psi(f_{\underline{\omega}}^j(x))
	&  \leqslant \frac{n_{k+1}}{n_{k+1}+\sum_{i=1}^{k} (n_i+p_i)} \int \psi\, d\mu_1 \\
	& + \|\psi\|_\infty \, \frac{\sum_{j=1}^{k} \big[n_j + K(\frac{L^{-2j}\vep}2 )\big]}{n_{k+1}} 
	 + \frac{\int \psi \,d\mu_2-\int \psi \,d\mu_1}{4} \\
	& < \int \psi\, d\mu_1 + \frac{\int \psi \,d\mu_2-\int \psi \,d\mu_1}{3}. 
\end{align*}
Analogously, it is not hard to check that 
\begin{align*}
\frac1{n_{k+1}+\sum_{i=1}^{k} (n_i+p_i)}\sum_{j=1}^{n_{k+1}+\sum_{i=1}^{k} (n_i+p_i)} \psi(f_{\underline{\omega}}^j(x))
> \int \psi\, d\mu_2 - \frac{\int \psi \,d\mu_2-\int \psi \,d\mu_1}{3} 
\end{align*}
for every $k\geqslant 1$ odd and large. This implies that  the Birkhoff averages $\big(\frac1n\sum_{j=0}^{n-1} \psi(f_{\underline{\omega}}^j(x)\big)_{n\geqslant 1}$ 
do not converge and completes the proof of the proposition.
\end{proof}

\begin{remark}\label{rmk:rotation}
The crucial point in ~\eqref{def:nk} is that the constants $K(\cdot)$ involved in the proof 
may grow arbitrarily fast but are independent of the shadowing points.
For instance, in the special case that $f$ is an irrational rotation on the circle, Remark~\ref{rmk:irrational-rot} ensures that $K(L^{-2n}\vep/2)\simeq  3 L^{2n} [\frac1\vep]\gg n$ grows exponentially
fast with $n$. 
\end{remark}

\subsection{The set of irregular points is Baire generic}\label{subsec2}
	
The strategy used in the proof of Theorem~\ref{thm:IFS} 
is slightly different from the classical strategies (see e.g \cite{Barreira1,LV}), where the use of the specification or the gluing orbit property is used to make successive approximations of points with some preestablished behavior. 
The difference does not
rely only on the fact that in the process of shadowing a finite piece of orbit, time lags between the shadowing times 
to become larger as these depend on the proximity scale.
In fact, while specification allows to make successive approximations $(z_k)_k$ of a certain initial point $x_0$  
by points $z_k$ whose time averages diverge  and so that $\sum_{k\geqslant 1} d(z_k, x_0)$ is small, the strategy here is to construct nested sequences of sets with a irregular behavior by the analysis of the images of dynamic balls. Such a finer control, provided by the frequent hitting times property, is not imediate 
even if the dynamical system satisfies the 
specification property. This will become evident in the proof of the following proposition:

\begin{proposition}\label{mainpropC}
Under the assumptions of Theorem~\ref{thm:IFS}, the set
$I_F(\varphi_\psi)$ is a Baire residual subset of $\Sigma_\kappa\times X$.
\end{proposition} 
\medskip

The proof of Proposition~\ref{mainpropC}, which occupies the remainder of this subsection, is divided in a three-step construction.
The first two steps, whose content is closer to the arguments presented in Subsection~\ref{subsec1}, consider non-typical behavior along 
paths in the semigroup $G$. In the third step, one constructs nested sets with oscilating behavior and prescribed initial condition for the skew-product.

Consider a strictly decreasing sequence of positive numbers $(\delta_k)_{k\geqslant 1}$ tending to zero. Associated to these, 
pick a strictly increasing sequence of large positive integers $(n_k)_{k\geqslant 1}$ satisfying
\eqref{def:nk} and in such a way that for each $\ell\geqslant 1$ the sets
$$
\Gamma_{2\ell-1}	:=\Big\{x\in X: \big|\frac{1}{n_{2\ell-1}}\sum_{j=0}^{n_{2\ell-1}-1}\psi(f_\kappa^j(x))-\int \psi \,d\mu_1\big|< \delta_{2\ell-1}\Big\}
$$ 
and
$$
\Gamma_{2\ell}:=
\Big\{x\in X:\big|\frac{1}{n_{2\ell}}\sum_{i=0}^{n_{2\ell}-1}\psi(f_\kappa^j(x))-\int \psi \,d\mu_2\big|< \delta_{2\ell}\Big\} 
$$
are non-empty. This is possible by ergodicity of the measures $\mu_i$ with respect to $f_\kappa$ ($i=1,2$), and the ergodic theorem. 

\medskip
Let $D\subset \Sigma_\kappa\times X$ 
be a countable and dense set. Fix $(\om_0,x_0)\in D$ and set 
$S_0=\{(\om_0,x_0)\}$. Now, for each $\ell\geqslant 1$ let $S_{2\ell-1} \subset \Gamma_{2\ell-1}$ be a $(f_\kappa, n_{2\ell-1}, 8\varepsilon)$-separated set in $X$, i.e. a set so that 
$$\max_{0\leqslant j \leqslant n_{2\ell-1}} d(f_\kappa^j(x), f_\kappa^j(y)) > 8\vep$$
for every distinct $x, y \in S_{2\ell-1}$. Analogously, let $S_{2\ell}\subset \Gamma_{2\ell}$ be a $(f_\kappa, n_{2\ell},8\varepsilon)$-separated set. 
Now consider a sequence of positive integers $(N_k)_{k \geqslant 1}$ {(independent of $\om_0$ and $x_0$)} 
so that 
\begin{equation}\label{eq:N_k-def}
L^{-n_{k+1}}\varepsilon< \min\Big\{\frac{L^{-2^{N_k-1}n_k}\varepsilon}{2^{2N_k-1}}, L^{\lfloor \log (\vep^{-1})\rfloor}\vep/4\Big\}\\
\;\;\;\mbox{and}\;\;\;\\
\lim_{k\to\infty}\frac{\sum_{i=0}^{k}\left[N_i n_i +\sum_{j=1}^{N_i-1} K(\frac{L^{-2^j\cdot n_j}\vep}{2^j} )+K(\frac{L^{-2^{N_i-1}\cdot n_i}\vep}{2^{N_i-1}})\right]}{n_{k+1}}=0.
\end{equation}

\medskip	
We explain the construction dividing it in steps, which involve the frequent hitting times property. 
First we explain the construction of sets which replicate the behavior of the points in $S_k$ for a 
certain number $N_k$ of times. Then we repeat such argument to construct a nested sequence of sets 
points which combine the different behaviors. 

\medskip
\noindent\textbf{$1^{st}$ Step:} \emph{Replication of the behavior of points in each set $S_k$}
\smallskip

Fix $k\geqslant 1$.
For each collection $\ud x_k=(x_1^k, \cdots,  x_{N_k}^k)\in  S_k^{N_k}$ we will explore ideas similar the ones used in Proposition~\ref{prop:divergent.exist} to construct a set $C(\ud x_k)\subset X$ of points which shadow pieces of orbits of points in $S_k$. Let us be more precise.
First recall that $B_1^k:=B(x_1^k,L^{-n_k}\vep)\subset B_{f_\kappa}(x_1^k,n_k,\vep)$ and that Lemma~\ref{le:est-balls} ensures 
$f_\kappa^{n_k}(B(x_1^k,L^{-n_k}\vep))$ contains the ball of radius $L^{-2n_k}\vep$ around $f_\kappa^{n_k}(x_1^k)$.
As the semigroup action has frequent hitting times 
	there are constants $K_1=K(L^{-2n_k}\vep)\geqslant 1$ and $0\leqslant p^k_1 \leqslant K_1$, a finite word
	 $\om^{(1)}:=\omega^{(1)}_{i_0} \omega^{(1)}_{i_1}  \dots \omega^{(1)}_{i_{p^k_1}} \in\{1,2, \dots, \kappa\}^{p^k_1}$ 
	 and a ball $B_2^k$ of radius $\frac{L^{-2n_k}\vep}{2^2}$ 
	so that 
\begin{align}
	B_2^k & \subset f_{\om^{(1)}}^{p^k_1} \Big(B(f_\kappa^{n_k}(x_1^k),L^{-2 n_k}\vep)\Big) \cap B \Big(x_2^k,\frac{L^{-2 n_k}\vep}{2}\Big) \nonumber \\
	&
	\subset f_{\om^{(1)}}^{p^k_1} \Big(B(f_\kappa^{n_k}(x_1),L^{-2 n_k}\vep)\Big) \cap B_{f_\kappa}(x_2^k,n_k,\vep) \nonumber \\
		&
	\subset f_{\om^{(1)}}^{p^k_1} \Big(\,f_\kappa^{n_k} \big( B_{f_\kappa}(x_1^k,n_k,\vep) \big)\,\Big) \cap B_{f_\kappa}(x_2^k,n_k,\vep).
	\label{eq:B1}
\end{align}
In particular
\begin{align}\label{eq:B2}
B_{f_\kappa}(x_1^k,n_k,\vep) \; \cap \; \big(\,f_{\om^{(1)}}^{p^k_1}\circ f_\kappa^{n_k}\,\big)^{-1} \Big(\,B_{f_\kappa}(x_2^k,n_k,\vep) \,\Big) \neq\emptyset.
\end{align}
Using Lemma~\ref{le:est-balls} once more, we conclude that 
$$
f_\kappa^{n_k}(B_2^k) \; \text{contains a ball of radius}\;  
L^{-2n_k}\times  \frac{L^{-2n_k}\vep}{2^2}=\frac{L^{-2^2n_k}\vep}{2^2}.
$$ 
Again, using the frequent hitting times property and ~\eqref{eq:B1}, there exist $K_2=K(\frac{L^{-2^2n_k}\vep}{2^2}) \geqslant 1$, $0\leqslant p_2^k \leqslant K_2$, a finite word  $\om^{(2)}:=\omega^{(2)}_{i_0} \omega^{(2)}_{i_1}  \dots \omega^{(2)}_{i_{p_2^k}} \in\{ 0,1, \dots, \kappa\}
^{p^k_2}$ and a ball $B_3^k$ of radius $\frac{L^{-2^2n_k}\vep}{2^4}$ 
	so that 
\begin{align*}
	B_3^k & \subset f_{\om^{(2)}}^{p_2^k}( f_\kappa^{n_k}(B_2^k)) \cap B(x_3^k,\frac{L^{-2^2n_k}\vep}{2^3}) \\
	& \subset f_{\om^{(2)}}^{p_2^k}( f_\kappa^{n_k} 
	\Big( f_{\om^{(1)}}^{p^k_1} \Big(\,f_\kappa^{n_k} \big( B_{f_\kappa}(x_1^k,n_k,\vep) \big)\,\Big) \cap B_{f_\kappa}(x_2^k,n_k,\vep) \Big)  
	 \cap B_{f_\kappa}(x_3^k,n_k,\vep).
\end{align*}
and its image 
$f_\kappa^{n_k}(B_3^k)$ contains a ball of radius
$\frac{L^{-2^3n_k}\vep}{2^4}$ (recall the first choice of constants in ~\eqref{eq:N_k-def}).

\medskip
Proceeding recursively, for each $3\leqslant i\leqslant N_k-1$ there exists 
a ball $B_i^k$ of radius $\frac{L^{-2^{i-1}n_k}\vep}{2^{2i-3}}$ so that the set  $f_\kappa^{n_k}(B_i^k)$ 
contains a ball of radius 
$\frac{L^{-2^{i}n_k}\vep}{2^{2(i-1)}}$ and there are $0\leqslant p^k_{i} \leqslant K_{i}=K(\frac{L^{-2^{i}n_k}\vep}{2^{2(i-1)}})$ and 
a finite word $\om^{(i)} 
\in\{1,2, \dots, \kappa\}^{p_{i}^k}$ so that 
$f_{\om^{(i)}}^{p_{i}^k} (f_\kappa^{n_{k}}(B_{i}^{k})) \cap B_{f_\kappa}(x_i^k,n_k,\vep)$
contains a ball $B_{i+1}^k$ of radius $\frac{L^{-2^i n_k}\vep}{2^{2i-1}}$. 
For each $\ud x_k=(x_1^k\cdots x_{N_k}^k)\in S_k^{N_k}$ we collect the data 
\begin{align}\label{eq:datak}
\big(\,n_k\,;\, (\om^{(1)},p^k_1),\dots ,(\om^{(N_k-1)}, p^k_{N_k-1})\,;\, (B_1^k, \dots, B_{N_k}^k)\,\big)
\end{align}
given by the previous algorithm, where the times $p^k_i$ (depending both on the size of the orbits 
and the points to be shadowed) are bounded by $K_i$. 
Moreover, by construction, the set 
\begin{align}\label{eq:defCxk}
C(\ud x_k)
	:=B_1^k(x^k_1)\cap 
	\bigcap_{i=2}^{N_k-1} \left( f_{{\underline{\omega}}}^{\sum_{j=1}^{i-1} (n_k+\,p_j^k)}\right)^{-1}\big(B^k_{i}(x_{i}^k)\big)
\end{align}
is non-empty, where ${\underline{\omega}} = {\underline{\omega}}^{(k)}(\ud x_k)$ is defined by
\begin{eqnarray}\label{eq:B-words-k}
{\underline{\omega}}^{(k)}(\ud x_k):=
	 (\, \underbrace{\kappa, \kappa, \dots \kappa}_{n_k \,\text{times}}, 
	\underbrace{\omega^{(1)}_{i_0} \omega^{(1)}_{i_1}  \dots \omega^{(1)}_{i_{p_1}}}_{= \,\underline{\omega}^{(1)}}, 
	\underbrace{\kappa, \kappa,  \dots \kappa}_{n_k \,\text{times} }, 
	\ldots, \underbrace{\omega^{(k-1)}_{i_0} \omega^{(k-1)}_{i_1}\,  \dots \, \omega^{(k-1)}_{i_{p_{k-1}}}}_{_{= \,\underline{\omega}^{(N_{k}-1)}}}, 
	\underbrace{\kappa, \kappa, \dots \kappa}_{n_k \,\text{times}}\, ).
\end{eqnarray}
For future reference we denote by $B^k_{i}(x_{i}^k)$ the balls $B^k_{i}$ as before, to emphasize that the latter are contained 
in a neighborhood of the point $x_{i}^k$. Moreover, any point $x\in C(\ud x_k)$
satisfies 
\begin{align*}\label{eq:conclusion-Cxk}
\begin{cases}
x\in  B_{f_\kappa}(x_1^k,n_k,\vep) \medskip \\ 
f_{{\underline{\omega}}}^{\sum_{j=1}^{i-1} (n_k+\,p_j^k)}(x) \in B_{f_\kappa}(x_i^k,n_k,\vep), 
	\quad \forall\, 1 \leqslant i \leqslant N_k.
\end{cases}
\end{align*}

\medskip
\noindent\textbf{$2^{nd}$ Step:} 
\emph{Construction of nested sets with oscilating behavior}

\smallskip

In this step we explain the construction of nested sets which shadow points
in successive collections  $S_k^{N_k}$. We first make explicit such construction involving collections
of points in $S_1^{N_1}$ and $S_2^{N_2}$. 

Given $\ud x_1=(x_1^1, \cdots,  x_{N_1}^1)\in  S_1^{N_1}$, let $C(\ud x_1)$ be the set
defined by ~\eqref{eq:defCxk} and consider the ball
$B_{N_1}^1(x_{N_1}^1) \subset B_{f_\kappa}(x_{N_1}^1,n_1,\vep)$ be the ball
of radius $\frac{L^{-2^{N_1-1}n_1}\vep}{2^{2N_1-3}}$  
whose image  $f_\kappa^{n_1}(B_{N_1}^1(x_{N_1}^1))$ contains a ball of radius $\frac{L^{-2^{N_1-1}n_1}\vep}{2^{2(N_1-1)}}$.
By choice of the constants (recall \eqref{eq:N_k-def}) we have $L^{-n_2}\varepsilon < \frac{1}{2}\frac{L^{-2^{N_1-1}n_1}\vep}{2^{2(N_1-1)}}$. 
Hence, the frequent hitting times property implies that there exists $0\leqslant P_{1} \leqslant K(\frac{L^{-2^{N_1-1}n_1}\vep}{2^{2(N_1-1)}})$, a finite word $\om^{(12)}:= \omega^{(12)}_{i_0} \omega^{(12)}_{i_1}  \dots \omega^{(12)}_{i_{P_1}}  \in\{1,2, \dots, \kappa\}^{P_1}$ and ball $B_1^2$ of radius $\frac12 L^{-n_2}\varepsilon$ so that 
$$
B_1^2 
	\subset f_{\ud\omega^{(12)}}^{P_1}(f_\kappa^{n_{1}}(B_{N_1}^{1})) \cap B(x_2^1,L^{-n_2}\varepsilon)
	\subset f_{\ud\omega^{(12)}}^{P_1}(f_\kappa^{n_{1}}(B_{N_1}^{1})) \cap B_{f_\kappa}(x_2^1,n_2,\vep).
$$ 
Then, we can apply the argument used in the $1^{st}$ Step recursively
(noting that when $k=2$ the only difference is that the ball $B_1^2$ is half the radius of the corresponding 
object in the $1^{st}$ Step) and deduce that, in addition to the data in \eqref{eq:datak},
 for each $2\leqslant \ell\leqslant k$ there exists an integer 
$0\leqslant P_{\ell} \leqslant K(\frac{L^{-2^{N_i-1}n_i}\vep}{2^{2(N_i-1)}})$, a finite word $\om^{(\ell (\ell+1))}:= \omega^{(\ell (\ell+1))}_{i_0} \omega^{(\ell (\ell+1))}_{i_1}  \dots \omega^{(\ell (\ell+1))}_{i_{P_\ell}}  \in\{1,2, \dots, \kappa\}^{P_\ell}$ 
so that the set
\begin{align*} 
C(\ud x_1, \ud x_2, \dots, \ud x_k)
	:=
	& B_1^1(x^1_1)  
	\; \cap \;
	\bigcap_{i=2}^{N_1-1} \left( f_{{\underline{\omega}}}^{\sum_{j=1}^{i-1} (n_j+\,p_j^k)}\right)^{-1}\big(B^1_{i}(x_{i}^1)\big)
	\\ 
	&
	\; \cap \;
	\bigcap_{\ell=2}^{k} 
	\bigcap_{i=2}^{N_{\ell}-1} 
	\left( f_{{\underline{\omega}}}^{\sum_{s=1}^{\ell-1} P_s + \sum_{j=1}^{i-1} (n_j+\,p_j^\ell)}\right)^{-1}
	\big(B^\ell_{i}(x_{i}^\ell)\big)
\end{align*}
is non-empty, where 
\begin{equation*}\label{eq:concateSkk+1}
	{\underline{\omega}} = 
	(\,
	{\underline{\omega}}^{(1)}(\ud x_1) 
	,\, 
	\om^{(12)}, 
	\,
	{\underline{\omega}}^{(2)}(\ud x_2) 
		,\, 
	\om^{(23)}, 
	\,
	\dots,
	\, 
	{\underline{\omega}}^{(k)}(\ud x_k) 
	\,)
\end{equation*}
It is immediate from the construction that 
$
C(\ud x_1, \ud x_2, \dots, \ud x_k, \ud x_{k+1}) \subset C(\ud x_1, \ud x_2, \dots, \ud x_k)
$
for every $k\geqslant 1$ and any collections $\ud x_i \in S_i^{N_i}$.

\medskip
\noindent\textbf{$3^{rd}$ Step:} 
\emph{Construction of nested sets in $\Sigma_\kappa\times X$ with oscilating behavior and prescribed initial condition for the skew-product map}
\smallskip

We first explain how to construct points $(\om,x)\in \Sigma_\kappa\times X$ whose orbit by $F$ starts close to the point $(\om_0,x_0)\in S_0$ and, subsequently, 
its projection $\pi_X(F^j(\om,x))$ over $X$ shadows the finite pieces of orbits of any chosen collection
of points $\ud x_1=(x_1^1, \cdots,  x_{N_1}^1)\in  S_1^{N_1}$. Let us be more precise.

\smallskip

Considering the product metric on $\Sigma_\kappa\times X$, the ball $B((\om_0,x_0),\vep)$ 
of radius $\vep$ in $\Sigma_\kappa\times X$ coincides with the product of the cylinder $[\om_0]_{[\lfloor \log \vep\rfloor, \lfloor \log (\vep^{-1})\rfloor]}$
by the ball $B(x_0,\vep)$. For that reason, in order to consider non-typical points for the skew-product it is natural to proceed in the construction 
of consider the set $f_{\om_0}^{\lfloor \log (\vep^{-1})\rfloor}(B(x_0,\vep))\subset X$, which contains the ball of radius $L^{\lfloor \log (\vep^{-1})\rfloor}\vep$ around
$f_{\om_0}^{\lfloor \log (\vep^{-1})\rfloor}(x_0)$.
Fix a collection $\ud x_1=(x_1^1, \cdots,  x_{N_1}^1)\in  S_1^{N_1}$ and observe that $L^{-n_1}\vep < L^{\lfloor \log (\vep^{-1})\rfloor}\vep/4$ (recall ~\eqref{eq:N_k-def}). 
Then, the frequent hitting times property implies that there exists $0\leqslant P_0\leqslant K(\vep)$
and a finite word $\om^{(01)} \in\{1,2, \dots, \kappa\}^{P_{0}}$ so that the image
$f_{\om^{(01)}}^{P_0}(B(f_{\om_0}^{\lfloor \log (\vep^{-1})\rfloor}(x_0),L^{\lfloor \log (\vep^{-1})\rfloor}\vep))$ contains a  ball $B_1^1 \subset B_{f_\kappa}(x_1^1,n_1,\vep)$ of radius 
$L^{-n_1}\vep/2$. Proceeding \emph{ipsis literis} as in the $2^{nd}$ Step (with $B_1^1$ as above) we conclude that
the modified sets 
\begin{eqnarray}\label{def. connec. sets 0}
C( x_0,\,\ud x_1, \dots,\ud x_k) 
	= B(x_0,\vep)\cap \big(f_{\om^{01}}^{P_0}f_{\om_0}^{\lfloor \log (\vep^{-1})\rfloor}\big)^{-1} \Big( C(\ud x_1, \ud x_2, \dots, \ud x_k) \Big)
\end{eqnarray}
are non-empty and nested, as desired. Even though we omit the dependence of the sets $C(x_0,\,\ud x_1, \dots,\ud x_k)$ on the finite word 
\begin{eqnarray}\label{def.omegaxs}
	{\underline{\omega}} = 
	(\,
	[\om_0]_{[\lfloor \log \vep\rfloor, \lfloor \log (\vep^{-1})\rfloor]}
	,\, 
	\om^{(01)}, 
	{\underline{\omega}}^{(1)}(\ud x_1) 
	,\, 
	\om^{(12)}, 
	\,
	{\underline{\omega}}^{(2)}(\ud x_2) 
		,\, 
	\om^{(23)}, 
	\,
	\dots,
	\, 
	{\underline{\omega}}^{(k)}(\ud x_k) 
	\,)
\end{eqnarray}
depending on $(x_0,\,\ud x_1, \dots,\ud x_k)$, 
used to determine it, we need to bookkeep it. Here, by some abuse of notation we are denoting by $	[\om_0]_{[\lfloor \log \vep\rfloor, \lfloor \log (\vep^{-1})\rfloor]}$
the finite word that defines the cylinder set.

\color{black}
	
\medskip	
Now we are in a position to resume the proof of Proposition~\ref{mainpropC} and to construct a $G_\delta$-dense subset of points with non-typical 
behavior for the skew-product. Take the non-empty compact set 
$$\mathfrak F'_k(\varepsilon,(\om_0,x_0)) 
=\bigcup \;\; 
[\om]\times \overline{C( x_0,\,\ud x_1, \dots, \ud x_k)} 
$$ 
where the union is taken among all possibilities $(\ud x_1 \cdots \ud x_k)\in  S_1^{N_1} \times ... \times S_k^{N_k}$ and 
$\om=\om(x_0,\ud x_1 \cdots \ud x_k)$ is the finite word determined by ~\eqref{def.omegaxs}. 
Define
$$
\mathfrak  F'(\varepsilon,(\om_0,x_0)):=\bigcap_{k=1}^{\infty} \mathfrak F'_k(\varepsilon,(\om_0,x_0))
	\quad\text{and}\quad 
	\mathfrak F':= \bigcup_{j=\lfloor\frac1{\vep_0} \rfloor}^\infty \bigcup_{(\om_0,x_0)\in D} \mathfrak F'\big(\frac{1}{j},(\om_0,x_0)\big),
$$  
where $\vep_0$ was chosen in ~\eqref{eq:choice-vep}.
By construction the sets in ~\eqref{def. connec. sets 0} are nested as well as the cylinders of increasing size in $\Sigma_\kappa$, and so 
$\mathfrak F'_{k+1}(\varepsilon,(\om_0,x_0))\subseteq \mathfrak F'_k(\varepsilon,(\om_0,x_0))$ for every point 
$(\om_0,x_0)\in D$, $\vep>0$ and $k\geqslant 1$. 
We consider, alternatively, the family 
$$
\mathfrak F_k(\varepsilon)
=\bigcup \;\;  
[\om]\times \overline{C( x_0,\,\ud x_1, \dots, \ud x_k)} 
$$ 
where the union is taken among all possible elements $(\om_0,x_0)\in D$ and $(\ud x_1, \dots, \ud x_k)\in  S_1^{N_1} \times ... \times S_k^{N_k}$,
and the finite words $\om$ are once more determined by ~\eqref{def.omegaxs}.
In this case, define
$$
\mathfrak F_k=\bigcup_{j=\lfloor\frac1{\vep_0} \rfloor}^{\infty}\mathfrak F_k\big(\frac{1}{j}\big)
		\quad\text{and}\quad 
	\mathfrak  F=\bigcap_{k=1}^{\infty}\mathfrak F_k
$$ 
By construction, {$\mathfrak  F_{k+1}\subset \mathfrak  F_k$ for every $k\geqslant 1$} and $\mathfrak  F' \subset \mathfrak  F$.
The proof of the proposition is now reduced to the proof of the following two lemmas.

\begin{lemma}
$\mathfrak F$ is a $G_\delta$-dense subset of $\Sigma_\kappa\times X$.
\end{lemma}

\begin{proof}
We claim that
$\mathfrak F'$ is dense in $\Sigma_\kappa\times X$.
Indeed, for $r>0$ and $(\om,x)\in \Sigma_\kappa\times X$ there
exist $j\geqslant 1$ and $(\om_0,x_0)\in D$ such that  $d_{\Sigma_\kappa\times X}((\om,x),(\om_0,x_0))<1/j<r/2$. Any point 
$(\tilde\om,y)\in \mathfrak F'(1/j,(\om,x_0))$ satisfies 
$d_{\Sigma_\kappa\times X}((\tilde \om,y),(\om_0,x_0))<1/j$ and, consequently, $d_{\Sigma_\kappa\times X}((\tilde \om,y),(\om,x))<2/j<r$.
As $\mathfrak F' \subset \mathfrak F$, this shows that $\mathfrak F$ is dense in $\Sigma_\kappa\times X$ as well.

\medskip
Now, in order to prove that $\mathfrak F$ is a $G_{\delta}$-set,
we define the sets 
$$
\mathfrak  G_k(\vep,(\om_0,x_0)):=\bigcup \big\{\; [\om] \times C(x_0,\,\ud x_1, \dots, \ud x_k): (\om_0,x_0)\in D,\; 
(\,\ud x_1, \dots, \ud x_k)\in  S_1^{N_1} \times ... \times S_k^{N_k}\big\},
$$ 
with the finite words $\om$ given as in ~\eqref{def.omegaxs}.
Note that $\mathfrak  G_{k+1}(\vep,(\om_0,x_0))\subset \mathfrak G_k(\vep,(\om_0,x_0))$ and 
$\mathfrak G_k(\vep,x_0)\subset \mathfrak F'_k(\vep,(\om_0,x_0))$ for every $k\geqslant 1$, $\vep>0$ and $(\om_0,x_0)\in D$. 
We also observe that $\mathfrak F'_{k+1}(\vep,(\om_0,x_0))\subset \mathfrak G_k(\vep,(\om_0,x_0))$ for every $k\geqslant 1$.  In fact, given 
$(\tilde \om,z)\in \mathfrak F'_{k+1}(\vep,(\om_0,x_0))$, there exists a $k$-uple $(\,\ud x_1, \dots, \ud x_{k+1})\in  S_1^{N_1} \times ... \times S_{k+1}^{N_{k+1}}$
and  a finite word $\om=\om(x_0,\,\ud x_1, \dots, \ud x_{k+1}) \in \{1,2, \dots, \kappa\}^{\mathbb N}$ determined by ~\eqref{def.omegaxs} in such a way that
$$
(\tilde \om,z) \in [\om] \times \ov{C( x_0\,\ud x_1 \cdots \ud x_{k+1})}
	\quad\text{and}\quad
f_{{\underline{\omega}}}^{\sum_{i=1}^{j}(N_i n_i+\sum_{j=1}^{N_i-1}p_j^i+P_{i-1})}(z) \in 
		\ov{B_{N_{j}}^{j}},
	\; \forall 1\leqslant j \leqslant k+1,
$$
where $B_{N_{j}}^{j}$ is a ball contained in $B_{f_\kappa}(x^j_{N_j}, n_j, \vep)$ as described above (one can replace $\om$ by $\tilde\om$ in the second
expression above because the cylinder $[\om]$ determines ${\sum_{i=1}^{j}(N_i n_i+\sum_{j=1}^{N_i-1}p_j^i+P_{i-1})}$ positive coordinates of an element
in $\Sigma_\kappa$). Furthermore, it is immediate from the construction that 
$$f_{\tilde {\underline{\omega}}}^{-(N_{k+1}n_{k+1}+\sum_{j=1}^{N_{k+1}-1}p^k_j+P_k)}(\ov B_{N_{k+1}}^{k+1})\subset B_{N_{k}}^{k}.
$$ 
Thus
$
f_{\tilde{\underline{\omega}}}^{\sum_{i=1}^{k}(N_in_i+\sum_{j=1}^{N_i-1}p_j^i+P_{i-1})}(z) \in  B_{N_{k}}^{k},
$
which proves that $(\tilde \om,z)\in \mathfrak G_k$.
Altogether, we conclude that
$$
\mathfrak F
	=\bigcap_{k = 1}^\infty\; \Big[ \bigcup_{j=\lfloor\frac1{\vep_0} \rfloor}^{\infty} \bigcup_{(\om_0,x_0)\in D} \mathfrak G_k(\frac{1}{j},(\om_0,x_0))\Big]
$$  
is obtainable as countable intersection of open sets, hence it is a $G_{\delta}$-set. This finishes the proof of the lemma.
\end{proof}

Finally, in order to complete the proof of Proposition~\ref{mainpropC} it remains to prove that all points in $\mathfrak F$ are Birkhoff irregular
for the skew-product $F$ and the continuous observable $\varphi_\psi: \Sigma_\kappa\times X \to \mathbb R$. More precisely:

\begin{lemma}\label{lemma-point-nontypical}
For any $(\om,x)\in \mathfrak F$ it holds that 
$$
\liminf_{n\to\infty} \frac{1}{n}\sum_{j=0}^{n-1}\psi(f^j_{{\underline{\omega}}}(x))
	< \limsup_{n\to\infty} \frac{1}{n}\sum_{j=0}^{n-1}\psi(f^j_{{\underline{\omega}}}(x)).
$$ 

\end{lemma}
	
\begin{proof}
By construction, for each $(\om,x)\in \mathfrak F$ and $k\geqslant 1$ there are $j\geqslant 1$,
$(\om_0,x_0)\in D$ and $(\,\ud x_1, \dots, \ud x_k)\in S_1^{N_1} \times ... \times S_k^{N_k}$
so that the first entries of the word ${\underline{\omega}}$  are of the form 
$(\,
	{\underline{\omega}}^{(1)}(\ud x_1) 
	,\, 
	\om^{(12)}, 
	\,
	{\underline{\omega}}^{(2)}(\ud x_2) 
	,\, 
	\om^{(23)}, 
	\,
	\dots,
	\, 
	{\underline{\omega}}^{(kk+1)}
	\,)$
and 
$$
f_{{\underline{\omega}}}^{(\ell n_i +\sum_{s=1}^{\ell-1} p_s^j+P_{i-1}) +\sum_{j=1}^{i-1} (N_j n_j+\sum_{s=1}^{N_j-1} p_s^j+P_{j-1})}(x) \in 
		\ov{B^i_{\ell}}
	\qquad \; \forall 1\leqslant \ell \leqslant N_i \quad
	\forall 1\leqslant i \leqslant k,
$$ 
where
	$B^{i}_{\ell}$ is a ball contained in the dynamic ball $B_{f_\kappa}(x^i_\ell, n_i, \frac{1}{j})$
	as described in the previous subsection.
In particular, for each $k\geqslant 1$ the iterate $y_k:=f_{\underline{\omega}}^{\sum_{i=1}^{k} (N_i n_i+\sum_{j=1}^{N_i-1}p_j^i+P_{i-1})}(x)$ belongs to {$B_{f_\kappa}(x_{1}^{k+1},n_{k+1},\frac1{j})$}.
By triangular inequality and the choice of constants in ~\eqref{eq:choice-vep} and ~\eqref{eq:N_k-def}, if $k\geqslant 1$ is even and large then
\begin{align*}
\Big| \sum_{j=0}^{n_{k+1}+\sum_{i=1}^{k} (N_i n_i+\sum_{j=1}^{N_i-1}p_j^i+P_{i-1})} & \psi(f_{\underline{\omega}}^j(x)) 
		 - n_{k+1} \, \int \psi \, d\mu_1 \Big| \\
		& \leqslant 
		\Big| \sum_{j=0}^{n_{k+1}+\sum_{i=1}^{k} (N_i n_i+\sum_{j=1}^{N_i-1}p_j^i+P_{i-1})} 
			\psi(f_{\underline{\omega}}^j(x)) - \sum_{j=0}^{n_{k+1}} \psi(f_\kappa^j(y_k)) \Big| \\
		& + \Big| \sum_{j=0}^{n_{k+1}} \psi(f_\kappa^j(y_k)) - n_{k+1} \, \int \psi \, d\mu_1 \Big|\\
		&\leqslant  \|\psi\|_\infty\sum_{i=1}^{k} (N_i n_i+\sum_{j=1}^{N_i-1}p_j^i+P_{i-1})  \\
		&{+ n_{k+1} \frac{\int \psi \,d\mu_2-\int \psi \,d\mu_1}{4} } \\
		& < 
		 \Big[\int \psi\, d\mu_1 + \frac{\int \psi \,d\mu_2-\int \psi \,d\mu_1}{3}\Big]
		 \; \Big(n_{k+1}+\sum_{i=1}^{k} (N_i n_i+\sum_{j=1}^{N_i-1}p_j^i+P_{i-1})\Big).
\end{align*}
Analogously, we get that
\begin{align*}
\frac1{n_{k}+\sum_{i=1}^{k} (N_i n_i+\sum_{j=1}^{N_i-1}p_j^i+P_{i-1})}\sum_{j=0}^{n_{k}+\sum_{i=1}^{k} (N_i n_i+\sum_{j=1}^{N_i-1}p_j^i+P_{i-1}) } \psi(f_{\underline{\omega}}^j(x))
			> \int \psi\, d\mu_2 - \frac{\int \psi \,d\mu_2-\int \psi \,d\mu_1}{3} 
\end{align*}
for every large and odd integer $k\geqslant 1$. This implies that  the Birkhoff averages $\big(\frac1n\sum_{j=0}^{n-1} \psi(f_{\underline{\omega}}^j(x)\big)_{n\geqslant 1}$ 
		do not converge, proving the lemma. 
	\end{proof}

\subsection{Proof of Theorem~\ref{thm:IFS}}\label{subsec3}
The proof of the theorem is now a direct consequence of the results presented in the previous sections. Indeed, 
one knows that $I_F(\varphi_\psi)$ is a Baire generic subset of $\{1,2,\dots, \kappa\}^{\mathbb Z} \times X$. Recalling that 
$\pi: \{1,2,\dots, \kappa\}^{\mathbb Z} \times X \to \{1,2,\dots, \kappa\}^{\mathbb Z}$ denotes the projection on the first coordinate, 
it follows from the Kuratowski-Ulam theorem (cf. \cite{KU}) that the set
$$
\Big\{ \om \in \Sigma_\kappa\colon I_{F}(\varphi_\psi) \cap \pi^{-1}(\om)\; \text{is a Baire residual subset of}\,\{\om\}\times X \Big\}
$$ 
is a Baire residual subset of $\Sigma_\kappa$. As $I_{F}(\varphi_\psi) \cap \pi^{-1}(\om)=\{\om\}\times I_{\mathcal F_\om}(\psi)$, this proves the 
first statement in the theorem.
The second one is a direct consequence of the inclusion \eqref{eq:inclus-irr} and the first statement. 
This completes the proof of the theorem.

\section{{Topological complexity on the skew-product}}\label{top-skew}

In this section we prove Theorem~\ref{thm:B}. 


\subsection{Entropy estimates obtained from ergodic measures}

Let $X$ be a compact metric space, $G_1=\{id,f_1, f_2, \dots, f_\kappa\}$ be a collection of bi-Lipschitz 
homeomorphisms on $X$, $\sigma: \{1,2,  \dots, \kappa\}^{\mathbb Z} \to \{1,2,  \dots, \kappa\}^{\mathbb Z}$ be the full shift and $F: \{1,2,\dots,\kappa\}^{\Z}\times X \to \{1,2,\dots,\kappa\}^{\Z}\times X$ be the skew product 
$$
F(\om,x)=(\sigma(\om),f_{\omega_0}(x)),
	\quad\text{for every} \; (\om, x) \in \{1,2,\dots,\kappa\}^{\Z}\times X.
$$ 
Assume that the semigroup action generated by $G_1$ has frequent hitting times. In rough terms, 
this ensures that $F$ has frequent hitting times along the fibers. Let 
$\varphi:\{1,2,\dots,\kappa\}^{\Z}\times X \to \R $ be a continuous observable. 
\medskip
{Assume that  $h_*(\varphi)$, given by ~\eqref{hestrela}, is strictly positive (otherwise there is nothing to prove).
Given an arbitrary and fixed $\gamma\in(0,1)$, let $\mu_1,\mu_2$ be $F$-invariant and ergodic probabilities
such that $\int \varphi \,d\mu_1< \int \varphi \,d\mu_2$ and $h_{\mu_i}(F) > h_*(\varphi)-\gamma$}.
Now, consider a strictly decreasing sequence of positive numbers $(\delta_k)_{k\geqslant 1}$ tending to zero so that 
$
\delta_1 < \frac14 (\int \varphi \,d\mu_2-\int \varphi \,d\mu_1)
$
and, associated to these, pick a strictly increasing sequence of large positive integers 
$(\mathfrak n_k)_{k\geqslant 1}$ so that,
for each $\ell\geqslant 1$, the set
$$
\Gamma_{2\ell-1}:=
\Big\{(\om,x)\in \{1,2,...,\kappa\}^{\Z} \times X: \Big|\frac{1}{n}\sum_{j=0}^{n-1}\varphi(F^j(\om,x))-\int \varphi \,d\mu_1\Big|< \delta_{2\ell-1},\;
\; \forall n\geqslant \mathfrak n_{2\ell-1} \Big\}
$$ 
has $\mu_1$-measure larger than $\frac12$
and
$$
\Gamma_{2\ell}:=
\Big\{
(\om,x)\in \{1,2,...,\kappa\}^{\Z} \times X:
	\Big| \frac{1}{n}\sum_{j=0}^{n-1}\varphi(F^j(\om,x)) -\int \varphi \,d\mu_2\Big|< \delta_{2\ell},\;
\; \forall n\geqslant \mathfrak n_{2\ell} \Big\}
$$
has $\mu_2$-measure larger than $\frac12$.
This is possible by ergodicity of the measures 
$\mu_1$ and $\mu_2$ together with the ergodic theorem and Lusin's theorem. 
By uniform continuity of $\varphi$ and Katok's entropy formula (recall Subsection~\ref{subsec:entropies}), 
there exists $\vep>0$ such that 
\begin{eqnarray}\label{eq-unifcont}
|\varphi(\omega,x)-\varphi(\omega',x')|< \frac{\int \varphi \,d\mu_2-\int \varphi \,d\mu_1}{4}
\end{eqnarray}
for every $(\omega,x),(\omega',x')\in \{1,2,...,\kappa\}^{\Z}\times X$ so that $d((\omega,x),(\omega',x'))<4\vep$
and, for each $i=1,2$,
\begin{eqnarray}\label{eq entropy-Katok-teoE}
\underline{h}_{\mu_i}(F,8\vep,\frac12)=\liminf_{n\rightarrow +\infty}\frac{1}{n}\log N^{\mu_i}_n(8\vep,\frac12) 
		> h_*(\varphi) - 2\gamma.
\end{eqnarray}
This constant $\vep>0$ will be fixed throughout the remainder of the argument.

\medskip

The next lemma allows one to select a special class of dynamically separated points inside each of the sets $\Gamma_\ell$.
More precisely:

\begin{lemma}\label{le:K-entropy}
There exists a sequence $(n_\ell)_{\ell\geqslant 1}$ of positive integers
such that: 
\begin{enumerate}
\item $n_\ell \geqslant \max\{\; \mathfrak n_\ell, \; (\ell+2) \cdot n_{\ell-1} + \Big[ \frac{2\log 2}{\log L} {+\log \vep^{-1}} \Big], \; \ell \cdot K(L^{-2n_{\ell-1}}\,L^{\log (\vep^{-1})^{-1}} \,2\vep)\;\}$ for every $\ell\geqslant 2$
	\medskip
\item there exists an $(n_\ell, 4\vep)$-separated set  $\mathcal S_\ell\subset \Gamma_\ell$ (with respect to $F$) 
of cardinality 
$$
\# \mathcal S_\ell \geqslant \exp(\,(h_*(\varphi) -3\gamma) n_\ell\,)
	\quad\text{for every $\ell\geqslant 1$.}
$$
\end{enumerate}
\end{lemma}

\begin{proof}
The proof is by induction on $\ell$. Set $\ell=1$. Since $\mu_1(\Gamma_1)\geqslant \frac12$, 
inequality~\eqref{eq entropy-Katok-teoE} implies that the minimum cardinality of a $(n,8\vep)$-generating set
for $\Gamma_1$ is bounded below by $N^{\mu_i}_n(8\vep,\frac12)$. Pick $n_1\geqslant \mathfrak n_1$
so that 
$$
N^{\mu_1}_n(8\vep,\frac12) \geqslant \exp(\,(h_*(\varphi) -3\gamma) n\,),
\quad \forall n\geqslant n_1.
$$
Recall that any $(n,8\vep)$-generating set is a $(n,4\vep)$-separated set. Hence, we deduce that there exists an 
$(n_1,4\vep)$-separated set $\mathcal S_1\subset \Gamma_1$ with cardinality 
 $\#\mathcal S_1 \geqslant e^{\,(h_*(\varphi) -3\gamma) n_1\,}$.

\smallskip
Now, assume that the lemma is proven for every $1\leqslant \ell < k$. If $\ell=k$ then pick a large integer
$$
n_\ell \geqslant \max\{\; \mathfrak n_\ell, \; (\ell+2) \cdot n_{\ell-1} + \Big[ \frac{2\log 2}{\log L} {+\log \vep^{-1}} \Big], \; \ell \cdot K(L^{-2n_{\ell-1}}\,L^{\log (\vep^{-1})^{-1}} \,2\vep)\;\}
$$
so that 
$$
N^{\mu_{\ell ({\tiny\mbox{mod}} 1)}}_n(8\vep,\frac12) \geqslant \exp(\,(h_*(\varphi) -3\gamma) n\,),
\quad \forall n\geqslant n_\ell.
$$
Then, taking $n=n_\ell$ above we conclude that 
that there exists an $(n_\ell,4\vep)$-separated set $\mathcal S_\ell\subset \Gamma_\ell$ with cardinality 
 $\#\mathcal S_\ell \geqslant e^{\,(h_*(\varphi) -3\gamma) n_\ell\,}$, as claimed. This proves the induction assumption and completes the proof of the lemma.
\end{proof}

By construction, the sequence $(n_\ell)_{\ell\geqslant 1}$ is strictly increasing and it satisfies
$$
\lim_{\ell\to+\infty} \frac{n_{\ell-1}}{n_{\ell}}=0
	\quad\text{and}\quad
\lim_{\ell\to+\infty} \frac{K(L^{-2n_{\ell-1}}\,L^{\log (\vep^{-1})^{-1}} \,2\vep)}{n_{\ell}}=0.	
$$
In particular this ensures that
The second expression reflects that the size of the orbits in $\Gamma_{\ell}$ are much larger than the largest possible transition time
necessary for shadowing pieces of orbits in $\Gamma_{\ell-1}$ and $\Gamma_{\ell}$ consecutively. 

\subsection{A Moran set of irregular points in $\{1,2,...,\kappa\}^{\Z}\times X$ with large entropy}\label{MLEntropy}

The main goal of this subsection is the construction of a Moran set, contained in the set $I_{F}(\varphi)$
of irregular in $\Sigma_\kappa\times X$, having large upper topological entropy.

\medskip
\noindent\textbf{$1^{st}$ Step:} \emph{Understanding dynamic balls for points in $\mathcal S_\ell$}
\smallskip

Fix $\ell\geqslant 1$,
and let $B_F((\om,x), n_\ell,\vep)\subset \Sigma_\kappa\times X$ be the dynamic ball with respect to $F$, centered at $(\om,x)$, of radius $\vep$ and 
length $n_\ell$.
Given $\om\in \Sigma_\kappa$ and integers $m \leqslant n$ we denote by 
$$
\om_{[m,n]} 
	:=\big\{ \, 
	\tilde \om \in \Sigma_\kappa \colon \om_j=\widetilde{\om}_j, \, \forall j\in [m,n]
	\,\big\}
$$
the cylinder set determined by the infinite word $\om$ at the positions in the interval $[m,n]$.
We need the following auxiliary lemma.

\begin{lemma}\label{estimates-dyn-balls-F}
For any $\vep>0$, $(\om,x)\in \Sigma_\kappa\times X$ and $\ell\geqslant 1$ one has that
$$
B_F((\om,x), n_\ell,\vep) = B_\sigma(\om, n_\ell, \vep) \times B_{\om}(x,n_\ell,\vep)
$$
where
$B_\sigma(\om, n_\ell, \vep)$
is the cylinder set $\om_{\,[-\lfloor\log(\vep^{-1})\rfloor\,,\,n_\ell + \lfloor\log(\vep^{-1})\rfloor]}$
and 
$$
B_{\om}(x,n_\ell,\vep):=\big\{y\in X \colon d_X(g_\om^j(x), g_\om^j(y))<\vep\, \text{ for every }\, 0\leqslant j \leqslant n_\ell \big\}.
$$
In particular,
\begin{enumerate}
\item \label{eq:images0} $F^{n_\ell}\big(\, B_F((\om,x), n_\ell,\vep)\, \big)
	 \supseteq 
	 \sigma^{n_\ell} \big( \om_{\,[-n_\ell+\lfloor \log(\vep^{-1})\rfloor,\,\lfloor -\log(\vep^{-1})\rfloor]} \big)
	 \times B(f^{n_\ell}_\om(x), L^{-2n_\ell}\vep); 
$
\item \label{eq:images}
$
F^{n_\ell+\lfloor \log(\vep^{-1}) \rfloor}  \big(\, B_F((\om,x), n_\ell,\vep)\, \big) 
$
\item[] \qquad \quad 
$
 \supseteq 
 	 \sigma^{n_\ell+\lfloor \log(\vep^{-1}) \rfloor} \big( \om_{\,[-n_\ell+\lfloor \log(\vep^{-1})\rfloor,\,\lfloor -\log(\vep^{-1})\rfloor]} \big)
	 \times B(f^{n_\ell+\lfloor \log(\vep^{-1}) \rfloor}_\om(x), L^{-(2{n_\ell}+\lfloor \log(\vep^{-1}) \rfloor)}\vep)
$
\end{enumerate}
\end{lemma}

\begin{proof}
The first claim is an immediate consequence of the fact that the step skew-products are constant on cylinders of size one in the shift, and that we are taking the maximum metric.
Then, item ~\eqref{eq:images0} follows from the definition and  Lemma~\ref{le:est-balls}, which ensures that 
$f_\omega^n(B_{\om}(x,n,\vep))$ contains the ball of radius $L^{-2n}\vep$ around $f_\omega^n(x)$.
Finally, the inclusion in item \eqref{eq:images} is obtained from item \eqref{eq:images0} using the same argument as above and 
iteration by $F^{\lfloor \log(\vep^{-1}) \rfloor}$.
\end{proof}

Comparing the expressions appearing in items~\eqref{eq:images0} and ~\eqref{eq:images} in the lemma, 
we conclude that any point $(\tilde \om,\tilde x)$ which $\vep$-shadows the piece of orbit of a point $(\om,x)$ during $n_\ell$ 
iterates is such $\sigma^n(\tilde \om)$ belongs to a cylinder of deepth $\lfloor \log(\vep^{-1}) \rfloor$.  As a consequence, one needs to make $\lfloor \log(\vep^{-1}) \rfloor$ further iterations to drop this dependence (cf. item~\eqref{eq:images} 
in Lemma~\ref{estimates-dyn-balls-F} above).

\medskip
\noindent\textbf{$2^{nd}$ Step:} 
\emph{Construction of a Moran set of points with oscilating behavior}
\medskip

Given $(\om_1,x_1)\in \mathcal S_1$ and $(\om_2,x_2)\in \mathcal S_2$, Lemma~\ref{estimates-dyn-balls-F} \eqref{eq:images}  implies that
$F^{n_1+\lfloor \log(\vep^{-1}) \rfloor}  \big(\, B_F((\om_1,x), n_1,\vep)\, \big)$
contains the set
 $$
  \sigma^{n_1+\lfloor \log(\vep^{-1}) \rfloor} \big(\om_{\,[-\lfloor\log(\vep^{-1})\rfloor\,,\,n_1 + \lfloor\log(\vep^{-1})\rfloor]} \big)
	 \times B(f^{n_1+\lfloor \log(\vep^{-1}) \rfloor}_\om(x), L^{-(2{n_1}+\lfloor \log(\vep^{-1}) \rfloor)}\vep).
$$
On the other hand, $B_{\om_2}(x_2,n_2,\vep)$ contains the ball centered at $x_2$ with radius $L^{-n_2}\vep$.
Then, the frequent hitting times property along the fiber of $F$ ensures that there exists 
$0\leqslant p_{1} \leqslant K(L^{-(2{n_1}+\lfloor \log(\vep^{-1}) \rfloor)}\vep)$ and a finite word 
$\om^{(12)}\in \{1,2,\dots, \kappa\}^p_1$ 
so that 
$$
f^{p_1}_{\om^{(12)}} (B(f^{n_1+\lfloor \log(\vep^{-1}) \rfloor}_\om(x), L^{-(2{n_1}+\lfloor \log(\vep^{-1}) \rfloor)}\vep))
\cap B_{\om_2}(x_2,n_2,\vep)
$$
contains a ball $B_1^2$ of radius $\frac12 L^{-n_2}\varepsilon$.
Thus
\begin{align*}
C(\,(\om_1,x_1),(\om_2,x_2)\,)
	& :=
	{\om_1}_{\,[-\log(\vep^{-1})\,,\,n_1 + \log(\vep^{-1})]} \times B_{\om_1}(x_1,n_1,\vep)
	\\
	& \; \cap \;
	F^{-(p_1+\lfloor \log(\vep^{-1}) \rfloor+n_1)} \, \Big( {\om_2}_{\,[-\log(\vep^{-1})\,,\,n_2 + \log(\vep^{-1})]} \times B_1^2
	\Big)
\end{align*}
is a compact and non-empty subset of $\Sigma_\kappa\times X$.
Observe now that $f_{\om_2}^{n_2+\lfloor \log(\vep^{-1}) \rfloor}(B_1^2)$ contains a ball of radius 
$\frac12 L^{-(2n_2+\lfloor -\log(\vep^{-1}\rfloor))}\varepsilon \geqslant L^{-n_3}\vep$. 
Thus, if $(\om_3,x_3)\in \mathcal S_3$, repeating the argument, 
there exists $0\leqslant p_{2} \leqslant K(L^{-n_3}\varepsilon)$, 
a finite word $\om^{(23)} \in \{1,2,\dots, \kappa\}^{p_2}$ and a ball 
$B_2^3\subset B_{\om_3}(x_3,n_3,\vep)$ of radius $\frac12 L^{-n_3}\varepsilon$
such that 
\begin{align*}
C(\,(\om_1,x_1),(\om_2,x_2),(\om_3,x_3)\,)
	& :=
	{\om_1}_{\,[-\log(\vep^{-1})\,,\,n_1 + \log(\vep^{-1})]} \times B_{\om_1}(x_1,n_1,\vep)
	\\
	& \; \cap \;
	F^{-(p_1+\lfloor \log(\vep^{-1}) \rfloor+n_1)} \, \Big( {\om_2}_{\,[-\log(\vep^{-1})\,,\,n_2 + \log(\vep^{-1})]} \times B_1^2
	\Big)
	\\
	& \; \cap \;
	F^{-\sum_{i=1}^2 (p_i+\lfloor \log(\vep^{-1}) \rfloor+n_i)} \, \Big( {\om_3}_{\,[-\log(\vep^{-1})\,,\,n_3 + \log(\vep^{-1})]} \times B_2^3
	\Big)
\end{align*}
is non-empty subset of $C(\,(\om_1,x_1),(\om_2,x_2)\,)$.
More generally, given $k\geqslant 1$ and a $k$-uple $((\om_1,x_1), \dots, (\om_k,x_k)) \in \mathcal S_1 \times \dots \times \mathcal S_k$, for each $1\leqslant i\leqslant k$ there exist 
$0\leqslant p_{i} \leqslant K(L^{-n_i} \vep)$, a finite word $\om^{(i (i+1))} \in \{1,2,\dots, \kappa\}^{p_i}$ 
and a ball $B_i^{i+1} \subset B_{\om_{i+1}}(x_{i+1},n_{i+1},\vep)$ of radius $\frac12 L^{-n_{i+1}}\varepsilon$ so that 
\begin{align*}
C(\,(\om_1,x_1), \dots, (\om_k,x_k)\,)
	& :=
	{\om_1}_{\,[-\log(\vep^{-1})\,,\,n_1 + \log(\vep^{-1})]} \times B_{\om_1}(x_1,n_1,\vep)
	\\
	& \; \cap \;
	F^{-(p_1+\lfloor \log(\vep^{-1}) \rfloor+n_1)} \, \Big( {\om_2}_{\,[-\log(\vep^{-1})\,,\,n_2 + \log(\vep^{-1})]} \times B_1^2
	\Big)
	\\
	& \quad \dots 
	\\
	& \; \cap \;
	F^{-\sum_{i=1}^{k-1} (p_i+\lfloor \log(\vep^{-1}) \rfloor+n_i)} \, \Big( {\om_k}_{\,[-\log(\vep^{-1})\,,\,n_k + \log(\vep^{-1})]} \times B_{k-1}^k
	\Big)
\end{align*}
is non-empty. By construction, the previous family of cylinder sets satisfy
$$
C(\,(\om_1,x_1), \dots, (\om_k,x_k),  (\om_{k+1},x_{k+1})\,)
\subset 
C(\,(\om_1,x_1), \dots, (\om_k,x_k)\,)
$$
for every
$((\om_1,x_1), \dots, (\om_k,x_k),(\om_{k+1},x_{k+1})) \in \mathcal S_1 \times \dots \times \mathcal S_k 
\times \mathcal S_{k+1}$ and $k\geqslant 1$.

\medskip
For the sake of completeness, let us emphasize that the transition times $p_i$ are functions of the collections 
$(\,(\om_1,x_1), \dots, (\om_k,x_k))$. Moreover, since transition times are not constant, one cannot ensure that points
in different cylinder sets of 
$$
 \mathcal{\tilde C}_k:=
	\big\{ 
	C(\,(\om_1,x_1), \dots, (\om_k,x_k)\,) \colon ((\om_1,x_1), \dots, (\om_k,x_k)) \in \mathcal S_1 \times \dots \times \mathcal S_k
	\big\}
$$
are separated by some suitable iterate of $F$. We proceed to show this is the case, at least for a relevant proportion of cylinder sets.
Let us be more precise.
For each $1\leqslant i \leqslant  k-1$ consider the $i^{\text{th}}$-transition time function
$$
\begin{array}{rccc}
p_i : &  \mathcal{\tilde C}_{k} & \to & \{0,1,\dots, K(L^{-(2{n_{i}}+\lfloor \log(\vep^{-1}) \rfloor)}\vep)\}.
\end{array}
$$
By the pigeonhole principle, for each $1\leqslant i \leqslant  k-1$ there exists 
$0\leqslant t_i\leqslant K(L^{-(2{n_i}+\lfloor \log(\vep^{-1}) \rfloor)}\vep)$ and a collection $ \mathcal{C}_k \subset  \mathcal{\tilde C}_k$ such that 
\begin{equation}\label{eq:card-Ck}
\#  {\mathcal C_k} 
	\geqslant \frac{\# \mathcal S_k \, \dots\, \# \mathcal S_2\, \# \mathcal S_1}{\prod_{i=1}^{k-1} \, 
	K(L^{-(2{n_i}+\lfloor \log(\vep^{-1}) \rfloor)}\vep)}
\end{equation}
and $p_i(C)=t_i$ for every $C\in  {\mathcal C_k}$.

\begin{lemma}\label{le:selection}
For each $k\geqslant 1$ let the subcollection ${\mathcal C}_k\subset \tilde{\mathcal C}_k$ and the integers
$(t_i)_{1\leqslant i \leqslant  k-1}$ be given as above. There exists $K_0>0$ 
such that:
\begin{enumerate}
\item for any distinct $C_1, C_2 \in {\mathcal C}_k$,\, $(\om,x)\in C_1$ and $(\tilde \om,\tilde x)\in C_2$
it holds that $d_{\alpha_k}((\om, x),(\widetilde{\om},\tilde x))>\vep$,
\item $\#  {\mathcal C}_k \geqslant K_0 \, e^{  (h_*(\varphi) -4\gamma) \,\alpha_k }$
\end{enumerate}
where $\alpha_k=n_k + \sum_{i=1}^{k-1} (n_i+\lfloor \log(\vep^{-1}) \rfloor+t_i)$.
\end{lemma}

\begin{proof}
Fix $k\geqslant 1$. For any distinct $k$-uples  
$C_1= \tilde C(\,(\om_1,x_1), \dots, (\om_k,x_k)\,)$ 
and $C_2= \tilde C(\,(\om_1,\tilde x_1), \dots, (\tilde \om_k, \tilde x_k)\,)$
belonging to $\tilde {\mathcal C}_k$
there exists $1\leqslant j \leqslant k$ so that $(\om_j,x_j)\neq (\tilde \om_j,\tilde x_j)$.
Thus, if $(\om,x)\in C_1$ and $(\widetilde{\om},\tilde x)\in C_2$ then
\begin{align*}
d(F^{\sum_{i=1}^{j} (n_i+ \lfloor \log (\vep^{-1})^{-1}\rfloor +t_i) }(\om,x), 
	& F^{\sum_{i=1}^{j} (n_i+ \lfloor \log (\vep^{-1})^{-1}\rfloor +t_i )}(\tilde \om,\tilde x)) 
	\\
	& 
	\geqslant d((\om_j,x_j), (\tilde \om_j,\tilde x_j)) -
	d(F^{\sum_{i=1}^{j} (n_i+ \lfloor \log (\vep^{-1})^{-1}\rfloor +t_i )}(\om,x),  (\om_j,x_j))
	\\
	&- d(F^{\sum_{i=1}^{j} (n_i+ \lfloor \log (\vep^{-1})^{-1}\rfloor +t_i )}(\tilde \om,\tilde x)),   (\tilde \om_j, \tilde x_j))
	>\vep,
\end{align*}
thus proving the first item in the lemma.
The proof of the second item follows from quantitative information concerning ~\eqref{eq:card-Ck}. Indeed, using Lemma~\ref{le:K-entropy}(1) 
one has that  ${n_{i+1}} \geqslant {K(L^{-(2{n_i}+\lfloor \log(\vep^{-1}) \rfloor)}\vep)}$ for every $1\leqslant i\leqslant k-1$
and
\begin{align*} 
\frac1{n_{i+1}} \; 
	\; e^{ - (h_*(\varphi) -3\gamma) \, [\lfloor \log(\vep^{-1}) \rfloor + K(L^{-(2{n_i}+\lfloor \log(\vep^{-1}) \rfloor)}\vep)]}
	\geqslant e^{-\gamma [n_{i+1} + \lfloor \log(\vep^{-1}) \rfloor +t_i]}
\end{align*}
for every large $i\geqslant 1$.
Together with Lemma~\ref{le:K-entropy}(2), this ensures that there exists $K_0>0$ so that 
\begin{align*}
& \frac{\# \mathcal S_k \, \dots\, \# \mathcal S_2\, \# \mathcal S_1}{\prod_{i=1}^{k-1} \, K(L^{-(2{n_i}+\lfloor \log(\vep^{-1}) \rfloor)}\vep)} \\
	& \geqslant 
	e^{  (h_*(\varphi) -3\gamma) \, [n_k+ \dots + n_2 + n_1]}
	\prod_{i=1}^{k-1}  \frac1{n_{i+1}} \, \frac{n_{i+1}}{K(L^{-(2{n_i}+\lfloor \log(\vep^{-1}) \rfloor)}\vep)} \\
	& \geqslant 
	\prod_{i=0}^{k-1}  \frac1{n_{i+1}} \; e^{  (h_*(\varphi) -3\gamma) \, [n_{i+1} + \lfloor \log(\vep^{-1}) \rfloor +t_i] }
	\; e^{ - (h_*(\varphi) -3\gamma) \, [\lfloor \log(\vep^{-1}) \rfloor + K(L^{-(2{n_i}+\lfloor \log(\vep^{-1}) \rfloor)}\vep)]}\\
	& \geqslant K_0 \, e^{  (h_*(\varphi) -4\gamma) \,\alpha_k }
\end{align*}
for every $k\geqslant 1$,
as claimed.
\end{proof}

Now, consider the Moran set $\mathfrak F$, obtained as a countable intersection of
compact and nested subsets, defined by
$
\mathfrak F:=\bigcap_{k\geqslant 1} \mathfrak F_k
$
where
$$
\mathfrak F_k:= 
	\bigcup
	\Big\{\overline{C(\,(\om_1,x_1), \dots, (\om_k,x_k)\,)}: C(\,(\om_1,x_1), \dots, (\om_k,x_k)\,)\in  {\mathcal C}_k \Big\}.
$$
We now observe that the compact and non-empty set $\mathfrak F$ is contained in $I_F(\varphi)$.

\begin{lemma}\label{irregular-teoE}
For any $(\om,x) \in \mathfrak F$ we have
	$$
	\liminf_{n\to\infty} \frac{1}{n}\sum_{j=0}^{n-1}\varphi(F^j(\om,x))
	< \limsup_{n\to\infty} \frac{1}{n}\sum_{j=0}^{n-1}\varphi(F^j(\om,x)).
	$$ 
\end{lemma}

\begin{proof}
This proof is identical to the proof of Lemma~\ref{lemma-point-nontypical}, and for that reason we shall omit it.  
\end{proof}

\medskip
\noindent\textbf{$3^{rd}$ Step:} 
\emph{Estimating the topological entropy of $\mathfrak{F}$}
\medskip

In order to estimate the topological entropy of the Moran set $\mathfrak{F}$ for the skew-product $F$ we will make use of the generalized 
entropy distribution principle stated as Theorem~\ref{thm:EDP}. For that, we define a special sequence of probability measures $\mu_k$ 
supported in points selected among the elements of sets in the family $\mathcal{C}_k$ as follows:
for each $C\in \mathcal{C}_k$ choose a point $z_C\in C$
and consider 
the probability measure
$$
\mu_k=\frac{1}{\#\mathcal{C}_k}
\sum_{ C\in \mathcal{C}_k}\delta_{z_C}.
$$

\begin{lemma}\label{le:acc}
If $\lim_{k\to\infty}\mu_k=\mu$ (in the weak$^*$-topology) then $\mu(\mathfrak{F})=1$.
\end{lemma}

\begin{proof}
	Suppose $\lim_{k\to\infty}\mu_k=\mu$ in the weak$^*$-topology. For each $\ell>0$ and $i>0$ we have  $\mu_{\ell+i}(\mathfrak{F}_{\ell+i})=1$ and as 
	$\mathfrak{F}_{\ell+i}\subset \mathfrak{F}_\ell$ then $\mu_{\ell+i}(\mathfrak{F}_{\ell})=1$. Since $\mathfrak{F}_\ell$ is a closed set, we have
	$\mu(\mathfrak{F}_\ell)\geqslant \limsup_{k\rightarrow \infty}\mu_k(\mathfrak{F}_\ell)=1$. Therefore, $\mu(\mathfrak{F})=\lim_{\ell\to\infty}\mu(\mathfrak{F}_\ell)=1$.
\end{proof}

In order to apply the generalized entropy distribution principle, one needs to estimate the measure of dynamic balls intersecting the Moran set. 
Let $\vep>0$ be small and $\mathbf{n}\geqslant 1$ be fixed and large, satisfying
\begin{equation}\label{nbig}
n_j \geqslant j \cdot K(L^{-2n_{j-1} + \lfloor \log(\vep^{-1}) \rfloor}\,2\vep) > \frac1{\gamma} \log K(L^{-2n_{j-1} + \lfloor \log(\vep^{-1}) \rfloor}\,2\vep)
\end{equation}
for every $j\geqslant \mathbf{n}$ (recall Lemma~\ref{le:K-entropy}). 
Pick also the unique $\ell=\ell(\mathbf{n})\geqslant 1$ so that 
$\alpha_{\ell-1} \leqslant \mathbf{n} <\alpha_{\ell}$, where the sequence $\alpha_\ell$ was defined in Lemma~\ref{le:selection}.
Consider an arbitrary dynamic ball $B_F((\om,x),\mathbf{n},\vep)\subset \Sigma_\kappa \times X$ 
which intersects the set $\mathfrak{F}$.

\begin{lemma}\label{lemma:estimative ball entropy}
 $\mu_{k+\ell}(B_F((\om,x),\mathbf{n},\vep))\leqslant e^{-\mathbf{n}(h_*(\varphi)-5\gamma)}$ for every large $k\geqslant 1$.
\end{lemma}

\begin{proof}
If $\mu_{k+\ell}(B_F((\om,x),\mathbf{n},\vep))=0$ there is nothing to prove. Otherwise $\mu_{k+\ell}(B_F((\om,x),\mathbf{n},\vep))>0$ and, using that
the probability $\mu_{k+\ell}$ is an atomic probability supported on points of $\mathfrak{F}_{k+\ell}$, we conclude 
that $\mathfrak{F}_{k+\ell}\cap B_F((\om,x),\mathbf{n},\vep)\neq \emptyset$.
We claim that all points in such intersection
belong to a single fixed cylinder  $C_*=C(\,(\om_1,x_1), \dots, (\om_{\ell},x_{\ell})\,)\in \mathcal{C}_{\ell}$. 
In fact, if $(\om', x')\in B_F((\om,x),\mathbf{n},\vep)\cap C'$,  $(\widetilde{\om},\tilde x)\in B_F((\om,x),\mathbf{n},\vep)\cap \tilde C$ with $C', \tilde C\in \mathcal{C}_{\ell}$ and $C'\neq C$ then item $(1)$ in Lemma~\ref{le:selection} guarantees that $d_{\alpha_{\ell}}((\om, x),(\widetilde{\om},\tilde x))>\vep$. As $\alpha_{\ell} > \mathbf{n}$ this leads to a contradiction with the fact that both points belong to the same dynamic ball of radius $\vep$ and length $\mathbf{n}$.

\smallskip
These estimates show that the number of cylinder sets constructed at each step (which involve the selection of those having the same transition times) are close to the exponential rate provided by entropy, a fact which imposes constraints of size on the transition times in comparison to the finite orbit size. Let us  be more precise.
By construction and inequality ~\eqref{nbig}, for each fixed $C_\ell\in \mathcal C_\ell$ one has
$$
\# \{C\in \mathcal{C}_{\ell+1} \colon C\cap C_\ell\neq\emptyset \} 
	\geqslant \frac{\# \mathcal S_{\ell+1}}{K(L^{-2n_{\ell} + \lfloor \log(\vep^{-1}) \rfloor}\,2\vep)}
	> e^{  (h_*(\varphi) -3\gamma) \, n_{\ell+1}} \, e^{-\gamma \, n_{\ell+1}}
$$
provided that $\ell$ is large (which means large $\mathbf{n}$). Using the same argument recursively we conclude that 
$
\# \{C\in \mathcal{C}_{k+\ell} \colon C\cap C_\ell\neq\emptyset \} 
	>  e^{  (h_*(\varphi)-4\gamma) \, [n_{k+\ell} +  \dots + n_{k+1}]} 
$
for every large $k$. Thus
\begin{align}\label{lbound}
\# \{C\in \mathcal{C}_{k+\ell} \colon C\cap C_*\neq\emptyset \} 
	& \leqslant  \frac{\#\mathcal{C}_{k+\ell}}{\min\limits_{C_\ell\in \mathcal C_\ell} \;   \# \{C\in \mathcal{C}_{k+\ell} \colon C\cap C_\ell\neq\emptyset \} } 
		\nonumber \\
	& < e^{ - (h_*(\varphi)-4\gamma) \, [n_{k+\ell} +  \dots + n_{k+1}]} \; \#\mathcal{C}_{k+\ell}.
\end{align}
Therefore, 
\begin{align*}
\mu_{k+\ell}(B_F((\om,x),\mathbf{n},\vep))
	&
	=
	\displaystyle\frac{1}{\#\mathcal{C}_{k+\ell}} \sum_{ \substack{C\in \mathcal{C}_{k+\ell} \\ C\cap C_*\neq\emptyset} }
				\delta_{z_C}\left(B_F((\om,x),\mathbf{n},\vep) \right) \\
	&\leqslant
	\displaystyle\frac{\# \{C\in \mathcal{C}_{k+\ell} \colon C\cap C_*\neq\emptyset \}  }{\#\mathcal{C}_{k+\ell}} 
 < e^{ - (h_*(\varphi)-4\gamma) \, [n_{k+\ell} +  \dots + n_{k+1}]}.
	\end{align*}
Notice that
\begin{align*}
(h_*(\varphi)-4\gamma) \cdot \sum_{i=k+1}^{k+\ell} n_i  
	& \geqslant (h_*(\varphi)-4\gamma)\Big[n_{k+\ell}
		+\sum_{i=1}^{k+\ell-1}(n_i+ \lfloor \log (\vep^{-1})^{-1}\rfloor +t_i) 
		- \sum_{i=1}^{k+\ell-1}(\lfloor \log (\vep^{-1})^{-1}\rfloor +t_i) \Big] \\
	& =(h_*(\varphi)-4\gamma) \Big[ \alpha_{k+\ell}
		- \frac{\sum_{i=1}^{k+\ell-1}(\lfloor \log (\vep^{-1})^{-1}\rfloor +t_i)}{\alpha_{k+\ell}} 
		\Big]
		 \\
	& \geqslant (h_*(\varphi)-5\gamma)\, \alpha_{k+\ell}
	\geqslant (h_*(\varphi)-5\gamma)\, \mathbf{n}
\end{align*}
for every large $k\geqslant 1$. Altogether this ultimately implies that 
$\mu_{k+\ell}(B_F((\om,x),\mathbf{n},\vep)) \leqslant e^{ -\mathbf{n} \, (h_*(\varphi)-5\gamma)\,}$, which completes the proof of the lemma.
\end{proof}

We are now in a position to prove the theorem. Indeed, on the one hand any accumulation point of $(\mu_k)_k$ is supported on $\mathfrak F$, by Lemma~\ref{le:acc}. On the other hand, Lemma~\ref{lemma:estimative ball entropy} (taking the $\limsup$ as $\ell$ tends to infinity) 
ensures that 
$
\limsup_{k\rightarrow \infty}\mu_k(B_F((\om,x),\mathbf{n},\vep))
	\leqslant  e^{-\mathbf{n}(h_*(\varphi)-5\gamma)}
$. 
Then, using the generalized entropy distribution principle we conclude that $h_{top}(F, I_F(\varphi))\geqslant  h_{top}(F,\mathfrak{F})\geqslant h_*(\varphi)-5\gamma$.
As $\gamma>0$ was chosen arbitrary, this shows that $h_{top}(F,I_F(\varphi)) \geqslant h_*(\varphi)$, and completes the proof of Theorem \ref{thm:B}.

\medskip


\section{Topological entropy of irregular points for the semigroup action}\label{top-semigroups}

This subsection is devoted to the proof of Theorem~\ref{thm:Bb}.
Let $X$ be a compact metric space, assume that the semigroup action generated by the collection $G_1=\{id,f_1, f_2, \dots, f_\kappa\}$ of bi-Lipschitz homeomorphisms has frequent hitting times, and let $\psi: X\to \mathbb R$ 
be a continuous observable which is not a coboundary with respect to $G_1$.

\subsection{Ghys-Langevin-Walczak entropy of the set of irregular points}\label{top-fibered-2}

Here we prove item (1) in Theorem~\ref{thm:Bb}. 
It is immediate that 
$$
s((G,G_1),n,\vep) \geqslant 
	s(\om,n, \vep, \delta)
$$
for every $n\geqslant 1$, $\vep>0$, $\om\in\Sigma_\kappa$ and $0<\delta<1$. Our mild assumptions seem not to allow to construct
a single sequence $\tilde \om \in\Sigma_\kappa$ whose number of irregular points for the observable $\psi$ and the 
sequential dynamical system $\mathcal F_{\tilde\om}$
grows at least at the exponential rate of $H^{\text{Pinsker}}(\psi) -2\gamma$. Alternatively, we will produce sets $E_n\subset X$ so that any
two points are $(n,\vep)$-separated by some element $\underline g\in G_n$ and whose Birkhoff averages of $\psi$ determined by such 
finite paths in the semigroup $G$ do not converge. Let us formalize this strategy. 

\smallskip
Assume that $H^{\text{Pinsker}}(\psi)>0$ (otherwise there is nothing to prove) 
and note that $H^{Katok}(\psi)=H^{\text{Pinsker}}(\psi)>0$
(recall ~\eqref{eq:Pinsker-Katok}).
Hence, for any fixed $0<\gamma,\delta<1$ 
there exist $\mu_1,\mu_2 \in \mathcal M_{erg}(F)$ so that $\int \varphi_\psi \,d\mu_1 <\int \varphi_\psi \,d\mu_2$ 
and $h_{\mu_i}^{\delta}(F)\geqslant H^{\text{Pinsker}}(\psi) -2\gamma$, for $i=1,2$. 
Using ~\eqref{eq:def-d-Katok}, one knows that
\begin{equation}\label{selection-GLW}
\lim_{\vep \rightarrow 0}\liminf_{n\rightarrow \infty} \frac{1}{n}\log s(\om,n, \vep, \delta)
	\geqslant H^{\text{Pinsker}}(\psi) -2\gamma
\end{equation} 
for $\mu_i$-almost every $\om$, $i=1,2$. 
One can choose for each $i=1,2$ a sequence $\om_i\in \Sigma_\kappa$ in the ergodic basin of attraction 
of the $\sigma$-invariant probability measure $\pi_*\mu_i$ satisfying \eqref{selection-GLW}.
The latter, together with the uniform continuity of $\psi$, allow us to select $\vep>0$ so that
$
|\psi(x)-\psi(x')|
	< \frac{\int \varphi_\psi \,d\mu_2-\int \varphi_\psi \,d\mu_1}{4}
$
for every $x,x'\in X$ so that $d(x,x')<4\vep$, and
$$
\liminf_{n\rightarrow \infty} \frac{1}{n}\log s(\om_i,n, 4\vep, \delta)
	\geqslant H^{\text{Pinsker}}(\psi) -3\gamma
$$
for every $i=1,2$. Write $\om_\ell=\om_1$ when
$\ell\geqslant 1$ is odd and $\om_\ell=\om_2$ for an even $\ell\geqslant 1$. 

\begin{lemma}\label{le:K-entropy-fiber-GLW}
	There exists a sequence $(n_\ell)_{\ell\geqslant 1}$ of positive integers 
	such that: 
	\begin{enumerate}
		\item $n_\ell \geqslant \max\{\; \mathfrak n_\ell, \; (\ell+2) \cdot n_{\ell-1} + \Big[ \frac{2\log 2}{\log L} \Big], \; 
		\ell \cdot \prod_{i=1}^{\ell-1} K(L^{-2{n_i}}2\vep) \, \exp( K(L^{-2{n_i}}2\vep) \log \kappa )
		\;\}$ 
		for every $\ell\geqslant 2$
		\medskip
		\item for each  $\ell\geqslant 1$ there exists a $(g_{\om_\ell},n_\ell, 4\vep)$-separated set  $\mathcal{S}_{\ell} \subset X$  
		with cardinality 
		$$
\#	\mathcal{S}_{\ell} 	 \geqslant \exp(\,(H^{\text{Pinsker}}(\psi) -3\gamma) n_\ell\,).
		$$
	\end{enumerate}
\end{lemma}

\begin{proof}
The proof is identical to the one of Lemma~\ref{le:K-entropy}, the differences being that the sequence $(n_\ell)_\ell$ is required to grow exponentially
with respect to the terms $K(L^{-2n_{i-1}}\,\,2\vep)$, and the separated sets are subsets of $X$
instead of subsets of some fibers $X_\om$ (which have a natural identification). 
Moreover, since we need not consider iterations by the skew-product, we need not take the extra 
$\lfloor \log (\vep^{-1}) \rfloor$ iterates necessary to control the dynamical distance with respect to the shift. 
We shall omit the details.
\end{proof}

Now, note that the frequent hitting times property for the semigroup action ensures that 
for each pair $(x_1,x_2) \in \mathcal S_1 \times \mathcal S_2$ there exists 
$0\leqslant p_{1} \leqslant K(L^{-2{n_1}})$ and a finite word 
$\hat \om^{(12)}=\hat \om^{(12)}(x_1,x_2) \in \{1,2,\dots, \kappa\}^{p_1}$ (depending on $x_1$ and $x_2$)
so that 
$$
f^{p_1}_{\om^{(12)}} (B(f^{n_1}_{\om_1}(x), L^{-2{n_1}}\vep))
\cap B_{\om_2}(x_2,n_2,\vep)
$$
is a subset of $X$ which contains a ball $B_1^2$ of radius $\frac12 L^{-n_2}\varepsilon$.
Thus
\begin{align}\label{GLW-En0}
C_{(\om_1,\, \hat \om^{(12)}(x_1,x_2),\,\om_2)}(\,x_1,x_2\,)
	& := B_{\om_1}(x_1,n_1,\vep)
	 \; \cap \;
	 (\, f^{p_1}_{\om^{(12)}} \circ f_{\om_1}^{n_1}\,)^{-1} \big( B_1^2 \big)
\end{align}
is a compact and non-empty subset of $X$.
As the number of finite words of length at most  $K(L^{-2{n_1}})$ is bounded above by $K(L^{-2{n_1}}) \, \kappa^{K(L^{-2{n_1}})}$, 
by the pigeonhole principle, for each $x_1\in \mathcal S_1$, 
	there exists 
	$0\leqslant t_1\leqslant K(L^{-2{n_1}}\vep)$,
	 a finite word $\om^{(12)}\in \{1,2,\dots, \kappa\}^{t_1}$ (depending only on $x_1$), and 
	a collection $\mathcal{\tilde C}_2$ of elements of the form ~\eqref{GLW-En0}
	with $\hat \om^{(12)}(x_1,x_2)=\om^{(12)}$ and
	\begin{equation*}
		\#  {\mathcal C_2} 
		\geqslant 
		 \# \mathcal S_{1} \cdot \frac{ \# \mathcal S_{2}}
			{K(L^{-2{n_1}}) \, \exp( K(L^{-2{n_1}}) \log \kappa )}. 
	\end{equation*}
In other words, the family $\mathcal C_2$ is so that map making the transition between the finite piece of orbit of $x_1$ and $x_2$ is 
the same $f^{t_1}_{\om^{(12)}}$, where $\om^{(12)}$ depends exclusively on $x_1$.
Let 
$$
\mathcal I_2=\big\{(x_1,x_2) \in \mathcal S_1\times \mathcal S_2 \colon C_{(\om_1,\,  \om^{(12)},\,\om_2)}(\,x_1,x_2\,) \in \mathcal C_2\big\}.
$$
Proceeding inductively we obtain, for each $k\geqslant 2$, a collection $\mathcal I_k\subset \mathcal{S}_{1} \times \dots \times \mathcal{S}_{k}$
so that
\begin{enumerate}
\item for each $(x_1, \dots, x_k) \in \mathcal I_k$ and 
$1\leqslant i\leqslant k-1$ there are 
$0\leqslant t_{i} \leqslant K(L^{-2n_i} \vep)$ and a finite word $\om^{(i (i+1))} \in \{1,2,\dots, \kappa\}^{t_i}$ (depending only on $x_1, x_2, \dots, x_{i-1}$),
and a ball $B_i^{i+1} \subset B_{\om_{i+1}}(x_{i+1},n_{i+1},\vep)$ of radius $\frac12 L^{-n_{i+1}}\varepsilon$ so that 
\begin{align}
	& C_{(\om_1,  \om^{(12)},\om_2, \om^{(23)},...,\om_k)}(\,x_1, \dots, x_k\,)  \nonumber \\
	& \qquad  :=
	B_{\om_1}(x_1,n_1,\vep) \; \cap \; \bigcap_{i=1}^{k-1} 	
		(\, 
		f^{p_i}_{\om^{(i(i+1)}} \circ f_{\om_i}^{n_i}
		\circ \dots \circ 
		f^{p_1}_{\om^{(12)}} \circ f_{\om_1}^{n_1} 
		\,)^{-1} (B_i^{i+1})
		\label{GLW-En}
\end{align}
is a non-empty subset of $X$;
\item the collection $\mathcal{C}_k$ of elements of the form ~\eqref{GLW-En}, parameterized by $\mathcal I_k$, satisfies
	\begin{equation*}
		\#  {\mathcal C_k} 
		\geqslant 
		\frac{\# \mathcal S_{k} \, \dots\, \# \mathcal S_{2}\, \# \mathcal S_{1}}
		{\prod_{i=1}^{k-1} K(L^{-2{n_i}}2\vep) \, \exp( K(L^{-2{n_i}}2\vep) \log \kappa )}.
	\end{equation*}
\end{enumerate}
For each $k\geqslant 1$ let 
$$m_k:=n_k+\sum_{i=1}^{k-1} (n_i+t_i)$$ 
and notice that although the transition length $t_i$ $(1\leqslant i\leqslant k-1)$ may vary with the elements in $\mathcal I_k$, by construction this
 is constant for any to $k$-uples $(\,x_1, \dots, x_k\,), (\,y_1, \dots, y_k\,)\in \mathcal I_k$ so that $x_j=y_j$ for all $1\leqslant j \leqslant i-1$.

\begin{lemma}\label{le:GLW-separated}
Given $k\geqslant 1$,
if $E_k\subset X$ is a collection of points, each one selected among a single element of $\mathcal C_k$ then 
$E_k$ is $(m_k,\vep)$-separated by elements of $G$. 
\end{lemma}

\begin{proof}
Fix $k\geqslant 1$ and 
let $E_k\subset X$ be an arbitrary collection of points, each one selected among a single element of $\mathcal C_k$.
Recall that, for each $\ell\geqslant 1$, the set $\mathcal{S}_{\ell} \subset X$ is $(g_{\om_\ell},n_\ell, 4\vep)$-separated. In consequence, the dynamic balls
$\{ B_{\om_i}(x, n_i,\vep) \colon x\in \mathcal S_i\}$ are pairwise disjoint, for every $1\leqslant i \leqslant k$, 
and the collection ${\mathcal C_k}$ is formed by a disjoint union of compact sets of the form \eqref{GLW-En}.
Let $x,y$ belong to distinct elements in ${\mathcal C_k}$ and let $(\,x_1, \dots, x_k\,), (\,y_1, \dots, y_k\,)\in \mathcal I_k$ be so that
$$
x \in C_{(\om_1,  \om^{(12)},\om_2, \om^{(23)},...,\om_k)}(\,x_1, \dots, x_k\,)
	\; \text{and}\;
	y \in C_{(\om_1,  \hat \om^{(12)},\om_2, \hat \om^{(23)},...,\om_k)}(\,y_1, \dots, y_k\,).
$$
The previous construction ensures that if $1\leqslant i \leqslant k$ is the smallest integer so that $x_i\neq y_i$ then $\om^{j(j+1)}=\hat \om^{j(j+1)}$
for every $1\leqslant j \leqslant j-1$. In particular, $g=f^{n_i}_{\om_i} \circ \dots \circ f_{\om_2}^{n_2} \circ f_{\om^{(12)}}^{t_1} \circ f_{\om_1}^{n_1} \in G_{m_i}$
satisfies $d(g(x), g(y))>\vep$. As $G_{m_i} \subset G_{m_k}$ we conclude that 
$E_k$ is $(m_k,\vep)$-separated by elements of $G$, as claimed. 
\end{proof}

As $(\mathcal C_k)_{k\geqslant 1}$ is a nested collection of disjoint and compact sets, the Moran set
$$ 
\mathfrak{F} =\bigcap_{k\geqslant 1} \; \Big[\bigcup_{C\in \mathcal C_k} \, C\Big] 
$$
is a non-empty compact subset of $X$.
Furthermore, as $\lim_{k\to\infty} \frac{m_k}{n_k}=1$ (recall item (1) in Lemma~\ref{le:K-entropy-fiber-GLW}),
 using the information of Lemma~\ref{le:GLW-separated} in ~\eqref{def:GLW-Y} one concludes that
\begin{align*}
\htop^{GLW}(\mathbb S, \mathfrak{F}) 
	& \geqslant \limsup_{k\to\infty} \frac1{m_k} \log s(G, G_1, \mathfrak F,m_k,\varepsilon) 
	  \geqslant \limsup_{k\to\infty} \frac1{m_k} \log \#  {\mathcal C_k} \\
	&  \geqslant \limsup_{k\to\infty} \frac1{m_k} \log  
		\frac{\# \mathcal S_{k} \, \dots\, \# \mathcal S_{2}\, \# \mathcal S_{1}}
		{\prod_{i=1}^{k-1} K(L^{-2{n_i}}2\vep) \, \exp( K(L^{-2{n_i}}2\vep) \log \kappa )} \\
	&  \geqslant H^{\text{Pinsker}}(\psi) -3\gamma.
\end{align*}
 Thus, in order to complete the proof of item (1) in Theorem~\ref{thm:Bb} it remains to prove the following:

\begin{lemma}
$\mathfrak{F}$ is contained in the irregular set $I_{\mathbb S}(\psi)$, defined by \eqref{eq:def-I-IFS}.
\end{lemma}

\begin{proof}
For every $x\in \mathfrak F$ there exists a nested sequence $(C_k)_{k\geqslant 1}$ 
of compact sets $$C_k=C_{(\om_1,  \om^{(12)},\om_2, \om^{(23)},...,\om_k)}(\,x_1, \dots, x_k\,) \in \mathcal C_k$$ 
so that $x\in \bigcap_{k\geqslant 1} C_k$.
The choice of constants in Lemma~\ref{le:K-entropy-fiber-GLW} guarantees, along an argument identical to the one in the proof of 
Lemma~\ref{lemma-point-nontypical},
that if $\om \in \bigcap_{k\geqslant 1} [\om_1\mid_{[0,n_1]} \om^{12} \om_2\mid_{[0,n_2]} \om^{23} \om_3\mid_{[0,n_3]} \dots \om_k\mid_{[0,n_k]}]$
then
$
\lim_{n\to\infty} \frac{1}{n}\sum_{j=0}^{n-1}\varphi(f^j_{{\underline{\omega}}}(x))
$ 
does not exist and, consequently, $x\in I_{\mathbb S}(\psi)$.
\end{proof}

\subsection{Bufetov entropy of the set of irregular points}\label{top-fibered2}

Here we prove item (2) in Theorem~\ref{thm:Bb}, namely that $h^B(\mathbb S,I_{\mathbb S}(\psi)) \geqslant h_*(\varphi_\psi)$. 
Indeed, using that $X$ is compact and \cite[Theorem~5.1]{JMW}  (see also Remark~2.4 in \cite{ZhuMa2021}), 
if $\Sigma_\kappa^+=\{1 , 2, \dots, \kappa\}^{\mathbb N}$ then 
$
\overline{Ch}_{\Sigma_\kappa^+ \times Z}(F) = \log \kappa + \overline{Ch_{\text{top}}}(\mathbb{S}, Z, \mathbb{P}_{sym}) 
$
for every $Z\subset X$ (the term in the left-hand side is the topological capacity of the skew-product from the aforementioned references 
and the right-hand side innvolves the topological capacity of the semigroup as in Subsection~\ref{sec:entropy-semigroup-annealed}). 
In particular
$$
\overline{Ch}_{\Sigma_\kappa^+ \times I_{\mathbb S}(\psi)}(F) = \log \kappa + \overline{Ch}_{I_{\mathbb S}(\varphi)}(\mathbb S).
$$
By compactness of $X$,
$$
h_{\text{top}}(F)=\overline{C\htop}(F,\Sigma_\kappa^+ \times X) = \log \kappa + \overline{Ch_{\text{top}}}(\mathbb{S}, X, \mathbb{P}_{sym}) 
	= \log \kappa +  h^B(\mathbb S, \mathbb{P}_{sym})
$$
Combining these two expressions with Theorem~\ref{thm:B} we conclude that 
$$
\htop(F) \geqslant  \overline{Ch}_{X}(F)
		=\overline{Ch}_{I_{\mathbb S}(\varphi)}(F)
		\geqslant h_*(\varphi),
$$
as desired. Since the two-sided shift $\Sigma_\kappa$ is the natural extension of $\Sigma_\kappa^+$, hence preserves the same topological entropy, 
the previous results remain unaltered when $\Sigma_\kappa^+$ is replaced by $\Sigma_\kappa$. 
This finishes the proof of Theorem~\ref{thm:Bb}.

\section{ Topological entropy of fibered irregular sets}\label{top-fibered}

This subsection is devoted to the proof of Theorem~\ref{thm:Bb-2}.
Let $X$ be a compact metric space, assume that the semigroup action $\mathbb S$ is generated by the collection $G_1=\{id,f_1, f_2, \dots, f_\kappa\}$ 
of bi-Lipschitz homeomorphisms and let $\psi: X\to \mathbb R$ 
be a continuous observable so that $\varphi_\psi$ is not a coboundary with respect to $F$.
Assume also that there exists a strongly transitive homeomorphism $f_i\in G_1$, for some $1\leqslant i \leqslant\kappa$. 
We will assume without loss of generality that $i=\kappa$. 
In particular, by compactness of $X$, for each $\delta>0$ there exists 
$N_\delta\geqslant 1$ so that
\begin{equation}\label{eq:top-exact}
\bigcup_{j=0}^{N_\delta} f_\kappa^j (B(x,\delta))=X \qquad 
\text{for every $x\in X$}.
\end{equation}
Let ${\mathfrak l}_x(\delta)>0$ denote the Lebesgue covering number of $\{f_\kappa^j (B(x,\delta))\}_{j=0}^{N_\delta}$. By compactness of $x$, 
there exists ${\mathfrak l}_\delta:=\min_{x\in X} {\mathfrak l}_x(\delta)>0$ so that: 
\begin{equation}\label{eq:top-exact-conseq}
\text{$\forall x,y\in X, \, \forall 0<\vep< {\mathfrak l}_\delta, \exists \,0 \leqslant p\leqslant N_\delta$
s.t.  $f_\kappa^p (B(x,\delta)) \supset B(y,\vep)$.}
\end{equation}

\begin{remark}
The latter implies that the semigroup $\mathbb S$ has the frequent hitting times
property where the transition maps can be obtained by a certain number of iterates of the map $f_\kappa$. 
The expression \eqref{eq:top-exact} can be compared to the quantitative estimates obtained in Remark~\ref{rmk:rotation}
in the special case that $f_\kappa$ is an isometry.
\end{remark}

\subsection{Fibered complexity of the set of irregular points}\label{top-fibered1}

Here we prove item (1), namely that $h^{path}_I(\mathbb S,\psi) \geqslant H^{\text{Pinsker}}(\psi)$.
Assume that $H^{\text{Pinsker}}(\psi)>0$ (otherwise there is nothing to prove).
By equality ~\eqref{eq:Pinsker-Katok} we know that $H^{Katok}(\psi)=H^{\text{Pinsker}}(\psi)>0$.
Hence, for any fixed $0<\gamma,\delta<1$ 
there exist $\mu_1,\mu_2 \in \mathcal M_{erg}(F)$ so that $\int \varphi_\psi \,d\mu_1< \int \varphi_\psi \,d\mu_2$ 
and $h_{\mu_i}^{\delta}(F)\geqslant H^{\text{Pinsker}}(\psi) -2\gamma$, for $i=1,2$.
Using the equality ~\eqref{eq:comparison-entropies} and the ergodic theorem, 
one concludes that there are $\om_i\in \Sigma_\kappa$ so that:
\begin{enumerate}
\item[(a)] $\lim_{n\to\infty} \frac1n \sum_{j=0}^{n-1}\delta_{\sigma^j(\om_i)} = \pi_*\mu_i$
\item[(b)] $\lim_{n\to\infty} \frac1n \sum_{j=0}^{n-1} \psi(f_{\om_i}^j(x))=\int \varphi_\psi \, d\mu_i$ for $\mu_{\om_i}$-a.e. $x\in X_{\om_i}$
\item[(c)] $h_{\mu}^{\delta}(\om_i) > H^{\text{Pinsker}}(\psi) -2\gamma$ 
\end{enumerate}
for $i=1,2$, where for some abuse of notation we denote by $\mu_{\om_i}$ denotes the sample measure of $\mu_i$ on $X_{\om_i}$
obtained by Rokhlin's disintegration theorem. 

\smallskip
Since $\gamma>0$ was chosen arbitrary, in order to prove  the claim 
it is sufficient to construct 
$\om\in \Sigma_\kappa$ and a fibered Moran subset of the fiber $X_\om$ with entropy larger than $H^{\text{Pinsker}}(\psi)-\gamma$.
Such a construction differs from the ones at Sections~\ref{top-skew} and ~\ref{top-semigroups} in the sense that, 
in order to capture information created by the relative entropies 
$h_{\mu_i}(F\!\mid\!\sigma)$, the point $\om\in \Sigma_\kappa$ is constructed inductively. In fact, we will provide a sequence of integers $(m_k)_{k\geqslant 1}$, 
a nested sequence of cylinders $[\omega_{1}, \omega_{2}, \dots, \omega_{m_k}]$, 
and for each $\tilde \om \in \bigcap_{k\geqslant 1} [\omega_{1}, \omega_{2}, \dots, \omega_{m_k}]$ a Moran set 
$
\mathfrak{F}_{\widetilde{\om}}\subset X_{\widetilde{\om}}
$
so that 
$$
h_{\mathfrak{F}_{\widetilde{\om}}}(\mathcal F_{\widetilde{\om}}) \geqslant H^{\text{Pinsker}}(\psi) -\gamma
$$
(here $h_{\mathfrak{F}_{\widetilde{\om}}}(\mathcal F_{\widetilde{\om}})$ denotes the fibered relative entropy of the set $\mathfrak{F}_{\widetilde{\om}}$, recall Subsection~\ref{fibentropies}). Albeit the construction of the Cantor sets $\mathfrak{F}_{\widetilde{\om}}$ has a different flavor from the construction of the Cantor
set in the previous section, we will use some of the technical tools developed in the latter sections. We will emphasize the differences in comparison to the 
previous construction and skim the details on the transition times.

\subsubsection{Fibered orbit growth estimates}\label{subsec:foge}

Consider a strictly decreasing sequence of positive numbers $(\delta_k)_{k\geqslant 1}$ tending to zero so that 
$
\delta_1 < \frac14 [\int \varphi_\psi \,d\mu_2-\int \varphi_\psi \,d\mu_1].
$
Associated to these, item (b) above together with Lusin's theorem allows one to pick a strictly increasing sequence of large positive integers 
$(\mathfrak n_k)_{k\geqslant 1}$ so that,
for each $\ell\geqslant 1$, the fibered sets
$$
\Gamma_{2\ell-1}:=
\Big\{ x\in X_{\om_1}: \Big|\frac{1}{n}\sum_{j=0}^{n-1}\psi(g_\om^j(x))-\int \varphi_\psi \,d\mu_1\Big|< \delta_{2\ell-1},\;
\; \forall n\geqslant \mathfrak n_{2\ell-1} \Big\}
$$ 
and
$$
\Gamma_{2\ell}:=
\Big\{ x \in X_{\om_2}: \Big|\frac{1}{n}\sum_{j=0}^{n-1}\psi(g_\om^j(x))-\int \varphi_\psi \,d\mu_2\Big|< \delta_{2\ell},\;
\; \forall n\geqslant \mathfrak n_{2\ell} \Big\}
$$ 
satisfy $\mu_{\om_1}(\Gamma_{2\ell-1}) > \frac12$ and $\mu_{\om_2}(\Gamma_{2\ell}) > \frac12$.
Moreover, using item (c) above and the definition of the fibered Katok $\delta$-entropy (taking $\delta=\frac12$ in ~\eqref{eq:def-d-Katok}) one can choose $\vep>0$ small so that
\begin{eqnarray}\label{eq entropy-Katok-fiber1}
	\liminf_{n\rightarrow +\infty}\frac{1}{n}\log s(\om_1,n, 8\vep,\Gamma_{2\ell-1})
	> H^{\text{Pinsker}}(\psi) -2\gamma
\end{eqnarray}
and
\begin{eqnarray}\label{eq entropy-Katok-fiber2}
	\liminf_{n\rightarrow +\infty}\frac{1}{n}\log s(\om_2,n, 8\vep,\Gamma_{2\ell})
	>  H^{\text{Pinsker}}(\psi) -2\gamma.
\end{eqnarray}
For each $i\geqslant 3$ we take $\om_i=\om_1$ if $i$ is odd and set $\om_i=\om_2$ otherwise.
By uniform continuity of $\psi$, reducing $\vep$ is necessary, we may also assume that
\begin{eqnarray}\label{eq-unifcont}
|\psi(x)-\psi(x')| 
	< \frac{\int \varphi \,d\mu_2-\int \varphi \,d\mu_1}{4}
\end{eqnarray}
for every $x,x'\in X$ so that $d(x,x')<4\vep$.
The next lemma provides a lower bound estimate on the cardinality of separated sets within the family of sets $(\Gamma_\ell)_{\ell\geqslant 1}$.

\begin{lemma}\label{le:K-entropy-fiber}
	There exists a sequence $(n_\ell)_{\ell\geqslant 1}$ of positive integers  
	satisfying
	\begin{equation}\label{eq:select-nell}
	n_\ell \geqslant 
			\max\{\; \mathfrak n_\ell, 
				\; (\ell+2) \cdot n_{\ell-1} + \Big[ \frac{2\log 2}{\log L}  \Big], 
				\; \ell \cdot (\log \vep^{-1} - \log  {\mathfrak l}_{L^{-2{n_{\ell-1}}}\vep}),
				\; \ell \cdot \frac1\gamma \cdot  N_{L^{-2{n_{\ell-1}}}\vep} \;\}
	\end{equation}
		for every $\ell\geqslant 2$, and so that for each  $\ell\geqslant 1$ there exists an $(\om_\ell,n_\ell, 4\vep)$-separated set  $\mathcal{S}_{\ell} \subset \Gamma_{\ell}$  
		with cardinality 
		$
		\#\mathcal{S}_{\ell} 	 \geqslant \exp(\,(H^{\text{Pinsker}}(\psi) -3\gamma) n_\ell\,).
		$
\end{lemma}

\begin{proof}
The proof is similar to the one of Lemma~\ref{le:K-entropy}, and for that reason we shall omit it.
\end{proof}

\subsubsection{Construction of fibered Moran sets of irregular points}\label{MSEntropyFiber}

We proceed to construct a point $\om\in \Sigma_\kappa$ and a special Moran subset of $X_\om$ formed by points with irregular behavior.
By Lemma~\ref{le:K-entropy-fiber}, given $k\geqslant 1$ and $1\leqslant i\leqslant k$ one can choose a $(\om_i,n_i,4\vep)$-separated set $\mathcal{S}_{i}\subset \Gamma_{i}$ so that $
\#	\mathcal{S}_{i} 	 \geqslant \exp(\,(H^{\text{Pinsker}}(\psi) -3\gamma) n_i\,).
		$
Given $x_1\in \mathcal{S}_{1}$ and $x_2\in \mathcal{S}_{2}$, one knows that 
$$
B(f^{n_1}_{\om_1}(x_1), L^{-2{n_1}}\vep) \subset f_{\om_1}^{n_1}(\, B_{\om_1}(x_1,n_1,\vep) \,)
	\quad\text{and}\quad
	B(x_2,L^{-n_2}\vep) \subset B_{\om_2}(x_2,n_2,\vep).
$$
The choice of constants in \eqref{eq:select-nell} guarantees that $L^{-n_2}\vep$ is smaller than the Lebesgue covering number ${\mathfrak l}_{L^{-2{n_1}}\vep}$. 
Thus, according to 
~\eqref{eq:top-exact-conseq}, 
one can pick $0\leqslant p_1\leqslant N_{L^{-2{n_1}}\vep}$ so that
$$
f_\kappa^{p_1}(B(f^{n_1}_{\om_1}(x_1), L^{-2{n_1}}\vep)) \supset B(x_2,L^{-n_2}\vep). 
$$
Thus
\begin{align*}
C_{(\om_1,\, \kappa^{p_1},\,\om_2)}(\,x_1,x_2\,)
	& := B_{\om_1}(x_1,n_1,\vep)
	 \; \cap \;
	 (\, f^{p_1}_{\kappa} \circ f_{\om_1}^{n_1}\,)^{-1} \big( B(x_2,L^{-n_2}\vep) \big)
\end{align*}
is a compact and non-empty subset of $X_{\om_1}$
and
\begin{align*}
(f_{\om_2}^{n_2}\circ f^{p_1}_{\kappa} \circ f_{\om_1}^{n_1}) (\,C_{(\om_1,\, \kappa^{p_1},\,\om_2)}(\,x_1,x_2\,)\,)
	 \supset  B(f_{\om_2}^{n_2}(x),L^{-2n_2}\vep) 
\end{align*}
where $\kappa^{p_1}$ denotes the 
finite word $(\kappa, \kappa, \dots, \kappa)\in \{1,2, \dots, \kappa\}^{p_1}$.
Proceeding recursively, for each $k\geqslant 1$, $(x_1, \dots, x_k) \in \mathcal{S}_{1} \times \dots \times \mathcal{S}_{k}$ and 
$1\leqslant i\leqslant k-1$ we obtain an integer  $0\leqslant p_{i} \leqslant N_{L^{-2{n_{i}}}\vep}$ so that
\begin{align*}
	& C_{(\om_1,  \kappa^{p_1},\om_2, \kappa^{p_2},...,\om_k)}(\,x_1, \dots, x_k\,) \\
	& \qquad  :=
	B_{\om_1}(x_1,n_1,\vep) \; \cap \; \bigcap_{i=1}^{k-1} 	
		(\, 
		f^{p_i}_{\kappa} \circ f_{\om_i}^{n_i}
		\circ \dots \circ 
		f^{p_1}_{\kappa} \circ f_{\om_1}^{n_1} 
		\,)^{-1} (\, B(x_{i+1},L^{-n_{i+1}}\vep)\,)
\end{align*}
is a non-empty subset of $X_{\om_1}$.
The next result ensures that one can select a large number of points for which the transition times function $p_i=p_i(\,x_1, \dots, x_i\,)$ is constant, 
and use this to construct a suitable $\om\in\Sigma_\kappa$ whose non-autonomous dynamical system $\mathcal F_\om$
has a irregular set with large entropy. 

\begin{lemma}\label{le:selection-fiber-omega}
	There exists $K_0>0$ and for each $k\geqslant 1$ there exists
	a collection $\mathcal I_k\subset  \mathcal{S}_{1} \times  \mathcal{S}_{2} \times \dots \times \mathcal{S}_{k}$ such that:
	\begin{enumerate}
		\item 
		for any $1\leqslant i \leqslant k-1$ there exists $1\leqslant t_i\leqslant  N_{L^{-2{n_i}}\vep}$ so that 
		$p_i(\,x_1, \dots, x_k\,)=t_k$ for every $(\,x_1, \dots, x_k\,)\in \mathcal I_k$,
		\item $\#  {\mathcal I}_k \geqslant K_0\, e^{  (H^{\text{Pinsker}}(\psi) -4\gamma) \,m_k }$
		\item $d_\om^{m_k}(x, y)>2\vep$
		for any points $x\in C_{(\om_1,  \kappa^{t_1},\om_2, \kappa^{t_2},...,\om_k)}(\,x_1, \dots, x_k\,)$ and 
		$y\in C_{(\om_1,  \kappa^{t_1},\om_2, \kappa^{t_2},...,\om_k)}(\,y_1, \dots, y_k\,)$ associated to distinct 
		$k$-uples $(\,x_1, \dots, x_k\,), (\,y_1, \dots, y_k\,) \in \mathcal I_k$,
	\end{enumerate}
	where $m_k=n_k + \sum_{i=1}^{k-1} (n_i+t_i)$ and $\om \in \bigcap_{k\geqslant 1} [\om_1,  \kappa^{t_1},\om_2, \kappa^{t_2},...,\om_k]$.
\end{lemma}

\begin{proof}
By the pigeonhole principle, for each $1\leqslant i \leqslant  k-1$ there exists 
$0\leqslant t_i\leqslant N_{L^{-2{n_i}}\vep}$ and a collection 
$\mathcal I_k\subset  \mathcal{S}_{1} \times  \mathcal{S}_{2} \times \dots \times \mathcal{S}_{k}$ such that 
item (1) above holds and, recalling that $n_i\gg N_{L^{-2{n_i}}\vep}$ and $\lim_{k\to\infty} \frac{m_k}{n_k}=1$ 
(recall ~\eqref{eq:select-nell}),
	\begin{align*}\label{eq:card-Ck-fibered-s}
		\#  {\mathcal I_k} 
		& \geqslant \frac{\# \mathcal S_{k} \, \dots\, \# \mathcal S_{2}\, \# \mathcal S_{1}}
					{\prod_{i=1}^{k-1} \, N_{L^{-2{n_i}}\vep}} \\
		& \geqslant e^{  (H^{\text{Pinsker}}(\psi) -3\gamma) \, [n_k+ \dots + n_2 + n_1]}
		\prod_{i=1}^{k-1}  \frac1{n_{i+1}} \, \frac{n_{i+1}}{N_{L^{-2{n_i}}\vep}} \\
		& \geqslant 
	 	 e^{  (H^{\text{Pinsker}}(\psi) -4\gamma) \, m_k}
	\end{align*}
	for every large $k\geqslant 1$, where $m_k=n_k + \sum_{i=1}^{k-1} (n_i+t_i)$. 
	In particular, $\#  {\mathcal I}_k \geqslant K e^{  (H^{\text{Pinsker}}(\psi) -4\gamma) \,m_k }$ for every $k\geqslant 1$.
	Now, if $\om \in \bigcap_{k\geqslant 1} [\om_1,  \kappa^{t_1},\om_2, \kappa^{t_2},...,\om_k]$
	and $x\in C_{(\om_1,  \kappa^{t_1},\om_2, \kappa^{t_2},...,\om_k)}(\,x_1, \dots, x_k\,)$ and 
		$y\in C_{(\om_1,  \kappa^{t_1},\om_2, \kappa^{t_2},...,\om_k)}(\,y_1, \dots, y_k\,)$ are points associated to distinct 
		$k$-uples $(\,x_1, \dots, x_k\,), (\,y_1, \dots, y_k\,) \in \mathcal I_k$ 
	and $j=\inf\{ 1\leqslant i \leqslant k \colon x_i\neq y_i\}$ then
$
		d_\om^{m_k}(x, y)
			\geqslant d(x_j,y_j) - 
		d(f_{\om}^{m_j}(x), x_j) - d(f_{\om}^{m_j}(y), y_j) >\vep.
$
This completes the proof of the lemma.
\end{proof}

The sequence $\om\in \Sigma_k$ above was determined by the concatenation of the cylinders 
${\om_i}_{[0,n_i]}$
and the cylinders determined by the finite words $\kappa^{p_i}$, i.e., 
\begin{equation}\label{eq:omega-select-fiber}
\om \in 
	 ({\om_1}_{[0,n_1]} \cdot \kappa^{p_1})
		\, \cap\,  
	\bigcap_{i\geqslant 1}  \sigma^{-\sum_{j=1}^i(p_j+n_j)} \;  ({\om_{i+1}}_{[0,n_{i+1}]} \cdot \kappa^{p_{i+1}}).
\end{equation}
Consider also the sets
\begin{equation*}\label{eq:Fk-fibered-infty}
\mathfrak{F}_{m_k} (\om) = \bigcup_{(x_1, \dots, x_k) \in \mathcal{S}_{1} \times \dots \times \mathcal{S}_{k}} 
				C_{(\om_1,  \kappa^{p_1},\om_2, \kappa^{p_2},...,\om_k)}(\,x_1, \dots, x_k\,)
\end{equation*}
identified as a subset of $X_\om$, 
where $m_k=n_k+\sum_{i=1}^{k-1} (n_i+t_i)$, and the Moran set
$$
\mathfrak{F}_{{\om}}:=\bigcap_{k\geqslant 1} \mathfrak{F}_{m_k}(\om)  \; \cap \; X_{{\om}}.
$$

\begin{lemma}
$\mathfrak{F}_{{\om}}$ is contained in the irregular set $I_\om(\psi)$ defined by \eqref{def-i-fibras}.
\end{lemma}

\begin{proof}
The argument follows the lines of the proof of Lemma~\ref{lemma-point-nontypical}, and for that reason we shall omit it. 
\end{proof}

\medskip

\subsubsection{The relative entropy of $\mathfrak{F}_\om$}

In order to provide lower bounds for the fiber topological entropy of the fibered  Moran set $\mathfrak{F_\om}$ for the non-autonomous dynamical system 
$\mathcal{F}_\om$ we will make use of the 
fibered version of the entropy distribution principle (cf. Proposition~\ref{pro. princ dist. entropy fiber}).
For each $k\geqslant 1$ consider the probability measure $\mu_{\om,k}$ 
supported on points selected among the elements of sets in the family 
$$
\mathcal{C}_k =\{C_{(\om_1,  \kappa^{p_1},\om_2, \kappa^{p_2},...,\om_k)}(\,x_1, \dots, x_k\,) \colon (x_1, x_2, \dots, x_k) \in \mathcal I_k\}.
$$ 
More precisely,
for each $C\in \mathcal{C}_k$ choose a point $z_C\in C$ 
and consider 
the probability measure
$$
\mu_{\om,k}=\frac{1}{\#\mathcal{C}_k}
\sum_{ C\in \mathcal{C}_k}\delta_{z_C}.
$$

\begin{lemma}\label{le:acc-fibered}
Any weak$^*$ accumulation point $\nu_\om$ of the sequence $(\mu_{\om,k})_{k\geqslant 1}$
satisfies $\nu_\om(\mathfrak{F}_\om)=1$.
\end{lemma}

\begin{proof}
Let $\nu_\om$ be an weak$^*$ accumulation point of the sequence $(\mu_{\om,k})_{k\geqslant 1}$.
Given $k\ge1$, using Portmanteau theorem on the characterization of weak$^*$ convergence, as $\mu_{\om,\ell}(\mathfrak{F}_{m_k} (\om))=1$ for every $\ell \geqslant k$ then  
$\nu_{\om}(\mathfrak{F}_{m_k} (\om))=1$. Thus $\nu_{\om}\big(\bigcap_{k\geqslant 1} \mathfrak{F}_{m_k} (\om)\big)=1$
as claimed.
\end{proof}

Now, let $\vep>0$ be small and $\mathbf{n}\geqslant 1$ be fixed and large.  Let $B_\om(x,\mathbf{n},\vep)\subset  X_\om$ be a dynamic ball which intersects the set $\mathfrak{F}_\om$ and let $\ell=\ell(\mathbf{n})\geqslant 1$ be so that 
$m_{\ell} \leqslant \mathbf{n} <m_{\ell+1}$, where the sequence $(m_\ell)_{\ell\geqslant 1}$ was defined in Lemma~\ref{le:selection-fiber-omega}.
In order to apply Proposition~\ref{pro. princ dist. entropy fiber} we need the following:

\begin{lemma}\label{lemma:estimative ball entropy-fiber}
	$\mu_{\om,k+\ell}(B_\om((x,\mathbf{n},\vep))\leqslant e^{-\mathbf{n}(H^{\text{Pinsker}}(\psi)-5\gamma)}$ for every large $k\geqslant 1$.

\end{lemma}

\begin{proof}
Suppose that  $\mu_{\om,k+\ell}(B_\om(x,\mathbf{n},\vep))>0$ (otherwise there is nothing to prove), hence
$B_\om(x,\mathbf{n},\vep)$ intersects $\mathfrak{F}_{\om,m_{k+\ell}}$.
	We observe that if $y_1\in B_\om(x,\mathbf{n},\vep)\cap C_1$ and  $y_2\in B_\om(x,\mathbf{n},\vep)\cap C_2$ where $C_1,C_2 \in \mathcal{C}_{\ell}$ 
	then item (3) in Lemma~\ref{le:selection-fiber-omega}, guarantees that 
	$ d_{\om}^{\mathbf n}(y_1,y_2) \geqslant d_{\om}^{m_{\ell}}(y_1,y_2)>2\vep$ whenever $C_1$ and $C_2$ are distinct. Thus, 
	all points in the set $B_\om(x,\mathbf{n},\vep) \cap \mathfrak{F}_{\om,m_{k+\ell}}$
	belong to a single fixed set  $C_\ell \in \mathcal{C}_{\ell}$. 
	Furthermore, 
	the construction in the proof of Lemma~\ref{le:selection-fiber-omega} yields that, for every $D\in \mathcal C_\ell$,
	$$
	\# \Big\{C\in \mathcal{C}_{k+\ell} \colon C\cap D\neq\emptyset \Big\} 
		\geqslant \frac{\# \mathcal S_{k+\ell} \, \dots\, \# \mathcal S_{\ell+1}}
					{\prod_{i=\ell}^{k+\ell-1} \, N_{L^{-2{n_i}}\vep}} 
		>  e^{  (H^{\text{Pinsker}}(\psi)-4\gamma) \, [n_{k+\ell} +  \dots + n_{k+1}]} 
	$$
	for every large $k$. In particular, using that 
	\begin{align*}
		\# \{C\in \mathcal{C}_{k+\ell} \colon C\cap C_\ell\neq\emptyset \} 
		& \leqslant  \frac{\#\mathcal{C}_{k+\ell}}{\min\limits_{ D \in \mathcal C_\ell} \;   \# \{C\in \mathcal{C}_{k+\ell} \colon C\cap D\neq\emptyset \} } 
		\nonumber \\
		& < e^{ - (H^{\text{Pinsker}}(\psi)-4\gamma) \, [n_{k+\ell} +  \dots + n_{k+1}]} \; \#\mathcal{C}_{k+\ell},
	\end{align*}
	that 
	$m_{k+\ell} > \mathbf n$, that $\lim_{k\to\infty} \frac{m_k}{n_k}=1$ and the choices at item (1) in Lemma~\ref{le:selection-fiber-omega}, we obtain that
	\begin{align*}
		\mu_{\om,k+\ell}(B_\om(x,\mathbf{n},\vep))
		&
		=
		\displaystyle\frac{1}{\#\mathcal{C}_{k+\ell}} \sum_{ \substack{C\in \mathcal{C}_{k+\ell}}}
		\delta_{z_C}\left(B_\om(x,\mathbf{n},\vep) \right) \\
		&\leqslant
		\displaystyle\frac{\# \{C\in \mathcal{C}_{k+\ell} \colon C\cap C_\ell\neq\emptyset \}  }{\#\mathcal{C}_{k+\ell}} 
		\\
		& < e^{ - (H^{\text{Pinsker}}(\psi)-4\gamma) \, [n_{k+\ell} +  \dots + n_{k+1}]}
		\\
		& < e^{ - (H^{\text{Pinsker}}(\psi)-5\gamma) \, m_{k+\ell}}
		 < e^{ - \mathbf n\, (H^{\text{Pinsker}}(\psi)-5\gamma)}
	\end{align*}
		for every large $k\geqslant 1$, thus concluding the proof of the lemma.
	
\end{proof}

We can now complete the proof of  item (1) in Theorem~\ref{thm:Bb-2}. Indeed, as the assumptions of the entropy distribution principle are satisfied
(as a consequence of Lemmas~\ref{le:acc-fibered} and ~\ref{lemma:estimative ball entropy-fiber} above)
we conclude that 
$$
h^{path}_I(\mathbb S,\psi) \geqslant h_{I_\om(\psi)}(\mathcal F_\om) \geqslant H^{\text{Pinsker}}(\psi)-5\gamma.
$$
As $\gamma>0$ was chosen arbitrary, we get that $h^{path}_I(\mathbb S,\psi) \geqslant  H^{\text{Pinsker}}(\psi)$, as desired. 

\subsection{Moran sets of irregular points with large fibered entropy}\label{MSEntropyFiber-2}
In this subsection we explain the necessary modifications in the argument explored in the previous subsections of in order to prove item (2) in
 Theorem~\ref{thm:Bb-2}, namely that the topological entropy of the set $\Sigma:=\{\om\in \Sigma_\kappa \colon h_{I_\om(\psi)}(\mathcal F_\om) \geqslant H^{\text{Pinsker}}(\psi)\}$ 
is larger or equal than $H^{\text{Pinsker}}_\sigma(\psi)$.

\medskip
Fix $\gamma>0$. By definition (cf. ~\eqref{Hestrela2-sigma}) there exists $\zeta>0$ so that $H^{\text{Pinsker}}_\sigma(\psi,\zeta)\geqslant 
H^{\text{Pinsker}}_\sigma(\psi)-\gamma$. In particular there exist $\mu_1,\mu_2 \in \mathcal M_{erg}(F)$ so that $\int\varphi_\psi\, d\mu_1< \int\varphi_\psi\, d\mu_2$, 
\begin{align}
\label{Hestrela2-sigma-2}
h_{\mu_i}(F\mid\sigma)\geqslant H^{\text{Pinsker}}(\psi) -\zeta
	\quad\text{and}\quad
	h_{\pi_*\mu_i}(\sigma)\geqslant H^{\text{Pinsker}}_\sigma(\psi)-\gamma
\end{align}
for $i=1,2$. As the probability measures $\mu_i$ are ergodic (taking $\delta=\frac14$ in ~\eqref{eq:comparison-entropies}) it follows that 
$
h_{\mu_i}^{1/4}(\ud\omega) = h_{\mu_i}(F\!\mid \!\sigma)  \geqslant H^{\text{Pinsker}}(\psi) -\zeta
$
for $(\pi)_*\mu_i$-almost every $\om\in \Sigma_\kappa$. 
Fix $0<\delta<\frac18$. By the definition of the fibered Katok 
entropy (recall ~\eqref{eq:def-d-Katok}) and 
 Lusin's theorem, for each $i=1,2$ there exists $\tilde \Omega_i\subset \Sigma_\kappa$ and $\vep>0$ small so that  $(\pi_*\mu_i)(\tilde \Omega_i) > 1-\delta$ and 
$\liminf_{n\rightarrow +\infty}\frac{1}{n}\log s(\om,n, 4\vep,\frac14)
	> H^{\text{Pinsker}}(\psi) -2\zeta,
$
for every $\om \in \tilde \Omega_i$, $i=1,2$. Using Lusin's theorem once more, we may take an integer $\mathfrak n_0\geqslant 1$ and a subset $\Omega_i\subset \tilde \Omega_i$ so that $(\pi_*\mu_i)(\Omega_i)\geqslant (1-2\delta)>\frac34$ and 
\begin{equation}\label{eq:fibersuniformCantor}
s(\om,n, 4\vep,\frac14)
	> e^{(H^{\text{Pinsker}}(\psi) -3\zeta)n}
	\quad \text{for every $n\geqslant \mathfrak n_0$, $\om\in \Omega_i$ and $i=1,2$.}
\end{equation}
We diminish $\vep$, if necessary, so that
$
|\psi(x)-\psi(x')| 
	< \frac{\int \varphi \,d\mu_2-\int \varphi \,d\mu_1}{4}
$ 
whenever $d(x,x')<4\vep$.

\medskip
Now we turn our attention to the convergence of Birkhoff averages on $\Sigma_\kappa\times X$. 
 Consider a strictly decreasing sequence of positive numbers $(\delta_k)_{k\geqslant 1}$ tending to zero so that 
$
\delta_1 < \frac14 [\int \varphi_\psi \,d\mu_2-\int \varphi_\psi \,d\mu_1].
$
Using Birkhoff ergodic theorem and Lusin theorem, there exists a strictly increasing sequence of large positive integers 
$(\mathfrak n_k)_{k\geqslant 1}$ so that $\mathfrak n_1>\mathfrak n_0$ and,
for each $\ell\geqslant 1$, the sets
$$
\Gamma_{2\ell-1}:=
\Big\{(\om,x)\in \{1,2,...,\kappa\}^{\Z} \times X: \Big|\frac{1}{n}\sum_{j=0}^{n-1}\varphi(F^j(\om,x))-\int \varphi \,d\mu_1\Big|< \delta_{2\ell-1},\;
\; \forall n\geqslant \mathfrak n_{2\ell-1} \Big\}
$$ 
and
$$
\Gamma_{2\ell}:=
\Big\{
(\om,x)\in \{1,2,...,\kappa\}^{\Z} \times X:
\Big| \frac{1}{n}\sum_{j=0}^{n-1}\varphi(F^j(\om,x)) -\int \varphi \,d\mu_2\Big|< \delta_{2\ell},\;
\; \forall n\geqslant \mathfrak n_{2\ell} \Big\}
$$
satisfy $\mu_{1}(\Gamma_{2\ell-1}) > 1-\delta$ and $\mu_{\om_2}(\Gamma_{2\ell}) > 1-\delta$ for every $\ell\geqslant 1$.
These sets are not suitable to our purposes as we need to decouple information on the shift space $\Sigma_\kappa$ and 
the fibers. This is done as follows.
\smallskip
Let $\mu_i=(\mu_{i,\om})_\om$ be the disintegration of $\mu_i$ in sample measures given by Rokhlin's disintegration theorem. More precisely, 
there exists a $\mu_i$-almost everywhere defined and measurable family of probability measures $\om \mapsto \mu_{i,\om}$
so that $\supp \mu_{i,\om}\subset X_\om$ and 
$\mu_i(A)=\int_{\Sigma_\kappa} \, \mu_{i,\om}(A\cap X_\om)\, d(\pi_*\mu_i)(\om)$ for every measurable set $A\subset \Sigma_\kappa\times X$.
We will use the following quantitative estimate.

\begin{lemma}\label{le:decoupled}
For each $\ell\geqslant 1$ there exist subsets $\Xi_{2\ell-1}, \Xi_{2\ell}\subset \Sigma_\kappa$ so that 
\begin{enumerate}
\item $(\pi_*\mu_1)(\Xi_{2\ell-1})>\frac14$ and $(\pi_*\mu_2)(\Xi_{2\ell})>\frac14$,
\item 
	$\mu_{1,\om}\big( X_\om \cap \Gamma_{2\ell-1}  \big) > \frac14$
	for every $\om\in \Xi_{2\ell-1}$, and
\item 
	$\mu_{2,\om}\big( X_\om \cap \Gamma_{2\ell}  \big) > \frac14$
	for every $\om\in \Xi_{2\ell}$.
\end{enumerate}
\end{lemma}

\begin{proof}
We use the following Fubini type estimate
involving to the sample measures $(\mu_{1,\om})_\om$: 
\begin{align*}
1-\delta < \mu_1(\Gamma_{2\ell-1}) = 
	& \int_{\Sigma_\kappa} \, \mu_{i,\om}(X_\om \cap \Gamma_{2\ell-1})\, d(\pi_*\mu_1)(\om) \\
	& \leqslant 
	\pi_*\mu_1 \Big( \om\in \Sigma_\kappa \colon  \mu_{i,\om}\big( X_\om \cap \Gamma_{2\ell-1} \big) > a \Big) \\
 	& + a \cdot \pi_*\mu_1 \Big( \om\in \Sigma_\kappa \colon  \mu_{i,\om}\big( X_\om \cap \Gamma_{2\ell-1} \big) \leqslant  a \Big)
\end{align*}
and, consequently,
$
\pi_*\mu_1 \Big( \om\in \Sigma_\kappa \colon  \mu_{i,\om}\big( X_\om \cap \Gamma_{2\ell-1} \big) > a \Big) > 1-\delta-a
$
for every $0<a<1$.
The choice $0<\delta<\frac14$ guarantees that 
$
\pi_*\mu_1 \Big( \om\in \Sigma_\kappa \colon  \mu_{i,\om}\big( X_\om \cap \Gamma_{2\ell-1} \big) > \frac14 \Big)
	> \frac12.
$
The claim concerning the sample measures $\mu_{2,\om}$ follows similarly. 
\end{proof}

Observe that $\pi_*\mu_1(\Omega_1\cap \Xi_{2\ell-1})>1-2\delta-\frac14>0$ and $\pi_*\mu_2(\Omega_2\cap \Xi_{2\ell})>1-2\delta-\frac14>0$ 
for every $\ell\geqslant 1$. 
Moreover, by Katok's entropy formula (recall ~\eqref{eq entropy-Katok}) 
one can compute $h_{\pi_*\mu_i}(\sigma)$ ($i=1,2$) using 
the sets $\Omega_1\cap \Xi_{2\ell-1}$ and $\Omega_2\cap \Xi_{2\ell}$:
$$
h_{\pi_*\mu_i}(\sigma)=\lim_{\vep\to 0} \;\Big[ \liminf_{n\rightarrow +\infty}\frac{1}{n}\log N^{\pi_*\mu_i}_n(8\vep,1-2\delta-\frac14)\,\Big],
$$
where $N^{\pi_*\mu_i}_n(8\vep,1-2\delta-\frac14)$ stands for the minimal cardinality of an $(n,8\vep)$-generating subset for a subset of $\Sigma_\kappa$
of $(\pi_*\mu_i)$-measure at least $1-2\delta-\frac14$.
In particular, for each $\ell\geqslant 1$ 
and $\vep>0$ is small then there exist $(n,4\vep)$-separated sets $\Xi_{2\ell-1,n}\subset \Omega_1\cap \Xi_{2\ell-1}$ and $\Xi_{2\ell,n}\subset \Omega_2\cap \Xi_{2\ell}$ such that
\begin{equation*}\label{eq:Xi-card}
\# \Xi_{2\ell,n} \geqslant e^{ (H^{\text{Pinsker}}_\sigma(\psi)-2\gamma)n}
	\quad\text{and}\quad
	\# \Xi_{2\ell-1,n} \geqslant e^{ (H^{\text{Pinsker}}_\sigma(\psi)-2\gamma)n}
\end{equation*}
for every large $n\geqslant 1$ (depending on $\ell$).
Moreover, combining the latter with \eqref{eq:fibersuniformCantor}, 
and item (2) in Lemma~\ref{le:decoupled} 
one can produce a strictly increasing sequence $(n_\ell)_{\ell\geqslant 1}$ of positive integers 
satisfying \eqref{eq:select-nell}	and such that for each $\ell\geqslant 1$:
\begin{itemize}
\item[(i)] $\# \Xi_{2\ell,n_{2\ell}} \geqslant e^{ (H^{\text{Pinsker}}_\sigma(\psi)-2\gamma)n_{2\ell}}
	\quad\text{and}\quad
	\# \Xi_{2\ell-1,n_{2\ell-1}} \geqslant e^{ (H^{\text{Pinsker}}_\sigma(\psi)-2\gamma)n_{2\ell-1}}$
\item[(ii)] for each $\om\in \Xi_{2\ell-1}$ there exists a $(\om,n_{2\ell-1}, 4\vep)$-separated set 
\begin{equation*}\label{fibered-uniform-separated-sets}
		\mathcal{S}_{\om,2\ell-1} \subset X_\om \cap \Gamma_{2\ell-1}
		\quad\text{so that}\quad
		\#\mathcal{S}_{\om,2\ell-1} 	 \geqslant \exp(\,(H^{\text{Pinsker}}(\psi) -3\zeta) n_{2\ell-1}\,).
\end{equation*}
\item[(iii)] for each $\om\in \Xi_{2\ell}$ there exists a $(\om,n_{2\ell}, 4\vep)$-separated set 
\begin{equation*}\label{fibered-uniform-separated-sets}
		\mathcal{S}_{\om,2\ell} \subset X_\om \cap \Gamma_{2\ell}
		\quad\text{so that}\quad
		\#\mathcal{S}_{\om,2\ell} 	 \geqslant \exp(\,(H^{\text{Pinsker}}(\psi) -3\zeta) n_{2\ell}\,).
\end{equation*}
\end{itemize}

These informations can now be used to construct a Moran subset of $\Sigma_\kappa \times X$ of skew-product type whose both the projection on $\Sigma_\kappa$
and the fibers have large entropy (see Figure 8.1 below for an illustration). 
\begin{figure}[htb]
\begin{center}
      \includegraphics[scale = 0.45]{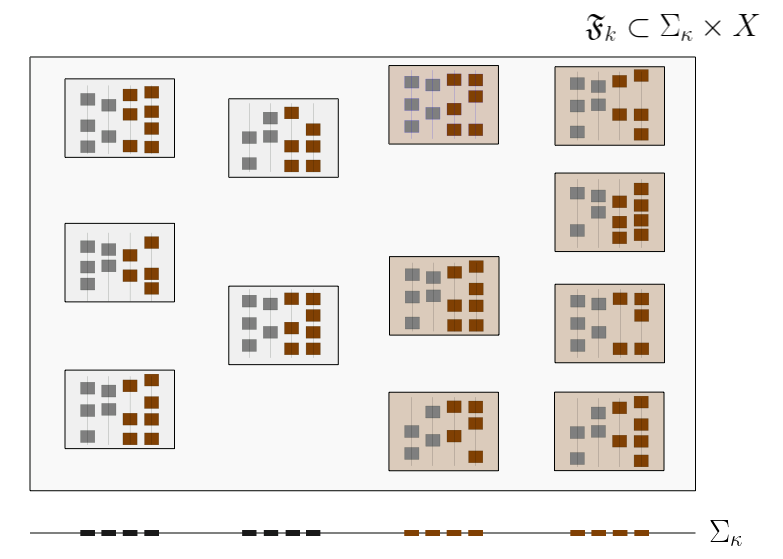}
        \caption{Construction of a specific Moran subset of $\Sigma_\kappa \times X$, formed from typical points for $\mu_1$ (represented in gray) and typical points for $\mu_2$ (represented in brown).}.
\end{center}
\end{figure}
More precisely:

\begin{proposition}\label{prop:skew-Cantor} 
There exists $K>0$ and a sequence $(\mathfrak S_k)_{k\geqslant 1}$ of compact subsets of $\Sigma_\kappa$ and a sequence $(\mathfrak F_k)_{k\geqslant 1}$ of compact subsets of $\Sigma_\kappa \times X$ so that, for each $k\geqslant 1$: 
	\begin{enumerate}
		\item $\mathfrak S_k$ is a finite union of cylinders in $\Sigma_\kappa$ with the same length, 
		\item $\mathfrak F_k \subset \pi^{-1}(\mathfrak S_k)$ and for each $\om\in \mathfrak S_k$ there exist $m_k=m_k(\om)\geqslant 1$ so that 
		$\lim_{k\to\infty} \frac{m_k}{n_k}=1$ and
		there exist $e^{  (H^{\text{Pinsker}}_\sigma(\psi) -4\gamma) \, [n_k+\dots + n_2+n_1]}$ points in $\mathfrak F_k \cap X_\om$ 
		that are $(\om,m_k,\vep)$-separated,
		\item The compact set 
		$$\mathfrak{F}:=\bigcap_{k\geqslant 1} \mathfrak{F}_k \subset \Sigma_\kappa\times X$$ is such that  
		$\mathfrak{F} \,\cap\, X_\om \subset I_\om(\psi)$ for every $\om\in \mathfrak S:=\bigcap_{k\geqslant 1} \mathfrak S_k$,
		\item $h_{I_\om(\psi)}(\mathcal F_\om) \geqslant H^{\text{Pinsker}}(\psi)-5\gamma$ for every $x\in \mathfrak S$.
	\end{enumerate}
\end{proposition}

\begin{proof}
This construction combines information from
Subsection~\ref{top-fibered1} and the sets described in items (i)-(iii) above.
For each $k\geqslant 1$, each $k$-uple $(\om_1, \om_2, \dots, \om_k) \in \Xi_{1,n_{1}} \times \Xi_{2,n_{2}} \times \dots\times \Xi_{k,n_{k}}$ 
and $(x_1, x_2, \dots, x_k)\in \mathcal{S}_{\om_1,1}\times \mathcal{S}_{\om_2,2} \times \dots \mathcal{S}_{\om_k,k}$
and each $1\leqslant i\leqslant k-1$ there exists an integer  $0\leqslant p_{i} \leqslant N_{L^{-2{n_{i}}}\vep}$ (depending on $\om_j$ and $x_j$, for $1\leqslant j\leqslant i-1$) 
so that
\begin{align*}
	& C_{(\om_1,  \kappa^{p_1},\om_2, \kappa^{p_2},...,\om_k)}(\,x_1, \dots, x_k\,) \\
	& \qquad  :=
	B_{\om_1}(x_1,n_1,\vep) \; \cap \; \bigcap_{i=1}^{k-1} 	
		(\, 
		f^{p_i}_{\kappa} \circ f_{\om_i}^{n_i}
		\circ \dots \circ 
		f^{p_1}_{\kappa} \circ f_{\om_1}^{n_1} 
		\,)^{-1} (\, B(x_{i+1},L^{-n_{i+1}}\vep)\,)
\end{align*}
is a non-empty subset of $X$. Proceeding as in the proof of Lemma~\ref{le:selection-fiber-omega}, one can extract 
a large subcollection of these sets for which the transition time functions depend only on $\om_1, \om_2, \dots, \om_k$. Indeed, 
for each $1\leqslant i \leqslant  k-1$ there exists 
$0\leqslant t_i=t_i((\om_1, \dots, \om_i)) \leqslant N_{L^{-2{n_i}}\vep}$ 
and a collection 
$\mathcal I_k(\om_1, \om_2, \dots, \om_{k}) \subset \mathcal{S}_{\om_1,1}\times \mathcal{S}_{\om_2,2} \times \dots \mathcal{S}_{\om_k,k}$ such that 
$$
\# \mathcal I_k(\om_1, \om_2, \dots, \om_{k}) \geqslant  K e^{  (H^{\text{Pinsker}}(\psi) -4\gamma) \,m_k }
$$
and
$$
p_i(\,x_1, \dots, x_i, \om_1, \dots, \om_i\,)=t_i \quad\text{for every} \; (\,x_1, \dots, x_k\,)\in \mathcal I_k(\om_1, \om_2, \dots, \om_{k}).
$$
Taking $m_k=n_k + \sum_{i=1}^{k-1} (n_i+t_i)$ (which is a function of $(\om_1, \om_2, \dots, \om_k)$), if
$(\,x_1, \dots, x_k\,), (\,y_1, \dots, y_k\,) \in \mathcal I_k(\om_1, \om_2, \dots, \om_k)$ are distinct $k$-uples and one considers points
$x\in C_{(\om_1,  \kappa^{t_1},\om_2, \kappa^{t_2},...,\om_k)}(\,x_1, \dots, x_k\,)$ and 
		$y\in C_{(\om_1,  \kappa^{t_1},\om_2, \kappa^{t_2},...,\om_k)}(\,y_1, \dots, y_k\,)$ then
$d_\om^{m_k}(x, y)>2\vep$ for any 
$$
\om\in \mathfrak S(\om_1, \om_2, \dots, \om_k):=({\om_1}_{[0,n_1]} \cdot \kappa^{t_1})
		\, \cap\,  
	\bigcap_{i\geqslant 1}  \sigma^{-\sum_{j=1}^i(t_j+n_j)} \;  ({\om_{i+1}}_{[0,n_{i+1}]} \cdot \kappa^{t_{i+1}}).
$$
Notice that the cylinder set $\mathfrak S(\om_1, \om_2, \dots, \om_k)\subset \Sigma_\kappa$ has size $m_k$, which is a function of the transition times 
$(t_1, t_2, \dots, t_{k-1})$ given by the pigeonhole principle applied to the fibered sets.  
There are 
$$
\#\Xi_{1,n_{1}} \times \#\Xi_{2,n_{2}} \times \dots\times \#\Xi_{k,n_{k}} \geqslant e^{ (H^{\text{Pinsker}}_\sigma(\psi)-2\gamma)[n_k+\dots + n_2+n_1]}
$$
such cylinders in $\Sigma_\kappa$ and, as $0\leqslant t_i \leqslant N_{L^{-2{n_i}}\vep}$, using the pigeonhole principle once more we get
a subcollection $\mathfrak S_k$ of such cylinder sets formed by cylinders in $\Sigma_\kappa$ of the same length so that
	\begin{align*}
		\#  {\mathfrak S_k} 
		 \geqslant \frac{e^{ (H^{\text{Pinsker}}_\sigma(\psi)-2\gamma)[n_k+\dots + n_2+n_1]}}
					{\prod_{i=1}^{k-1} \, N_{L^{-2{n_i}}\vep}} 
		 \geqslant 
	 	 e^{  (H^{\text{Pinsker}}_\sigma(\psi) -4\gamma) \, [n_k+\dots + n_2+n_1]}.
	\end{align*}

This allow us to consider also the nested sequence of compact sets $\mathfrak F_k\subset \Sigma_\kappa\times X$ defined by
\begin{equation*}\label{eq:Fk-fibered-on-product}
\mathfrak{F}_{k} = \bigcup_{\mathfrak S(\om_1, \om_2, \dots, \om_k) \in \mathfrak S_k} 
		\; \bigcup_{(\,x_1, \dots, x_k\,)\in \mathcal I_k(\om_1, \om_2, \dots, \om_{k})}
				\mathfrak S(\om_1, \om_2, \dots, \om_k) \times C_{(\om_1,  \kappa^{t_1},\om_2, \kappa^{t_2},...,\om_k)}(\,x_1, \dots, x_k\,),
\end{equation*}
as represented by Figure 8.1. 
We emphasize that the transition times $t_i$ are functions of $(\om_1, \om_2, \dots, \om_{i})$. 
These sets comply with the requirements of items (1) and (2) in the proposition.

By construction
$
\mathfrak{F}:=\bigcap_{k\geqslant 1} \mathfrak{F}_k
$
is a Moran set.  It is not hard to check that the same ideas explored in previous sections ensure that items (3) and (4) also hold, we shall omit the details.
\end{proof}

\section{Linear cocycles}\label{sec:linear-cocycles}

The main goal of this section is to prove Theorems~\ref{thm:abstract} and ~\ref{thm:rigidity}
on the irregular behavior and rigidity for linear cocycles.  First we shall comment on the fibered complexity of
the projective cocycle.

\subsection{Projective cocycles and complexities}\label{subsec:projectivelinear}

The projective cocycle $P_A : \Sigma_\kappa \times \mathbf P \mathbb R^d \to \Sigma_\kappa \times \mathbf P \mathbb R^d$ was defined in ~\eqref{eq:skewP}, and is a skew-product
with compact fibers. 
Consider the continuous potential $\varphi_A: \Sigma_\kappa \times \mathbf P \mathbb R^d \to \mathbb R$ given by 
\begin{equation}\label{eq:potential.cocycles}
\varphi_A(\om,v) = \log \frac{\|A(\om)\cdot v\|}{\|v\|},
	\quad \text{for}\; (\om,v) \in \Sigma_\kappa \times \mathbf P \mathbb R^d.
\end{equation}
Given $\nu\in \mathcal M_{erg}(\sigma)$, if the cocycle $A$ is {strongly irreducible} 
then Furstenberg formula ensures that
$$
\lambda_+(A,\nu)=\int \varphi_A(\om,v) \, d\eta(v)  d\nu(\om)
$$ 
where $\eta$ is any $\nu$-stationary measure (see \cite{Furstenberg} or \cite[Theorem~6.8]{{VianaB}}). While the previous expression allows to compute
the top Lyapunov exponent of the linear cocycle $A$ as an average of a continuous potential, it does not provide any immediate regularity of the 
Lyapunov exponent as a function of $\nu$ as this relies ultimately on the analysis of the space of $\nu$-stationary measures as $\nu$ varies. 
\begin{remark}\label{rmk:averages-cocycle}
As a consequence of the Oseledets theorem, if $\om$ belongs to a total probability subset of $\Sigma_\kappa$ then 
 for \emph{every} normalized vector $v\in \mathbb R^d$ the limit
$$
\lim_{n\to\infty} \frac1n\sum_{j=0}^{n-1} \varphi(F_A^j(\om,v))
	= \lim_{n\to\infty}  \frac1n \sum_{j=0}^{n-1} \log \frac{\|A(\sigma^j(\om))\cdot A^{(j)} (\om)v\|}{\|A^{(j)}(\om) v\|}
	= \lim_{n\to\infty}  \frac1n \log \frac{\|A^{(n)} (\om)v\|}{\|v\|}
$$
does exist and coincides with one of the Lyapunov exponents $(\lambda_i(\mu))_{1\leqslant i\leqslant d}$.
\end{remark}

We now observe that the non-wandering set of each projective map $P_{A_i}:\mathbf P \mathbb R^d \to \mathbf P \mathbb R^d$ ($1\leqslant i \leqslant \kappa$) 
is contained in the projectivization of invariant subspaces of $A_i$. In fact, as a consequence of the Jordan canonical form theorem, we have the following
possibilities for the non-wandering set of the restriction of each map $P_{A_i}$ to a projectivization of an eigenspace $E$ associated to $A_i$: 
(i)  if $A_i\mid_E$ has no nilpotent term then $P_{A_i}\mid_{\mathbf P E}$ is the identify or an isometry in case of a real or complex eigenvalue, respectively;
(ii) if $A_i\mid_E$ has nilpotent terms then there exists $F\subsetneq E$ so that ${\mathbf P F}$ is an attractor for $P_{A_i}\mid_E$ and 
$P_{A_i}\mid_{\mathbf P F}$ is the identify or an isometry in case of a real or complex eigenvalue, respectively.
In either case, the refinement of open coverings of $\mathbf P \mathbb R^d$ under the action of $P_{A_i}$ has sublinear growth, and the same holds for any
composition of such maps.Therefore, $h_\mu(P_A) = h_{\pi_*\mu}(\sigma) $ for every $\mu\in \mathcal M_{inv}(P_A)$ and every non-autonomous dynamical system generated by these have zero topological entropy, ie, $\htop(\mathcal F_\om)=0$.

Recall that for each $P_A$-invariant probability $\mu$ one has that $h_\mu(P_A) \leqslant h_{\pi_*\mu}(\sigma)+\int \htop(\mathcal F_\om) d\pi_*\mu$  (cf. \cite{LW}).
As $H^{\text{Pinsker}}(\varphi_A)=0$ then
\begin{align}
h_*(\varphi_A) 
& = \sup\Big\{ 
c\geqslant 0 \colon 
 \text{there exist} \, \mu_1,\mu_2 \in \mathcal M_{erg}(P_A) \,\text{so that} \hfill  \nonumber \\ 
& \hspace{2cm}  \, h_{\pi_*\mu_i}(\sigma)\geqslant c  \;\text{and}\, \int\varphi_A\, d\mu_1< \int\varphi_A\, d\mu_2
\Big\} \nonumber
\end{align}
is bounded above by $\htop(\sigma)=\log\kappa$.
The next lemma relates the top Lyapunov exponents with integrals for the projective cocycle by means of lifted measures.
\begin{lemma}\label{le:lift}
For every $\nu\in \mathcal M_{erg}(\sigma)$ there exists $\mu\in \mathcal M_{erg}(P_A)$ so that 
$
\int \varphi_A(\om,v)\, d\mu(\om,v)
	\lambda_+(A,\nu)
$
and
$
	\pi_*\mu=\nu.
$
\end{lemma}

\begin{proof}
First we construct invariant measures for the projective cocycle which can express the top Lyapunov exponent in an integral form.
Given $\nu\in \mathcal M_{erg}(\sigma)$, the Oseledets theorem ensures there exists $1\leqslant \tilde \kappa\leqslant \kappa$ and that for $\nu$-almost
every $\om\in \Sigma_\kappa$ there exists a splitting $\{\om\}\times \mathbb R^d=E^1_\om\oplus E^2_\om\oplus \dots \oplus E^{\tilde \kappa}_\om$, varying measurably with $\om$,  so that $\lambda_i(A,\mu)=\lim_{n\to\infty} \frac1n \log \|A^n(\om)v\|$ for every $v\in E^i_\om\setminus\{0\}$ and $1\leqslant i \leqslant \tilde \kappa$. As $\lambda^+(A,\nu)=\lambda_1(A,\nu)$ and $\omega\mapsto \mathbf P E^{1}_\om$ is a measurable family of submanifolds of 
$\mathbf P \mathbb R^d $ of constant dimension, 
one may consider the measurable family of probabilities $(\mu_\om)_\om$ so that $\supp \mu_\om \subset\{\om\}\times \mathbf P \mathbb R^d$ and given by
$$
\mu_\om(E) = 
\begin{cases}
\begin{array}{ll}
\text{Leb}\mid_{\{\om\} \times \mathbf P E^1_\om)}(B) &, \text{if}\; \dim E^1_\om>1 \\
\delta_{(\om,\mathbf P E^1_\om)}(B) &, \text{otherwise}
\end{array}
\end{cases}
$$
for every measurable set $B\in \{\om\} \times \mathbf P\mathbb R^d$. It is clear that the probability $\mu=\int \mu_\om\, d\nu(\om)$ is $P_A$-invariant, 
possibly non-ergodic. By some abuse of notation we still denote by $\mu$ an ergodic component of the previous probability. 
By construction, and Remark~\ref{rmk:averages-cocycle}, one can check that
$\int \varphi_A(\om,v)\, d\mu(\om,v)
	\lambda_+(A,\nu)
$ and $
	\pi_*\mu=\nu$,
as claimed.
\end{proof}

This lemma ensures that every $\sigma$-invariant and ergodic probability measure can be lifted to a $P_A$-invariant probability so that its Lyapunov exponent
can be computed as the time average of the potential $\varphi_A$ according to its lift. 
For general linear cocycles the upper semicontinuity of the top Lyapunov exponent 
can be used (see e.g. the proof of \cite[Proposition~2.12]{Tian2} to prove that 
\begin{equation}\label{eq:T}
h_*(\varphi_A)  
\geqslant \sup \Big\{h_\nu(\sigma) \colon \nu\in \mathcal M_{erg}(\sigma) \;\text{and}\, \lambda_+(A,\nu) >\inf_{\eta\in \mathcal M_{erg}(\sigma)} 
	\lambda_+(A,\eta)  \Big\}.
\end{equation}

\begin{corollary}\label{cor:crit}
Let $\nu_{max}$ be the unique measure of maximal entropy for the shift $\sigma: \Sigma_\kappa \to \Sigma_\kappa$. 
The following hold:
\begin{enumerate}
\item If $\lambda_+(A,\nu_{max}) >\inf_{\eta\in \mathcal M_{erg}(\sigma)} \{ \lambda_+(A,\eta)\} $ then $h_*(\varphi_A)  =\log \kappa$,
\item If the support of $\nu_{max}$ is not contained in a compact subgroup of $SL(d,\mathbb R)$, the cocycle is strongly irreducible
	and some of the matrices has norm one then $h_*(\varphi_A)  =\log \kappa$.
\end{enumerate}
\end{corollary}

\begin{proof}
Item (1) is a trivial consequence of \eqref{eq:T}. As for item (2) it is well known that the probability $\nu_{max}$ is the Bernoulli probability measure with equidistributed weights ie. $\nu_{max}=(\frac1\kappa,\frac1\kappa,\dots, \frac1\kappa)^{\mathbb Z}$. Hence, Furstenberg criterion for
positive Lyapunov exponents of random walks on groups of matrices in Theorem~\ref{thm:Furstenberg} ensures that, under the assumptions of
item (2) one has that $\lambda_+(A,\nu_{max}) >0=\inf_{\eta\in \mathcal M_{erg}(\sigma)} \{ \lambda_+(A,\eta)\}$. Item (2) now follows from item (1).
\end{proof}

We observe that while the previous corollary provides sufficient conditions for $h_*(\varphi_A)$ to be maximal, it is still unclear whether $h_{*}(\varphi_A)$ coincides with 
$\log\kappa$ this holds in full generality.  

\subsection{Lyapunov irregular behavior for linear cocycles}\label{sec:p-abstract}

\medskip
Here we prove Theorem~\ref{thm:abstract}.
Notice first that as the cocycle $A$ takes values in the subgroup $\mathscr G$ of $\mathrm{SL}(d,\mathbb R)$ then 
$\|A^n(\om)^{-1}\|^{-1} \leqslant 1 \leqslant \|A^n(\om)\|$ for every $\om\in\Sigma_\kappa$. In consequence, 
$\lambda_-(A,\mu)\leqslant 0 \leqslant \lambda_+(A,\mu)$ for every $\sigma$-invariant and ergodic probability measure $\mu$.
By assumption, the group generated by the matrices $\{A_1,A_2, \dots, A_\kappa\}$ is a non-compact subgroup of $SL(d, \mathbb R)$. In particular there
exists a matrix $A_j$ so that $\|A_j\|>1$. Assume, without loss of generality that $j=\kappa$. 

The argument to prove that the set of Lyapunov non-typicalpoints is a Baire residual subset of $\Sigma_\kappa$ is somewhat similar to the proof of Theorem~\ref{thm:IFS}, taking into account the fibered dynamics of the projectivized cocycle $P_A$, defined in ~\eqref{eq:skewP}..
Let $\mu$ denote the Dirac measure at $\underline \kappa=(\kappa,\kappa,\kappa, \dots)$.
As $\lambda_-(A,\mu)<\lambda_+(A,\mu)$ there exists a non-trivial 
hyperbolic splitting $\{\underline \kappa\} \times \mathbb R^d= E_{\underline \kappa}^s\oplus E_{\underline \kappa}^u$
for the matrix $A(\kappa)$, hence it admits a positively invariant unstable cone $C^+$ and a negatively invariant 
stable cone $C^-$.
Using adapted norms (see e.g. \cite{Shub}) it is folklore in the theory of uniform hyperbolicity that the previous cones can be chosen of arbitrarily close diameter, 
in order to guarantee the following:

\begin{lemma}\label{le:UH-cones}
There exists $\theta\in (0,1)$ and $\zeta>0$ so that the cones
$$
C^+=\big\{ w=w_+ \,+\, w_- \in E^+\oplus E^- \colon \|w_-\| \leqslant \zeta \|w_+\| \big\}
$$
and 
$$
C^-=\big\{ w=w_+ \,+\, w_- \in E^+\oplus E^- \colon \|w_+\| \leqslant \zeta \|w_-\| \big\}
$$
satisfy the following properties:
\begin{enumerate} 
\item[(i)] $A(\underline \kappa) (C^+)\subsetneq C^+$ and $A(\underline \kappa)^{-1} (C^-)\subsetneq C^-$,
\item[(ii)] $\|A(\underline \kappa)^n \cdot w\| \geqslant \theta^{-n} \|w\|$ for every $w\in C^+$ and $n\geqslant 1$,
\item[(iii)] $\|A(\underline \kappa)^{n} \cdot w\| \leqslant \theta^{n} \|w\|$ whenever $A(\underline \kappa)^{j} \cdot w \in C^-$
	for every $0\leqslant j \leqslant n$.
\end{enumerate}
\end{lemma}

In view of the previous lemma, one can construct points with Lyapunov irregular behavior provided that one can control recurrence along the fiber dynamics. 
By construction, the skew-product $P_A$ is locally constant and for each $1\leqslant j \leqslant \kappa$,
the map $f_j: \mathbf P \mathbb R^d \to \mathbf P \mathbb R^d$ defined by $f_j=P_A(\underline j,\cdot)$ (where $\underline j=(j,j,j, \dots)$) 
is a bi-Lipschitz  homeomorphism with Lipschitz constant $L:=\max_{\om\in \Sigma_\kappa} \frac{\|A(\om)\|\;\;}{\|A(\om)^{-1}\|^{-1}}$. 
Additionally,  by the uniform hyperbolicity of the matrix $A(\underline \kappa)$,
there exist open neighborhoods $D_A$ and $D_R$ of an attracting set $A$ and a repelling set $R$ for the projective map
$$
P_A(\underline \kappa, \cdot) : \mathbf P \mathbb R^d\to \mathbf P \mathbb R^d
	\quad\text{where}\quad
	P_A(\underline \kappa, v)= \frac{A(\underline \kappa)\cdot v}{\|A(\underline \kappa)\cdot v\|}
$$
obtained by the projectivization of unstable and stable cones in Lemma~\ref{le:UH-cones}, respectively.
Even though the attracting and repelling sets need not be transitive, the projective map 
$f_\kappa$ admit invariant and ergodic probability measures 
$\nu_+$ and $\nu_-$ supported on a topological attractor and repeller, respectively (see Figures 9.1 and 9.2 below).

\begin{figure}[htb]\label{fig-cocycles1}
\begin{center}
        \includegraphics[scale = 0.4]{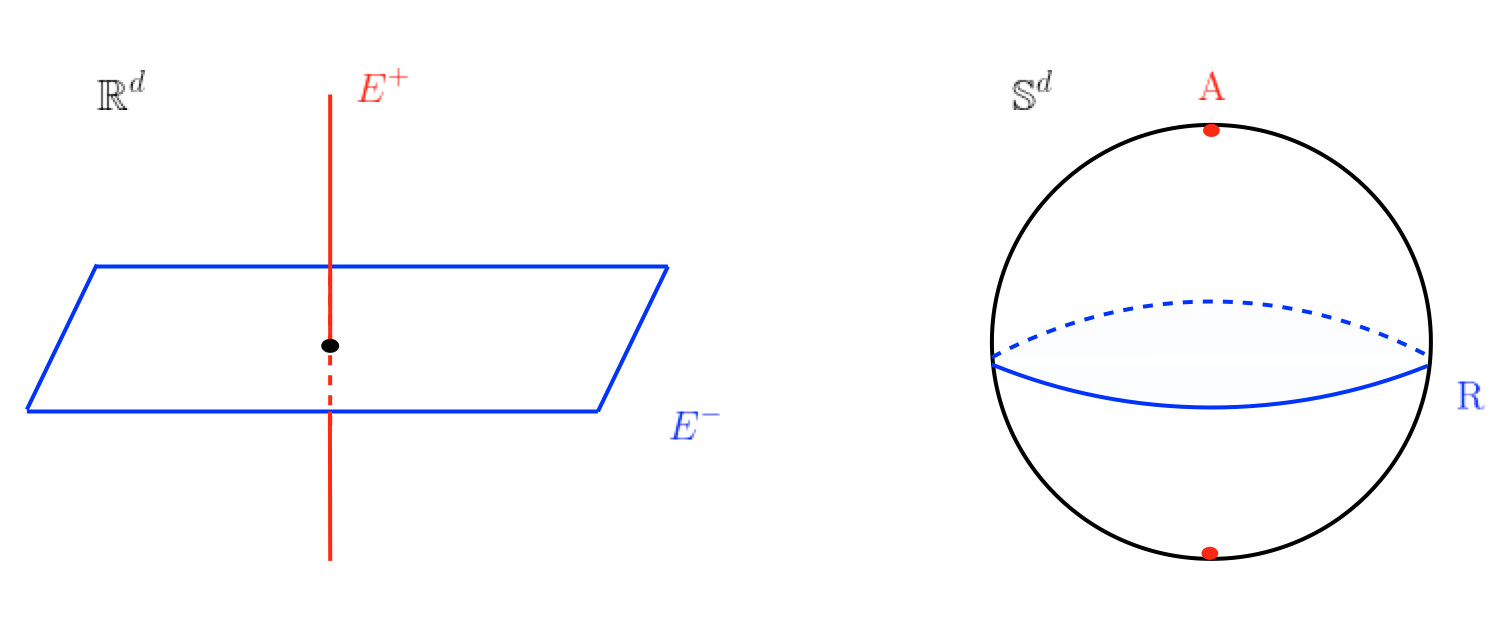}
        \caption{Unstable and stable eigenspaces for the hyperbolic linear map (on the left) generating attracting $A$ and repelling $R$ on the double covering ${\mathbf S}^d$ of the projective space (on the right)}.
\end{center}
\end{figure}

\begin{figure}[htb]\label{fig-cocycles2}
\begin{center}
                \includegraphics[scale = 0.4]{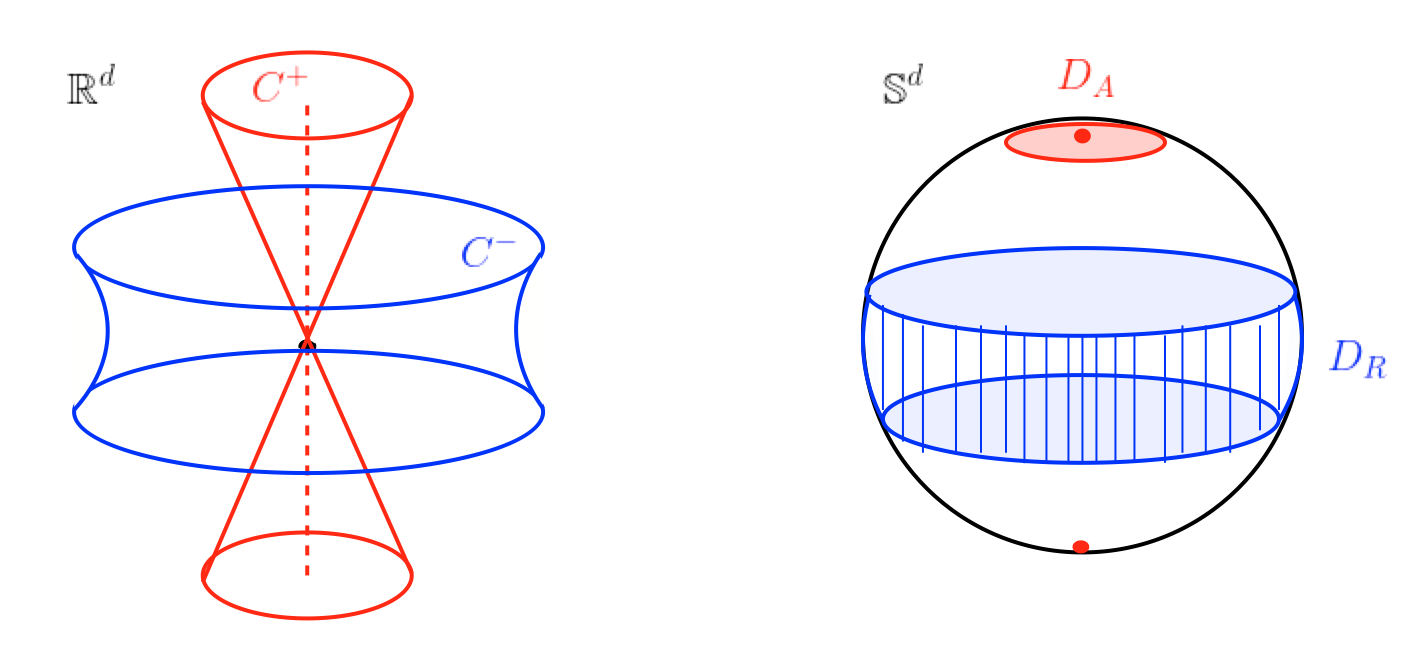}
        \caption{Unstable and stable cones for the original hyperbolic automorphism (on the left) inducing neighborhoods of $A$ and $R$ with attracting and repelling behavior for the projective map (on the right).}
\end{center}
\end{figure}

 By Lemma~\ref{le:UH-cones}, the continuous potential $\varphi_A:\Sigma_\kappa\times \mathbf P\mathbb R^d\to\mathbb R$ defined in
\eqref{eq:potential.cocycles} satisfies
$\varphi_A(\underline \kappa,w_1) \leqslant \log \theta < 0 < -\log \theta \leqslant  \varphi_A(\underline \kappa,w_2)$
for every $w_1\in D_R$ and $w_2\in D_A$. Thus,
\begin{equation}\label{eq:locus}
\int_{\Sigma_\kappa\times \mathbf P \mathbb R^d} \varphi_A(\om,v)\, d(\mu\times \nu_-) < \int_{\Sigma_\kappa\times \mathbf P \mathbb R^d} \varphi_A(\om,v)\, d(\mu\times \nu_+)
\end{equation}
The assumption that the cocycle is strongly projectively accessible is equivalent to say that the semigroup action generated by the collection $G_1=\{id,f_1, f_2, \dots, f_\kappa\}$ has frequent hitting times. While Theorem~\ref{thm:IFS} cannot be applied directly here (the potential $\varphi_A$ depends on both coordinates of the skew-product) the proof is identical, and consists of combining both accessibility of the base and fibered dynamics to produce points with erratic behavior.
Let us skim the details.

\medskip 
Fix an arbitrary $v\in \mathbf P \mathbb R^d$. We claim that there exists a Baire residual subset $\mathscr R_v\subset \Sigma_\kappa$ so that
the Ces\`aro averages
$
\frac1n\sum_{j=0}^{n-1} \varphi_A(F_A^j(\om,v))
$
diverge.
In contrast to the previous sections, the potential $\varphi_A$ does not depend exclusively on the variable in $\mathbf P \mathbb R^d$.  
Assumption ~\eqref{eq:locus} ensures that a source of different asymptotic behavior is located at the fiber
$\{\underline \kappa\} \times \mathbf P \mathbb R^d$. Hence, although the argument is similar to the one in previous sections, 
we have now to keep a control on the recurrence of the base dynamics as well. 
We will
detail the construction of a dense set of points in $\Sigma_\kappa$ with respect to which $v$ has Lyapunov irregular behavior,
as the remaining modifications in the construction of the Baire residual set will carry straightforwardly 
following the strategy used in Subsection~\ref{subsec2}.

\medskip
Choose $N_0\geqslant 1$ such that $2^{-N_0}<\vep$ and let $\zeta>0$ be given by Lemma~\ref{le:UH-cones}.
Given an arbitrary $\om^{(0)}\in \Sigma_\kappa$ and $0<\vep\ll \zeta$,
we claim that there exists a point $\om\in \Sigma_\kappa$
so that $d_{\Sigma_\kappa}(\om, \om^{(0)})<\vep$ but the Ces\`aro averages $\frac1n\sum_{j=0}^{n-1} \varphi_A(F_A^j(\om,v))$
diverge. Let $v_1 \in \mathbf P E_{\underline \kappa}^+$ and $v_2\in \mathbf P  E_{\underline \kappa}^-$  and 
take a sequence of positive integers $(n_k)_{k\geqslant 1}$ satisfying ~\eqref{def:nk}.
We first consider the pair
$$
\big(\, (\omega^{(0)}_1,\omega^{(0)}_2, \dots, \omega^{(0)}_{N_0})\; , \;  A^{(N_0)}(\om^{(0)}) \cdot v\,\big)
	\in \{1,2,\dots, \kappa\}^{N_0} \times \mathbb R^d.
$$

As the projectivized maps are bi-Lipschitz, the strong projective accessibility together with Lemma~\ref{le:est-balls}
implies that there exist a constant $K(L^{-2n_1}\vep)\geqslant 1$, a fibered transition time $0\leqslant p_1 \leqslant K(L^{-2n_1}\vep)$ 
and a finite word
$(\omega^{(1)}_1,\omega^{(1)}_2, \dots, \omega^{(1)}_{p_1}) \in \{1, 2, \dots, \kappa\}^{p_1}$ so that  
$v_1:=A^{(p_1)}(\omega^{(1)}_1,\omega^{(1)}_2, \dots, \omega^{(1)}_{p_1})(A^{(N_0)}(\om^{(0)}) \cdot v)\in D_A$. 
By Lemma~\ref{le:UH-cones},
$$
\log \frac{\|A^{(n_1)} (\underline\kappa)\cdot v_1\|}{\|v_1\|} \geqslant \theta^{-1} n_1.
$$
Now, using once more the strong accessibility property, there exists 
$0\leqslant p_2 \leqslant K(L^{-2n_2}\vep)$
and a sequence $(\omega^{(2)}_1,\omega^{(2)}_2, \dots, \omega^{(2)}_{p_2}) \in \{1, 2, \dots, \kappa\}^{p_2}$ 
so that  
$$
v_2:=A^{(p_2)}(\omega^{(2)}_1,\omega^{(2)}_2, \dots, \omega^{(2)}_{p_2})(v_1)
	\in 
	C^-
$$
and $v_2 \in B(v_-, L^{-2n_2}\vep) \subset \mathbf P (\mathbb R^d \cap C^-)$ (here and throughout, by some abuse of notation, we keep denoting by $v$ the equivalence class of the vector $v\in \mathbb R^d$
in $\mathbf P \mathbb R^d$ ).
Since the maps $f_j$ are $L$-Lipschitz, $v_-$ is a fixed point for $f_\kappa$,  
we conclude that $f_\kappa^j (v_2) \in B(v_-,\zeta)$ for every $0\leqslant j \leqslant n_2$.
Equivalently,  $A^j(\underline \kappa)\cdot v_2 \in C^-$ for every $0\leqslant j \leqslant n_2$.
Using Lemma~\ref{le:UH-cones} once more,
$$
\log \frac{\|A^{(n_2)} (\underline\kappa)\cdot v_2\|}{\|v_2\|} \geqslant \theta n_2.
$$
Repeating this procedure inductively and uniform continuity of $\varphi_A$ we obtain that the sequence
	\begin{equation*}\label{eq:sequence-cocycles}
	{\underline{\omega}} = ( [\om_0]_{[-N_0,N_0]},
					 \underbrace{\kappa, \kappa, \dots \kappa}_{n_1 }, 
					\omega^{(1)}_{1} \omega^{(1)}_{2}  \dots \omega^{(1)}_{{p_1}}, 
					\underbrace{\kappa, \kappa, \dots \kappa}_{n_2 }, 
					\omega^{(2)}_1 \omega^{(2)}_{2}  \dots \omega^{(2)}_{{p_2}}, 
					\underbrace{\kappa, \kappa, \dots \kappa}_{n_3 }, \ldots)
	\end{equation*}
is an element in $\Sigma_\kappa^+$ according to which
$$
\liminf_{n\to\infty} \frac1n\sum_{j=0}^{n-1} \varphi_A(F_A^j(\om,v))
	< \limsup_{n\to\infty} \frac1n\sum_{j=0}^{n-1} \varphi_A(F_A^j(\om,v)).
$$
As $\vep>0$ and $\om^{(0)}\in \Sigma_\kappa$ where chosen arbitrary this proves that the set of points $\om\in \Sigma_\kappa$ for which $v$ has Lyapunov non-typical behavior along the orbit of $\om$ is dense in $\Sigma_\kappa$.
As mentioned before, the modifications to obtain a Baire residual subset of $\Sigma_\kappa$ along the strategy used in Subsection~\ref{subsec2} are left as an easy but thorough 
exercise to the interested reader. 
The second statement in the theorem is an immediate consequence of the first one. Indeed, by separability, there exists a countable and dense subset 
$\mathscr D\subset \mathbf P \mathbb R^d$. Then, 
$
\mathscr R=\bigcap_{v\in \mathscr D} \; \mathscr R_v
$
is Baire generic in $\Sigma_\kappa$ (it is a countable intersection of Baire generic subsets of $\Sigma_\kappa$) 
which satisfies the requirements of the theorem.

\begin{remark}
In the special context of $SL(2,\mathbb R)$-cocycles satisfying the assumptions of Theorem~\ref{thm:abstract},
the top Lyapunov exponent associated to the ma\-xi\-mal entropy measure for the 
shift is strictly positive (cf. Theorem~4 (f) in \cite{DGR}) and, consequently,  
$h_*(\varphi_A)=\log \kappa$ (recall Corollary~\ref{cor:crit}(i)). 
While Theorem 1.1 in \cite{Tian2} guarantees
that the set of points in $\Sigma_\kappa$ for which the top Lyapunov exponent is not well defined have full topological entropy, 
our arguments show that there exist a full topological entropy subset of $\Sigma_\kappa$ whose Lyapunov exponents are not well defined for 
a dense set of directions in $\mathbb R^2$.
\end{remark}

\subsection{A dichotomy leading to Lyapunov irregular behavior}\label{sec:p-rigidity}

Let us prove Theorem~\ref{thm:rigidity}.
Assume that $A: \Sigma_\kappa \to SL(3,\mathbb R)$ is a locally constant hyperbolic cocycle and that its hyperbolic spliting $\Sigma_2\times \mathbb R^3=E^s\oplus E^u$ 
satisfies $\dim E^s=2$.
By the $C^0$-robustness of uniform hyperbolicity, which can be defined in terms of cone fields,
there exists a $C^0$-open neighborhood $\mathscr U \subset C_{\text{loc}}(\Sigma_\kappa,SL(3,\mathbb R))$ 
formed by uniformly hyperbolic cocycles. 
Moreover, in the simpler context of locally constant cocycles, it is immediate to check that the proof of \cite[Theorem~2.8]{BRV}  
guarantees there exists a $C^0$-open and dense subset $\mathscr V\subset \mathscr U$ such that $\mathscr V=\mathscr V_1\cup \mathscr V_2$,  
each $\mathscr V_i$ is a $C^0$-open set and
either: (i) a cocycle $B\in \mathscr V_1$ has one-dimensional finest dominated splitting 
(ii) for each $B\in \mathscr V_2$ there exist two periodic points $\om,\tilde \om \in \Sigma_\kappa$ of periods $p,\tilde p\geqslant 1$, respectively, 
such that $B^{p}(\tilde \om)\mid_{E^s_{B,\tilde \om}}$ has two real eigenvalues and $B^{p}(\om)\mid_{E^s_{B,\om}}$ has a complex eigenvalue
(notice that the $C^1$-topology in the case of diffeomorphisms in \cite{BRV} translates as $C^0$-topology for the linear derivative cocycles $Df$). 
The hyperbolic subspace ${E^s_{B,\om}}$ varies continuously with the cocycle $B$ and the restriction of $B^p$ to that subspace is completely determined
by its complex eigenvalue (up to conjugacy). More precisely,
for each $B\in \mathscr V_2$ there exists $M_B\in GL(3,\mathbb R)$ so that $M_B\mid_{\mathbb R^2\times\{0\}}=E^s_{B,\om}$ and 
$$
(M_B^{-1}B M_B)\mid_{\mathbb R^2\times\{0\}}	
= \lambda_B \begin{pmatrix}
\cos \theta_B & -\sin \theta_B\\
\sin \theta_B & \cos \theta_B 
\end{pmatrix},
$$
where $|\lambda_B|<1$ and $\theta_B\neq 0$.
It makes sense to consider submersion 
$$
\varrho: \mathscr V_2 \to \mathbb R \quad\text{given by}\quad
	B\mapsto \rho(B^{p}(\om)\mid_{E^s_{B,\om}}):=\theta_B,
$$
where $\rho(B^{p}(\om)\mid_{E^s_\om})$ denotes the rotation number of the projectivization of $B^{p}(\om)\mid_{E^s_\om}$ on $\mathbf P {E^s_\om} \simeq\mathbf S^1$. In particular,  
$\{ B\in \mathscr V_2 \colon \rho(B^{p}(\om)\mid_{E^s_\om}) \in \mathbb Q\}$ is a meager subset of $\mathscr V_2$.
As $B\in SL(3,\mathbb R)$, the eigenvalue $\lambda_B$ is uniquely determined by the real eigenvalue on the unstable subbundle. 
Hence, up to conjugacy the matrix $B\mid_{E^s_{B,\om}}$ is completely determined by $\theta_B$ and, consequently, 
$$
\{B \in \mathscr V_2 \colon \rho(B^{p}(\om)\mid_{E^s_{B,\om}}) \in \mathbb Q  \}
$$
is a zero Haar measure subset of $\mathscr V_2$. 
In particular,  there exists a $C^0$-Baire residual and full Haar measure subset $\mathscr R\subset \mathscr U$ so that for every 
$B\in \mathscr R$ 
either:
\begin{enumerate}
\item there are two periodic points for $\sigma$ whose projective matrices correspond to a irrational rotation and 
a Morse-Smale circle diffeomorphism, or
\item has a finest dominated splitting in three one-dimensional subbundles. 
\end{enumerate} 
Up to a continuous change of coordinates we may assume that the stable bundle $E^s_{B,\om}$ does not depend on $\omega$, meaning that
it coincides with a single eigenspace $E^s_B$.  In the first case the cocycle has  strong projective accessibility and by the last statement of Theorem~\ref{thm:abstract}, making use of the previous identification of the subbundle $E^s$ with a single subspace $\mathbb E_B^s$, there exists a dense subset $\mathscr D\subset \mathbf P\mathbb E^s$
	 and a Baire residual subset $\mathscr R\subset \Sigma_\kappa$ so that 
	$$
		\liminf_{n\to\infty} \frac1n \log \|B^n(\om) v\| < \limsup_{n\to\infty} \frac1n \log \|B^n(\om) v\|
	$$
holds for every 
		$\om\in \mathscr R$ and every $v\in \mathscr D$,
as claimed. This completes the proof of the theorem.

\subsection{The absence of Lyapunov irregular behavior is rare}

In the case of locally constant cocycles one can actually show that the absence of a large set of points with Lyapunov irregular behavior leads to a strong rigidity
of the cocycle, namely it is cohomologous to a constant matrix cocycle.  Recall
that two cocycles $A,B: \Sigma_\kappa \to \mathscr G$ over the shift $\sigma$ are \emph{cohomologous} if there exists a change of variables
$C: \Sigma_\kappa \to \mathscr G$ so that 
\begin{equation*}\label{eq:cohom}
A(\om)=C(\sigma(\om)) B(\om) C(\om)^{-1}
	\quad\text{for every } \om\in \Sigma_\kappa.
\end{equation*} 
We are interested in describing geometric conditions under which the dichotomy in the context of Theorem~\ref{thm:rigidity} is mutually exclusive. 
We observe that, in general, periodic points having the same eigenvalues is a condition substantially weaker than saying that the cocycle is cohomologous to a constant cocycle (see e.g. \cite{Sad1}). In other words, the abelian and non-abelian Liv\v{s}ic theorems differ substantially. In the three-dimensional context of Theorem~\ref{thm:rigidity}, due to the diagonalization and absence of nilpotent terms this turns out to be the case. More precisely:

\begin{proposition}\label{cor:rigidity}
Let $\mathscr R\subset \mathscr U$ be the Baire residual subset given by Theorem~\ref{thm:rigidity} and let $A\in \mathscr R$
be such that it admits a finest dominated splitting in three one-dimensional subbundles. The following are equivalent:
\begin{enumerate}
\item the set of Lyapunov non-typical points is empty,
\item all periodic points $\om,\widetilde{\om}\in \Sigma_\kappa$ have the same Lyapunov spectrum,
\item every cocycle $A\in \mathscr R$ is cohomologous to a constant cocycle with simple Lyapunov spectrum.
\end{enumerate}
If any of the three properties above fails then the set of Lyapunov non-typical
points is a Baire generic subset of $\Sigma_\kappa$.
\end{proposition}

\begin{proof}
Let $A\in \mathscr R$  be as above. We prove the implications $(1) \Rightarrow (2) \Rightarrow (3) \Rightarrow (1)$  separately.

\medskip
\noindent $(1) \Rightarrow (2)$ 
\smallskip

As $A\in \mathscr R$ admits a finest dominated splitting in three one-dimensional subbundles 
$T\mathbb R^3 = E^u\oplus E^{s,1}\oplus E^{s,2}$, each of the subbundles is invariant by the action of the cocycle and 
varies H\"older continuously with the base point.
 This makes the cocycles diagonalizable. Indeed, for each $\om\in\Sigma_\kappa$ there exists a base $\mathcal B_1=(v_1(\om), v_2(\om),v_3(\om))$ of $\mathbb R^3$, varying H\"older continuously with $\om$, so that the change of basis matrix $C(\om)$ between the canonical basis and 
$\mathcal B_1$ is such that the continuous cocycle
$
\tilde A(\om):=C(\sigma(\om))^{-1}\, A(\om)\, C(\om)
$
preserves the canonical axes, hence  
\begin{equation}\label{diag}
\tilde A(\om)
=
\begin{pmatrix}
a_1(\om)& 0&0\\
0& a_2(\om) &0 \\
0 &0 & a_3(\om) 
\end{pmatrix},
\qquad \om\in \Sigma_\kappa.
\end{equation}
Assume now that there exist periodic points $\om,\widetilde{\om}\in \Sigma_\kappa$ with different Lyapunov spectrum.
As the shift $\sigma$ satisfies the specification property 
the Birkhoff averages of some of 
the observables $\log a_i(\cdot)$ ($i=1,2,3$) differ at those periodic points, then the set of irregular points is actually Baire generic, hence non-empty 
(see e.g. \cite{Barreira1}). 

\medskip
\noindent $(2) \Rightarrow (3)$ 
\smallskip

As the cocycle $A$ is diagonalizable, we will assume it is conjugate to $\tilde A$ given by \eqref{diag}. In this context, the assumption is equivalent to say 
that the Birkhoff averages of the continuous observable $\log a_i(\cdot)$ is equal to $\lambda_i$ at all periodic points, for $i=1,2,3$. 
Then, the classical Liv\v{s}ic theorem for H\"older continuous real valued cocycles ensures that for each $i=1,2,3$ there exists a 
H\"older {continuous} function $u_i: \Sigma_\kappa \to \mathbb R$ so that $\log a_i= u_i - u_i\circ \sigma +\lambda_i$ or,
equivalently, $a_i=e^{- u_i\circ \sigma} e^{\lambda_i} e^{u_i}$. Altogether, taking
$$
U(\om)=	
\begin{pmatrix}
e^{u_1(\om)}& 0&0\\
0& e^{u_2(\om)} &0 \\
0 &0 & e^{u_3(\om)}
\end{pmatrix}
$$
we conclude that
\begin{equation*}
[C(\sigma(\om)) U(\sigma(\om))]^{-1}\, A(\om)\, [C(\om) U(\om)]
=
\begin{pmatrix}
\lambda_1 & 0&0\\
0& \lambda_2 &0 \\
0 &0 & \lambda_3
\end{pmatrix}
\qquad 
\end{equation*}
is a constant matrix for every $\om\in \Sigma_\kappa$, which proves that $A$ is cohomologous to a constant cocycle.

\medskip
\noindent $(3) \Rightarrow (1)$ 
\smallskip

Assume there exists a matrix $B_0\in \mathscr G$ and a measurable $C: \Sigma_\kappa \to \mathscr G$ so that
$A(\om)=C(\sigma(\om)) \,B_0\, C(\om)^{-1}$ {for every} $\om\in \Sigma_\kappa$. In this case 
$A^{(n)}(\om)= C(\sigma^n(\om)) B_0^n C(\om)^{-1}$ for every $n\geqslant 1$, and the Lyapunov exponents of $A$
are well defined at all points and coincide with the logarithm of the modulus of the eigenvalues of the matrix $B_0$.

\medskip
Finally, the last statement in the corollary is a consequence of \cite{Barreira1}. 
\end{proof}

\subsection{Projective accessibility and strong irreducibility}

In this short subsection we provide simple examples to relate the assumptions of Furstenberg theorem (Theorem~\ref{thm:Furstenberg}) and Theorem~\ref{thm:abstract}. The first example 
shows that the strong projective accessibility property on the cocycle is a sufficient but not necessary condition for the irregular behavior to be prevalent in the projective space. 

\begin{example}\label{ex:different-types}
Consider the following $SL(3,\mathbb R)$ matrices
$$
A_1
	= 
\begin{pmatrix}
-3 & 0&0\\
0& \frac1{\sqrt3} & \frac1{\sqrt3}\\
0 & \frac1{\sqrt3} & 0
\end{pmatrix},
	\quad
A_2
	= 
\begin{pmatrix}
3 & 0&0\\
0& \frac13 \cos \theta & -\frac13\sin \theta\\ \smallskip
0 & \frac13\sin \theta & \; \frac13\cos \theta
\end{pmatrix}
	\quad\text{and}\quad
A_3
	= 
\begin{pmatrix}
-2 & 0&0\\
0& \frac1{\sqrt2} & \frac1{\sqrt2}\\
0 & \frac1{\sqrt2} & 0
\end{pmatrix},
$$
where 
$\theta\notin \mathbb Q$. 
Let $A: \Sigma_2 \to SL(3,\mathbb R)$ be the locally constant linear cocycle determined by $A\mid_{[i]}=A_i$, $i=1,2$. 
It is clear that for every Bernoulli probability measure
$\nu$ on $\Sigma_2$ one has that $\lambda_+(A,\nu)=\log 3 >0$ and, 
$$
\lim_{n\to\infty} \frac1n \log \| A^n(\om) v \| = \log 3\,
\quad\text{for \emph{every}} \; \om\in \Sigma_2 \; \text{and}\; v \in \mathbb R^3\setminus <e_1>,
$$
where $<e_1>$ stands for the one dimensional subspace generated by $e_1=(1,0,0)$.
In particular, there is no Lyapunov irregular behavior concerning the top Lyapunov exponent and all vectors 
there exists an open and dense subset of $\mathbf P \mathbb R^3$ formed by vectors whose Lyapunov exponent is well defined. 
Still, Theorem~\ref{thm:abstract} guarantees that there exists a dense subset $\mathscr D\subset \mathbf P (\{0\} \times \mathbb R^2)$
so that for for each $v\in \mathscr D$ there exists a Baire generic subset of points $\om\in\Sigma_2$ so that $\lim_{n\to\infty} \frac1n \log \| A^n(\om) v \|$
does not exist.

\smallskip
Consider, alternatively, the locally constant cocycle $B: \Sigma_2 \to SL(3,\mathbb R)$ given by $B\mid_{[1]}=A_1$ and $B\mid_{[2]}=A_3$.
In this case, using that $\mathbb R^3=<e_1> \oplus <e_1>^\perp$ is a dominated splitting for the cocycle, it is simple to verify that for each $v\in \mathbf P\mathbb R^3 \setminus <(1,0,0)>$
$$
\limsup_{n\to\infty} \frac1n\log \|A^n(\om)v\| = \log 3 
	\quad\text{if and only if}\quad 
	\limsup_{n\to\infty} \frac1n\# \{0\leqslant j \leqslant n-1 \colon \sigma^j(\om)\in[1]\} = 1
$$
and
$$
\liminf_{n\to\infty} \frac1n\log \|A^n(\om)v\| = \log 2 
	\quad\text{if and only if}\quad 
	\limsup_{n\to\infty} \frac1n\# \{0\leqslant j \leqslant n-1 \colon \sigma^j(\om)\in[1]\} = 0.
$$
Both properties on the right hand-side above hold for a Baire generic subset of points in $\Sigma_2$
as a consequence of \cite[Theorem~3.1]{Barreira1}. 
\end{example}

\begin{example}\label{ex:irreduciblevsaccessible}
Consider the $SL(2,\mathbb R)$ matrices
$$
A_1
	= 
\begin{pmatrix}
2 & 1\\
1& 1 
\end{pmatrix}
	\quad\text{and}\quad
A_2
	= 
\begin{pmatrix}
 \cos \theta & -\sin \theta\\ \smallskip
\sin \theta & \; \cos \theta
\end{pmatrix},
$$
and take the locally constant cocycles $A,B : \Sigma_2 \to SL(2,\mathbb R)$ given by
and $A\mid_{[1]}=B\mid_{[1]}=A_1$, $A\mid_{[2]}=A_2$ and $B\mid_{[2]}=A_2 \cdot A_1$, for some small $\theta\notin \mathbb Q_+$. 
The cocycle $A$ is projectively accessible as $A_2$ induces due a minimal rotation on $\mathbf P \mathbb R^2$, hence it is strongly irreducible. 
In turn, $B\mid_{[1]}$ and $B\mid_{[2]}$ do not preserve any finite union of one-dimensional subspaces, hence the cocycle $B$ is strongly irreducible. 
Moreover, it is clear that $B$ is a hyperbolic cocycle and, consequently, the projectivizations of these matrices preserve admit a proper attracting  
set in $\mathbf P \mathbb R^2$. Thus $B$ is not strongly projectively accessible.
\end{example}

In the next example we derive an application to the context of symplectic matrices. 

\begin{example}
Given a fixed $\kappa\geqslant 1$ consider the closed subset 
\begin{align}
\mathscr U & =\{(A_1, \dots, A_\kappa) \in \text{Sp}(4,\mathbb R)^{\kappa} \colon \exists 1\leqslant i \leqslant \kappa\; \text{so that}\;\|A_i\|=1\}  \nonumber\\
		&= \bigcup_{i=i}^\kappa \, \{(A_1, \dots, A_\kappa) \in \text{Sp}(4,\mathbb R)^{\kappa} \colon \|A_i\|=1\}, \label{eq:union}
\end{align}
which is a finite union of closed smooth submanifolds 
of the Lie group $\text{Sp}(4,\mathbb R)^{\kappa}$, endowed with the Haar measure. 
We claim that there exists a Baire generic and full Haar measure subset of $\kappa$-uples 
$(A_1, \dots, A_\kappa) \in \mathscr U$ such that either:
\begin{enumerate}
\item the closure of the subgroup of $\text{Sp}(4,\mathbb R)^{\kappa}$ generated by $\{A_1, \dots, A_\kappa\}$ is compact; or 
\item there exists a Baire residual subset $\mathscr R\subset \Sigma_\kappa$ 
and a dense subset $\mathscr D\subset \mathbf P\mathbb R^d$ so that the Lyapunov exponents are not well defined 
along the direction determined by $v\in \mathscr D$, for every $\om\in \mathscr R$.
\end{enumerate}
Using the decomposition ~\eqref{eq:union}, it is enough to prove that the previous dichotomy holds for 
a Baire generic and full Haar measure subset of
$
\{(A_1, \dots, A_\kappa) \in \text{Sp}(4,\mathbb R)^{\kappa} \colon \|A_i\|=1\},
$
for every $1\leqslant i \leqslant \kappa$. 
Every $4\times 4$ symplectic matrix of norm one is parameterized in a one to one way (up to conjugacy) by their eigenvalues $e^{\pm \theta_1 {\bf i}}=a\pm {\bf i}\, b$ and $e^{\pm \theta_2 {\bf i}}=c\pm {\bf i}\, d$ 
(recall Example~\ref{ex:symplectic}). Hence, as 
$$
\{ (\theta_1,\theta_2)\in \mathbb R^2 \colon \exists m,n\in \mathbb Z\; \text{so that}\; (a\pm {\bf i}\, b)^n (c\pm {\bf i}\, d)^m=1 \}
$$
is a meager and zero Lebesgue measure subset of $\mathbb R^2$ then 
the set $\mathcal S_i$ of $\kappa$-uples $(A_1, \dots, A_\kappa) \in \text{Sp}(4,\mathbb R)^{\kappa}$ so that the projectivization of the map $A_i$
is a minimal isometry is Baire generic and has full Lebesgue measure in $\{(A_1, \dots, A_\kappa) \in \text{Sp}(4,\mathbb R)^{\kappa} \colon \|A_i\|=1\}.$
For such $\kappa$-uples the projective cocycle is strongly projectively accessible.
In particular either 
the closure of the subgroup of $\text{Sp}(4,\mathbb R)^{\kappa}$ generated by $\{A_1, \dots, A_\kappa\}$ is compact or the assumptions of Theorem~\ref{thm:abstract}  are met, proving the dichotomy.
\end{example}

\begin{example}\label{ex:meager}
Let $d\geqslant 2$ and $A: \Sigma_\kappa \to \mathrm{SL}(d,\mathbb R)$ be a H\"older continuous 
hyperbolic cocycle admitting a hyperbolic spliting $\mathbb R^d=E_\om^s\oplus E_\om^u$ for every $\om\in \Sigma_\kappa$, with $\dim E_\om^u=1$ 
and such that the positive Lyapunov exponent of all periodic points is a constant $\lambda>1$. The hyperbolic subbundles are well known to be H\"older continuous
\cite{KH}.
Thus, Liv\v{s}ic's theorem ensures that the H\"older continuous observable
$\Sigma_\kappa\ni \om \mapsto \log \|A\mid_{E^u_\om}(\om)\|$ is cohomologous to the constant observable $\log\lambda$ by a H\"older continuous function. In consequence, there exists an open and dense subset of directions in the projectivized space for which the top Lyapunov exponent is well defined at all points: 
for each $\om\in \Sigma_\kappa$ the following limit exists 
$$
\lim_{n\to+\infty} \frac1n \log \|A^{(n)}(\ud\omega)v\| = \log \lambda
	\quad \text{for every} \; v\in (\{\om\}\times \mathbb R^d) \setminus E^s_{\om}.
$$
In consequence, the set of Lyapunov irregular points for the top Lyapunov exponent is meager.
We note that the limit may not exist if one considers alternatively negative iteration, 
in which case the iteration all vectors that do not lie in the one-dimensional
sub-bundle $E^s_{\om}$  approximate the codimension one stable subbundle.
\end{example}

\bigskip
\subsection*{Acknowledgments} This work was initiated during a postdoctoral position of GF at the Federal University 
of Bahia, supported by INCTMat-Brazil. PV was partially supported by the project `New trends in Lyapunov exponents' (PTDC/MAT-PUR/29126/2017), by CMUP (UID/MAT/00144/2013), 
and by Funda\c c\~ao para a Ci\^encia e Tecnologia (FCT) - Portugal, through the grant CEECIND/03721/2017 of the Stimulus of Scientific Employment, Individual Support 2017 Call. The authors are grateful to  Z.~Buczolich, M.~Carvalho, D. Ma, P. Sun and X. Tian for useful comments and references.



\begin{thebibliography}{00}

\bibitem{AR}
L. Abramov and V. Rokhlin,
The entropy of a skew product of measure-preserving transformations,
\emph{Amer. Math. Soc. Transl.} Ser. 2, 48 (1966), 255--265.

\bibitem{AAN}
E. Akin, J. Auslander and A. Nagar,
\newblock Variations on the concept of topological transitivity,
\newblock  \emph{Studia Math.} 235 (2016), 225--249. 

\bibitem{AP19}
V. Ara\'ujo and V. Pinheiro.
Abundance of wild historic behavior.
\emph{Bull. Braz. Math. Soc.}, 52 (2019), 41--76. 



\bibitem{Barreira}
L. Barreira and J. Schmeling, Sets of non-typicalpoints have full topological entropy and full Hausdorff dimension, 
\emph{Israel J. Math.} 116 (2000), 29--70.

\bibitem{Barreira1}
L. Barreira, J. Li and C. Valls, 
Irregular sets are residual. \emph{Tohoku Math. J.} 2:66 (2014) 471--489.




\bibitem{BRV}
M. Bessa, J. Rocha and P. Varandas,
\newblock Uniform hyperbolicity revisited: index of periodic points and equidimensional cycles, 
\newblock \emph{Dynam. Sys.},  33:4 (2018) 691--707.

\bibitem{BeTV}
M. Bessa, J. Torres and P. Varandas,
\newblock Topological features of flows with the reparametrized gluing orbit property, 
\newblock  \emph{J. Diff. Eq.}, 262:8 (2017) 4292--4313. 



\bibitem{Bis2013}
A. Bi\'s.
\newblock {An analogue of the variational principle for group and pseudogroup actions.}
\newblock \emph{Ann. Inst. Fourier} 63:3 (2013) 839--863.

\bibitem{Bis2018}
A. Bi\'s.
Topological and measure-theoretical entropies of nonautonomous dynamical systems,
\emph{J. Dyn. Diff. Equat.} (2018) 30:273--285



\bibitem{Bo92}
T. Bogensch\"utz, 
Entropy, pressure, and a variational principle for random dynamical systems, 
\emph{Random Comp. Dynam.} 1 (1992) 219--227


\bibitem{BC}
T. Bogensch\"utz and H. Crauel, 
\emph{The Abramov-Rokhlin formula}. In: Krengel U., Richter K., Warstat V. (eds) Ergodic Theory and Related Topics III. 
Lecture Notes in Mathematics, vol 1514.  (1992) Springer, Berlin, Heidelberg. 



\bibitem{BV14}
T. Bomfim and P.~Varandas,
\newblock Multifractal analysis for weak Gibbs measures: from large deviations to irregular sets.
\newblock {\em Ergod. Th. $\&$ Dynam. Sys.}  37: 1 (2017) 79--102.

\bibitem{BTV}
T. Bomfim, J. Torres and P. Varandas,
\newblock Topological features of flows with the reparametrized gluing orbit property, 
\newblock  \emph{J. Diff. Eq.}, 262:8 (2017) 4292--4313. 

\bibitem{BTV2}
T. Bomfim, J. Torres and P. Varandas,
\newblock The gluing orbit property and partial hyperbolicity,
\newblock  \emph{J. Diff. Eq.} 272 (2021) 203--221. 

\bibitem{BV}
 T. Bomfim and P. Varandas, 
 The gluing orbit property, uniform hyperbolicity and large deviation principles for semiflows, 
 \newblock  \emph{J. Diff. Eq.},  267 (2019) 228--266.


 \bibitem{Bo71} R. Bowen, {Periodic points and measures for Axiom A diffeomorphisms,}  
 \emph{Trans. Amer. Math. Soc.} 154 (1971), 377-397.

\bibitem{BN}
L. Bowen and A. Nevo,
Pointwise ergodic theorems beyond amenable groups.
\emph{Ergod. Th. Dynam. Sys.}33:3 (2013) 777--820.


 \bibitem{Bufetov}
A. Bufetov. 
Topological entropy of free semigroup actions and skew product transformations.
\emph{J. Dynam. Control Syst.} 5:1 (1999) 137--143.

\bibitem{Buf}
A. Bufetov,
Convergence of spherical averages for actions of free groups,
\emph{Annals of Math.} 155:3 (2002) 929--944.

\bibitem{BKK}
A. Bufetov, M. Khristoforov and A. Klimenko,
Ces\`aro convergence of spherical averages for measure-preserving actions of Markov semigroups and groups. 
\emph{Int. Math. Res. Not. IMRN} 21 (2012) 4797--4829.


\bibitem{CRV}
M.~Carvalho, F.~Rodrigues, P.~Varandas.
\newblock Semigroups actions of expanding maps.
\newblock \emph{J. Stat. Phys.} 116:1 (2017) 114--136.

\bibitem{CRV2}
M.~Carvalho, F.~Rodrigues, P.~Varandas.
\newblock A variational principle for free semigroup actions, 
\emph{Adv. Math.}, 334 (2018) 450--487.

\bibitem{CV19}
M. Carvalho and P. Varandas, Genericity of historic behavior for maps and flows, Preprint 2020

\bibitem{CKS}
E. Chen, T.  K\"upper and L. Shu.
Topological entropy for divergence points. 
\emph{Ergod. Theory Dynam. Syst.} 25 (2005) 1173--1208. 

\bibitem{DER}
L. J. D\'iaz, S. Esteves and J. Rocha.
Skew product cycles with rich dynamics: From totally non-hyperbolic dynamics to fully prevalent hyperbolicity,
\emph{Dynam. Sys}. 31:1 (2016) 1--40.

\bibitem{DGR}
L. J. D\'iaz, K. Gelfert and M. Rams.
Entropy spectrum of Lyapunov exponents for nonhyperbolic step skew-products and elliptic cocycles. 
\emph{Commun. Math. Phys.} 367 (2019), 351--416. 


\bibitem{DOT}
 Y. Dong,  P. Oprocha,  X. Tian.
 {On the irregular points for systems with the shadowing property,} \emph{Ergod. Th. Dynam. Sys.},
 38:6 (2018) 2108--2131.


\bibitem{DK} P.~Duarte and S.~Klein,
\newblock \emph{Lyapunov Exponents of Linear Cocycles - Continuity via Large Deviations}.
\newblock Atlantis Studies in Dynamical Systems, Series volume 3, Atlantis Press, 2016, 263 pages. 

\bibitem{EKW}
A. Eizenberg, Y. Kifer, and B. Weiss. 
Large deviations for $\mathbb Z^d$ -actions. 
\emph{Comm. Math. Phys.} 164:3 (1994) 433--454.

. 

\bibitem{Fu} 
A. Furman. 
\newblock On the multiplicative ergodic theorem for uniquely ergodic systems. 
Ann. Inst. Henri Poincar\'e Probab. Stat. 33 (6) (1997) 797--815.

\bibitem{FK} 
H. Furstenberg and H. Kesten. Products of random matrices. 
\emph{Ann. Math. Statist.}, 31 (1960) 457--469.

\bibitem{Furstenberg} H.~Furstenberg,
\newblock Noncommuting random products. 
\newblock \textit{Trans.\ Amer.\ Math.\ Soc.\ }108 (1963), 377--428. 

\bibitem{GLW}
E. Ghys, R. Langevin and P. Walczak.
\newblock {Entropie geometrique des feuilletages.}
\newblock \emph{Acta Math.} 16 (1988) 105--142.


\bibitem{GM} I.~Ya.~Gol'dsheid and G.A.~Margulis,
\newblock Lyapunov exponents of a product of random matrices, 
\newblock \textit{Russian Math.\ Surveys} 44 (1989), no.\ 5, 11--71.

\bibitem{Grig}
R. Grigorchuk,
Ergodic theorems for actions of free groups and free semigroups,
\emph{Mathematical Notes}, 65:5 (1999) 654--657.

\bibitem{GR} Y.~Guivarc'h and A.~Raugi,
\newblock Propri\'et\'es de contraction d'un semi-groupe de matrices inversibles. Coefficients de Liapunoff d'un produit de matrices al\'eatoires ind\'ependantes.
\newblock \textit{Israel J.\ Math.\ }65 (1989), no.\ 2, 165--196. 

\bibitem{Herman}
M. Herman. 
Construction d'un diff\'eomorphisme minimal d'entropie topologique non nulle, 
\emph{Ergod. Th. Dynam. Sys.},  1:1 (1981) 65--76.

\bibitem{HS}
K. Hofmann and L. Stoyanov,
Topological entropy of group and semigroup actions,
\emph{Adv. Math.} 115 (1995) 54--98.

\bibitem{HN}
A. J. Homburg and M. Nassiri
Robust minimality of iterated function systems with two generators.
\emph{Ergod. Th. Dynam. Sys.} 34:6 (2014) 1914--1929.

\bibitem{JMW}
Y. Ju, D. Ma and Y. Wang.
Topological entropy of free semigroup actions for noncompact sets. 
\emph{Discrete Cont. Dynam. Sys.}  39:2 (2019) 995--1017.


\bibitem{Kalinin} B. Kalinin, Liv\v{s}ic theorem for matrix cocycles, 
\emph{Annals of Math.}, 2011, 173 (2), 1025--1042.

\bibitem{Katok}
A. Katok, Lyapunov exponents, entropy and periodic orbits for diffeomorphisms, 
\emph{Publ. Math. IHES} 51 (1980) 137--173.

\bibitem{KH} A.~Katok and B.~Hasselblatt,
\newblock {\em Introduction to the Modern Theory of Dynamical Systems},
\newblock Cambridge University Press, 1995.

\bibitem{Ki86} 
Y. Kifer. \emph{Ergodic Theory of Random Transformations.} Birkha\"user, Boston, 1986


\bibitem{KS}
S. Kolyada and L. Snoha.
\newblock {Topological entropy of nonautonomous dynamical systems}.
\newblock \emph{Random Comput. Dynamics} 4:2\&3 (1996) 205--233.

\bibitem{KU}
K. Kuratowski and S. Ulam,
Quelques propri\'et\'es topologiques du produit combinatoire,
\emph{Fund. Math.} 19:1 (1932) 247--251. 


\bibitem{LiTang}
Z. Li and D. Tang, 
Entropies of random transformations on a non-compact space. 
\emph{Results Math}, 74:120 (2019) https://doi.org/10.1007/s00025-019-1046-3

\bibitem{LV}
 H. Lima and P. Varandas, 
 On the rotation sets of generic homeomorphisms on the torus $\mathbb T^d$, 
\emph{ Ergod. Th. Dynam. Sys.}  (to appear) 

\bibitem{Lindenstrauss}
E. Lindenstrauss,
Pointwise theorems for amenable groups,
\emph{Eletronic Research Announcements,}
5 (1999) 82--90  


\bibitem{Li01}
P.-D. Liu.
\newblock Dynamics of random transformations smooth ergodic-theory.
\newblock \emph{Ergod. th. \& Dynam. Sys.}, 21: 1279--1319, 2001.

\bibitem{Livsic}
A. N. Liv\v{s}ic. 
Cohomology of dynamical systems. 
\emph{Math. USSR Izvestija} 6 (1972), 1278--1301.


\bibitem{LW}
F.~Ledrappier and P.~Walters.
\newblock A relativised variational principle for continuous transformations.
\newblock \emph{J. London Math. Soc.} 16, 3, (1977), 568--576.

\bibitem{Nakano}
Y. Nakano. 
Historic behaviour for random expanding maps on the circle.
\emph{Tokyo J. Math.}  40:1 (2017) 165--184.


\bibitem{Nevo}
A. Nevo,
\emph{Pointwise ergodic theorems for actions of groups},
Chapter 13, Handbook of dynamical systems, vol. 1B, Edited by B. Hasselblat and A. Katok, 1999.

\bibitem{NevoStein}
A. Nevo and E. Stein, A generalization of Birkhoff's pointwise ergodic theorem, 
\emph{Acta Math.} 173:1 (1994), 135--154.


\bibitem{OliTian} K. Oliveira and X. Tian, {Non-uniform hyperbolicity and non-uniform specification}, 
\emph{Trans. Amer. Math. Soc.},   365 (8),  2013,  4371--4392.

\bibitem{OW}
D. Ornstein and B. Weiss, 
Entropy and isomorphism theorems for actions of amenable groups, 
\emph{J. Anal. Math.} 48 (1987), 1--141.

\bibitem{O} V. Oseledets, A multiplicative ergodic theorem: Lyapunov characteristic numbers for dynamical systems, \emph{Trans. Moscow Math. Soc.}, 19 (1968) 197-231.

\bibitem{Pe97}
Ya. Pesin.
\newblock {\em Dimension theory in dynamical systems}.
\newblock University of Chicago Press, 1997.
\newblock Contemporary views and applications.

\bibitem{Pi}
M. Pinsker,
Dynamical systems with completely positive and zero entropy, 
\emph{Dokl. Akad. Nauk SSSR} 133 (1960) 1025--1026.


\bibitem{PS} C. Pfister, W.  Sullivan,  {On the topological entropy of saturated sets,}  \emph{Ergod. Th. Dynam. Sys.} 
27 (2007) 929--956.

\bibitem{RTZ}
X. Ren, X.  Tian and Y. Zhou.
On the topological entropy of saturated sets for amenable group actions.
Preprint 2008.05843v1

\bibitem{JMP}
F. Rodrigues and P. Varandas, Specification properties for group actions and thermodynamics of expanding semigroups, 
\emph{J. Math. Phys.} 57, (2016) 052704

\bibitem{Ro}
 V. A. Rohlin, Lectures on the entropy theory of measure-preserving transformations, 
 \emph{Russian Math. Surveys}, 22:5 (1967) 1--52.
 
  \bibitem{SSV} M. Stadlbauer, S. Suzuki and P. Varandas, {Thermodynamical formalism for random non-uniformly expanding maps},
 \emph{Comm. Math. Phys.}, 385:1 (2021) 369--427.


\bibitem{SVY} N.~Sumi, P. Varandas and K.~Yamamoto, {Partial hyperbolicity and specification},
\emph{Proc. Amer. Math. Soc.}, 144:3 (2016) 1161--1170.

\bibitem{SVY2} 
N. Sumi, P. Varandas and K. Yamamoto, Specification and partial hyperbolicity for flows, 
\emph{Dynam. Sys.} 30:4, (2015) 501--524. 

\bibitem{Sun0}
P. Sun, {Minimality and the gluing orbit property}, 
\emph{Discrete Cont. Dynam. Sys.}
39:7 (2019) 4041--4056.

\bibitem{Sun}
P. Sun, {Zero-entropy dynamical systems with gluing orbit property}, 
\emph{Adv. Math.} 372:7 (2020) 107294.

\bibitem{Tao}
T. Tao,
Failure of the $L^1$ pointwise and maximal ergodic theorems for the free group,
\emph{ Forum of Math., Sigma,} 3:27 (2015) doi:10.1017/fms.2015.28


\bibitem{TV03}
F. Takens and E. Verbitski.
\newblock On the variational principle for the topological entropy of certain non-compact sets.
\newblock {\em Ergod. Th. Dynam. Sys.},  23:1 (2003) 317--348.

\bibitem{Templeman}
A. Templeman,
\emph{Ergodic theorems for group actions: Informational and thermodynamical aspects}, Mathematics and its applications, Springer Science+Business Media, B.V. 1992.



\bibitem{Tian}
X. Tian, Nonexistence of Lyapunov exponents for matrix cocycles. 
\emph{Ann. Inst. Henri Poincar\'e - Probab. Stat.} 53 (2017), no. 1, 493--502.

\bibitem{Tian2}
X. Tian, Lyapunov non-typicalpoints of matrix cocycles and topological entropy. Preprint 2015. 

\bibitem{TV}
X. Tian and P. Varandas, 
Topological entropy of level sets of empirical measures for non-uniformly expanding maps, 
\emph{Discrete Cont. Dynam. Sys.} 37:10 (2017) 5407--5431. 

\bibitem{Tho2010}
D. Thompson, The irregular set for maps with the specification property has full topological pressure
\emph{Dyn. Sys.} 25:1 (2010) 25--51.


\bibitem{Tho2012} D. Thompson,  {Irregular sets, the $\beta$-transformation and the almost specification property}, \emph{Trans. Amer. Math. Soc.} 364:10 (2012) 5395--5414.


\bibitem{Paulo} P. Varandas, { Non-uniform specification and large deviations for weak Gibbs measures}. 
\emph{J. Stat. Phys}, 2012, 146 (2): 330-358.

\bibitem{Va20}
P. Varandas, On the cohomological class of fiber-bunched cocycles on semi simple Lie groups. \emph{Rev. Mat. Iberoam.} 36:4 (2020) Published online DOI:10.4171/rmi/1160



\bibitem{VianaB} M.~Viana,
\newblock Lectures on Lyapunov exponents,
\newblock Cambridge University Press, 2014. 


\bibitem{Yoc}
J.-C. Yoccoz, Some questions and remarks about $SL(2, \mathbb R)$-cocycles. In: Brin, M., Hasselblatt, B., Pesin,
Y. (eds.) \emph{Modern Dynamical Systems and Applications}, pp. 447--458. Cambridge University Press, Cambridge, 2004.

\bibitem{ZC13}
X. Zhou and E. Chen.
\newblock Multifractal analysis for the historic set in topological dynamical systems.
\newblock {\em Nonlinearity}, 26, no. 7, 1975--1997, 2013.


\bibitem{ZhuMa2021} L. Zhu and D. Ma,   {The upper capacity topological entropy of free semigroup actions for certain 
non-compact sets}, \emph{ J. Stat. Phys.}  182:19 (2021).



\bibitem{Sad1}
 V. Sadovskaya,
Cohomology of $GL(2,\mathbb R)$-valued cocycles over hyperbolic systems. 
\emph{ Discrete Contin. Dyn. Sys.} 33 (2013), no. 5, 2085--2104. 


\bibitem{Shub} M. Shub, \textit{Global stability of dynamical systems}, Springer Verlag, (1987).





\end{thebibliography}
\end{document}